\documentclass[11pt]{article}
\usepackage{amsthm,amssymb,amsmath,enumerate,graphicx,psfrag, comment}
\usepackage{algorithm}
\usepackage{algorithmic}
\usepackage{url}
\usepackage{authblk}
\usepackage{caption}
\usepackage{tabu}
\usepackage{thmtools}
\usepackage{thm-restate}
\usepackage[margin=1in]{geometry}

\declaretheorem{theorem}
\declaretheorem[sibling=theorem]{lemma}
\declaretheorem{claim}
\usepackage{hyperref}

\usepackage{color}

\usepackage{hyperref}

\usepackage{cleveref}

\title{Obstructions for three-coloring and list three-coloring $H$-free graphs}

\author[1]{Maria Chudnovsky\footnote{Partially supported
by NSF grant DMS-1550991 and U.S. Army Research Office grant W911NF-16-1-0404.}}
\author[2]{Jan Goedgebeur\footnote{Supported by a Postdoctoral Fellowship of the Research Foundation Flanders (FWO).}}
\author[3]{Oliver Schaudt}
\author[4]{Mingxian Zhong}

\affil[1]{\emph{\small Princeton University, Princeton, NJ 08544,
USA. E-mail:~mchudnov@math.princeton.edu}}
\affil[2]{\emph{\small Ghent University, Ghent, Belgium. E-mail: jan.goedgebeur@ugent.be}}
\affil[3]{\emph{\small RWTH Aachen University, Aachen, Germany. E-mail: schaudt@mathc.rwth-aachen.de}}
\affil[4]{\emph{\small Columbia University, New York, NY 10027,
USA. E-mail: mz2325@columbia.edu}}

\date{}

\begin{document}

\maketitle

\begin{center}
\emph{In loving memory of Ella.}
\end{center}~\\

\begin{abstract}
A graph is $H$-free if it has no induced subgraph isomorphic to $H$.
We characterize all graphs $H$ for which there are only finitely many minimal
non-three-colorable $H$-free graphs. Such a characterization was previously 
known only in the case when $H$ is connected.
This solves a problem posed by Golovach \emph{et al.}
As a second result, we characterize all graphs $H$ for which there are only 
finitely many $H$-free minimal obstructions for list 3-color\-ability.\\
%
%The second result implies that when parameterized by the number of vertices that need to be removed to destroy all induced copies of a fixed connected graph $H$, the 3-colorability problem admits a polynomial kernel if and only if $H$ is a path on at most six vertices.
% Combined with previous results this answers a question by Jansen and Kratsch.

\noindent\textbf{Keywords:} graph coloring, critical graph, induced subgraph.
\end{abstract}

%{\color{red}
%
%\bigskip
%
%TODO's for journal version:
%
%\begin{itemize}
%\item Include appendix with alternative computer-free proof?
%\item ...
%\end{itemize}
%
%}

%\thispagestyle{empty}
%
%\newpage
%
%\clearpage
%\setcounter{page}{1}

\section{Introduction}

A \emph{$k$-coloring} of a graph $G=(V,E)$ is a mapping $c : V \to \{1,\ldots,k\}$ such that $c(u) \neq c(v)$ for all edges $uv \in E$.
If a $k$-coloring exists, we say that $G$ is \emph{$k$-colorable}.
The related decision problem -- does a given input graph admit a $k$-coloring? -- is called the \emph{$k$-colorability problem}; it is one of the most famous NP-complete problems.
Let $L$ be a mapping that maps each vertex of $G$ to a subset of $\{1,\ldots,k\}$.
%We call $L$ a \emph{list system} for $G$, and each set $L(v)$ a \emph{list}.
We say that the pair $(G,L)$ is \emph{colorable} if there is a $k$-coloring $c$ of $G$ with $c(v) \in L(v)$ for each $v \in V(G)$.
The \emph{list $k$-colorability problem} is the following: given a pair $(G,L)$ with $L(v) \subseteq \{1,\ldots,k\}$ for each $v \in V(G)$, decide whether $(G,L)$ is colorable.
%Note that the list $k$-colorability problem generalizes the $k$-colorability problem.
Note that the list $k$-colorability problem generalizes both the $k$-colorability problem and the precoloring extension problem.
In the $k$-colorability problem we have $|L(v)|=k$ for all $v \in V(G)$, while in the precoloring extension problem we have $|L(v)|\in \{1,k\}$ for all $v \in V(G)$.
In this paper we study the minimal obstructions for $k$-colorability and list $k$-colorability: minimal subgraphs that prevent a graph from being $k$-colorable or list $k$-colorable.

Let $H$ and $ G$ be graphs. We say that $H$ is an {\em induced subgraph} of $G$
if $V(H) \subseteq V(G)$, and $u,v \in V(H)$ are adjacent in $H$ if and
only if $u,v$ are adjacent in $G$. For   $X \subseteq V(G)$, we denote by
$G|X$ the induced subgraph of $G$ with vertex set $X$, and we say that
$X$ {\em induces} $G|X$. If $G|X$ is isomorphic to $H$, we say that 
{\em $X$ is an $H$ in $G$}.
If $X \neq V(G)$, we say that $G|X$ is a {\em proper} induced subgraph of $G$.
We say that $G$ is \emph{$k$-chromatic} if it is $k$-colorable but not 
$(k-1)$-colorable.
A graph is called \emph{$(k+1)$-vertex-critical} if it is $(k+1)$-chromatic, but every induced proper subgraph is $k$-colorable.
For example, the class of $3$-vertex-critical graphs is the family of all odd cycles.
In view of the NP-hardness of the $k$-colorability problem for $k\geq 3$, there is little hope of giving a characterization of the $(k+1)$-vertex-critical graphs that is of use in algorithmic applications.
The picture changes if one restricts the structure of the graphs under consideration, and the aim of this paper is to describe this phenomenon.

We use the following notation. Given two graphs $G$ and $H$, we say that
$G$ {\em contains} $H$ if some induced subgraph of $G$ is isomorphic to $H$.
If $G$ does not contain $H$, we say that $G$ is {\em $H$-free}.
For a family $\mathcal{H}$ of graphs, $G$ is {\em $\mathcal{H}$-free}
if $G$ is $H$-free for every $H \in \mathcal{H}$.
Moreover, we write $G+H$ for the graph that is the disjoint union of $G$ and 
$H$.
For all $t$, let $P_t$ denote the {\em path} on $t$ vertices, which is the graph with vertex set $\{p_1, \ldots, p_t\}$ such that $p_i$ is adjacent to $p_j$ if 
and only if $|i-j|=1$.

In an earlier paper, we proved the following theorem, solving a problem posed by Golovach \emph{et al.}~\cite{GJPS15} and answering a question of Seymour~\cite{SeyPriv}.

\begin{theorem}[Chudnovsky \emph{et al.}~\cite{CGSZ15}]\label{thm:old}
Let $H$ be a connected graph.
There are only finitely many 4-vertex-critical $H$-free graphs if and only if $H$ is an induced subgraph of $P_6$.
\end{theorem}

In view of our result, Golovach \emph{et al.}~\cite{GJPS15} posed the problem of extending the above theorem to a complete dichotomy for arbitrary graphs $H$.
While this seems to be an incremental question at first sight, it requires 
entirely different machinery to be settled.

A second natural generalization of Theorem~\ref{thm:old} is to go from 4-vertex critical graphs to minimal obstructions to list 3-colorability.
This more technical generalization is motivated, among other things, by
a theorem of Jansen and Kratsch~\cite{JK13}, that 
says that if there is a finite list of $H$-free minimal obstructions  to list-3-colorability, then a polynomial kernelization of the 3-coloring problem exists when parameterized by the number of vertices needed to hit all induced copies of $H$. This application is outside of the scope of this paper, and so we skip the precise definitions.

\subsection{Our contribution}

We answer both questions mentioned above.
We obtain the following dichotomy theorem which now fully settles the problem of characterizing all graphs $H$ for which there are only finitely many 4-vertex-critical $H$-free graphs.

\begin{restatable}{theorem}{secondthm}
	\label{thm:coloring-characterization}%
Let $H$ be a graph.
There are only finitely many $H$-free 4-vertex-critical graphs if and only if $H$ is an induced subgraph of $P_6$, $2P_3$, or $P_4 + kP_1$ for some $k \in \mathbb N$.
\end{restatable}

%We remark that the proof of Theorem~\ref{thm:coloring-characterization} uses a different method than that of Theorem~\ref{thm:old}. T
The tools  used in~\cite{CGSZ15} to prove Theorem~\ref{thm:old} were tailored specifically for the $P_6$-free case and do not generalize well, while our new approach is significantly more powerful.
% in the sense that we can deal with both, connected and disconnected graphs $H$.
%The idea is to transfer the problem to the more general list setting defined below, and solve it there.
The idea is to transfer the problem to the more general list setting and solve it there, showing that there is a constant $C$ such that minimal obstructions 
have bounded size at most $C$.
This generality comes at a certain cost: the upper bound we get is far from 
sharp.

Our second main result is the analogue of Theorem~\ref{thm:coloring-characterization} in the list setting.
To state it, we need the following concept.
%Let $G$ be a graph and let $L$ be a mapping that maps each vertex of $G$ to a subset of $\{1,\ldots,k\}$.
%We call $L$ a \emph{list system} for $G$, and each set $L(v)$ a \emph{list}.
%We say that the pair $(G,L)$ is \emph{colorable} if there is a $k$-coloring $c$ of $G$ with $c(v) \in L(v)$ for each $v \in V(G)$.
%The \emph{list $k$-colorability problem} is the following: given a pair $(G,L)$ with $L(v) \subseteq \{1,\ldots,k\}$ for each $v \in V(G)$, decide whether $(G,L)$ is colorable.
%Note that the list $k$-colorability problem generalizes both the $k$-colorability problem and the precoloring extension problem.
%In the $k$-colorability problem we have $|L(v)|=k$ for all $v \in V(G)$, while in the precoloring extension problem we have $|L(v)|\in \{1,k\}$ for all $v \in V(G)$.

We call a pair $(G,L)$ with $L(v) \subseteq \{1,2,3\}$, $v \in V(G)$, a \emph{minimal list-obstruction} if $(G,L)$ is not colorable but for all proper induced 
subgraphs $A$ of $G$ the pair $\left(A,L\right)$ is colorable.
Here and throughout the article, if $H$ is an induced subgraph of $G$, then by
$(H,L)$ we mean $H$ with the list system $L$ restricted to $V(H)$.
We prove the following theorem.

\begin{restatable}{theorem}{thirdthm}\label{thm:list-characterization}
Let $H$ be a graph. There are only finitely many $H$-free minimal list-obstructions if and only if $H$ is an induced subgraph of $P_6$, or of $P_4 + kP_1$ for some $k \in \mathbb N$.
\end{restatable}

Note that there are infinitely many $2P_3$-free minimal list-obstructions while there are only finitely many 4-vertex-critical $2P_3$-free graphs.
Thus, Theorem~\ref{thm:coloring-characterization} is not a special case of Theorem~\ref{thm:list-characterization}.
Moreover, the fact that there are only finitely many $P_6$-free minimal list-obstructions does not follow from Theorem~\ref{thm:old}.

\subsection{Previous work}

%Let us now mention a few results in this line of research.
It is known that the $k$-colorability problem is NP-hard for $k \geq 3$ on $H$-free graphs unless if $H$ is a disjoint union of paths~\cite{Hol-color,K-L-col-g,K-K-T-W-col-g,L-G-color}.
This motivates the study of graph classes in which a path is forbidden as an induced subgraph in the context of the complexity of the $k$-colorability problem.
Regarding the computational complexity of the 3-colorability problem, the state of the art is the polynomial time algorithm to decide whether a $P_7$-free graph admits a 3-coloring~\cite{BCMSSZ15}. 
The algorithm actually solves the harder list 3-colorability problem.
%defined in the discussion above Theorem~\ref{thm:list-characterization}. 
For any $k\geq 8$, it is not known whether deciding $3$-colorability for a $P_k$-free graph is polynomial time solvable or not.

There are quite a few results regarding the number of critical $H$-free graphs.
To describe these results, consider the following definition.
If $H$ is a graph, a \emph{$(k+1)$-critical $H$-free graph} is a graph $G$  that is $H$-free, $(k+1)$-chromatic, and every $H$-free proper (not necessarily induced) subgraph of $G$ is $k$-colorable.
Note that there are finitely many $(k+1)$-critical $H$-free graphs if and only if there are finitely many $(k+1)$-vertex-critical $H$-free 
graphs \cite{Hoang2}. 
These critical graphs are of special interest since they form a canonical \textit{no-certificate} for $k$-colorability. Given a decision problem, a solution algorithm is called \emph{certifying} if it provides, together with the yes/no decision, a polynomial time verifiable certificate for this decision. (A canonical \textit{yes-certificate} would be a $k$-coloring of the graph; since most existing graph coloring algorithms are constructive they can hence easily provide a yes-certificate. A canonical no-certificate would be a $(k+1)$-critical subgraph of bounded size).
%(see previous section).
% \jg{Feel free to drop the previous sentence if you think it's too much repetition.}
For all $t$, we let $C_t$ denote the  {\em cycle} on $t$ vertices,
that is a graph with vertex set $c_1, \ldots, c_t$ and such that 
$c_i$ is adjacent to $c_j$ if and only if $|i-j| \in \{1,t-1\}$.

Bruce \emph{et al.}~\cite{BHS09} proved that there are exactly six 4-critical $P_5$-free graphs.
Later, Maffray and Morel~\cite{MM12}, by characterizing the 4-vertex-critical $P_5$-free graphs, designed a linear time algorithm to decide 3-colorability of $P_5$-free graphs.
Randerath et al.~\cite{randerath_04} have shown that the only 4-critical $(P_6,C_3)$-free graph is the Gr\"otzsch graph.
More recently, Hell and Huang~\cite{hell_14,Hell-survey-coloring} proved that there are exactly four 4-critical $(P_6,C_4)$-free graphs.
They also proved that there are only finitely many $k$-critical $(P_6,C_4)$-free graphs, for all $k$.
As mentioned earlier, we proved Theorem~\ref{thm:old} which says that there are only finitely many 4-vertex-critical $P_6$-free graphs, namely 80.

Recently~\cite{GS15}, two of the authors of this paper developed an enumeration algorithm to automate the case analysis performed in the proofs of the results mentioned above.
Using this algorithm, it was shown that there are only finitely many $4$-critical $(P_7,C_k)$-free graphs, for both $k=4$ and $k=5$.
Since there is an infinite family of $(P_7,C_6,C_7)$-free graphs, only the case of $(P_7,C_3)$-free graphs remains open. 
It was also shown that there are only finitely many $4$-critical $(P_8,C_4)$-free graphs. 
% \jg{Should we also mention that in~\cite{GS15} we also determined all 4-critical $(P_3+P_2)$-free graphs and 4-critical $(P_4+2P_1)$-free graphs?}
For more details on this line of research we recommend the two excellent survey papers by Hell and Huang~\cite{Hell-survey-coloring} and Golovach \emph{et al.}~\cite{GJPS15}.

\begin{comment}

Using the algorithm from~\cite{GS15} and adapting it to the list case, we were able to determine the exact number of $P_6$-free minimal list-obstructions with at most 9 vertices. 
It turns out that there are many such obstructions compared to the fact that there are only 80 $P_6$-free 4-vertex-critical graphs~\cite{CGSZ15}.
The counts can be found in Table~\ref{table:counts_animals-P6free} and our implementation of this algorithm can be downloaded from~\cite{listcriticalpfree-site}. 
All obstructions up to 9 vertices can also be downloaded from the \emph{House of Graphs}~\cite{hog} at \url{http://hog.grinvin.org/Critical}. 
% Uncomment for journal version

\begin{table}[ht!]
\centering
{\footnotesize
\begin{tabu}{ l ccccccccc }
\tabucline[0.9pt]{-}\\[-7pt]
Vertices 							&1	& 2 &  3 & 4 & 5 & 6 & 7 & 8 & 9\\[1pt]
Obstructions	& 1& 1 &  4 & 43 & 117 & 1 806 & 34 721 & 196 231 & 1 208 483 \\[1pt]
\tabucline[0.9pt]{-} 
\end{tabu}
}
\caption{Counts of all $P_6$-free minimal list-obstructions with at most 9 vertices, up to swapping the lists and labels of the vertices. In total, there are 1 441 407 such list-obstructions.}

\label{table:counts_animals-P6free}
\end{table}

\end{comment}

\subsection{Structure of the paper}

In Section~\ref{sec:prelims} we state the relevant definitions and the notation used in later sections.

In Section~\ref{sec:smalllists} 
we develop the  concept of the so-called \emph{propagation path}, which is 
the main tool in showing  that there are only finitely many $H$-free minimal list-obstructions (for the right choices of $H$) with lists of size at most 2.
In particular, we show that for every minimal list-obstruction with lists of
size at most 2, we can delete at most four vertices so that what remains  is the union of four propagation paths.

In Section~\ref{sec:sufficiency-P6} we prove that there are only finitely many $P_6$-free minimal list-obstructions.
We split the proof into two parts. In the first part we show that there are only finitely many $P_6$-free minimal list-obstructions where every list is of size at most 2, which amounts to studying $P_6$-free propagation paths. This step
has a computer-aided proof. We also have a computer-free proof of this fact, 
but it is tedious, and we decided to only include a sketch of it. 
In the second part of the proof we reduce the general problem to the case 
solved in the first part.
Here we rely on a structural analysis, making use of a structure theorem for 
$P_t$-free graphs.

Using a similar approach, in Section~\ref{sec:sufficiency-2P3} we prove that there are only finitely many $2P_3$-free 4-vertex-critical graphs (of course, certain modifications are needed, because the list version is false in this case).
In Section~\ref{sec:sufficiency-P4+kP1} we show that there are only finitely many $P_4+kP_1$-free minimal list-obstructions.

In Section~\ref{sec:necessity} we prove the necessity in the statement of Theorem~\ref{thm:coloring-characterization} and Theorem~\ref{thm:list-characterization}, providing infinite families of $H$-free 4-vertex-critical graphs and minimal list-obstructions.

Our main results are proven in Section~\ref{sec:together} where we put together the  results mentioned above.
%In Section~\ref{sec:kernelization} we prove Theorem~\ref{thm:polykernel} building on the results from earlier sections and a method due to Jansen and Kratsch~\cite{JK13}.
%We close the paper by stating two problems for further research in Section~\ref{sec:further-research}.

\section{Preliminaries}\label{sec:prelims}

All graphs in this paper are finite and simple.
Let $G$ and $H$ be graphs and let $X$ be a subset of $V(G)$. 
We denote 
by $G\setminus X$ the graph $G|(V(G) \setminus X)$.
If $X=\{v\}$ for some $v \in V(G)$, we write $G \setminus v$ instead of $G \setminus \{v\}$.
If $G|X$ is isomorphic to $H$, then we say that \emph{$X$ is an $H$ in $G$}.
We write $G_1+\ldots+G_k$ for the disjoint union of graphs $G_1,\ldots,G_k$.
The \emph{neighborhood} of a vertex $v\in V(G)$ is denoted by $N_G(v)$
(when there is no danger of confusion, we sometimes write $N(v)$).
For a vertex set $S$, we use $N(S)$ to denote $(\bigcup\limits_{v\in S}N(v)) \setminus S$.

For $n\geq 0$, we denote by \emph{$P_n$ the chordless path on $n$ vertices}. 
For $n\geq 3$, we denote by \emph{$C_n$ the chordless cycle on $n$ vertices}.
By convention, when explicitly describing a path or a cycle, we always list the vertices in order. 
Let $G$ be a graph. 
When $G|\{p_1,\ldots,p_n\}$ is the path $P_n$, we say that \emph{$p_1$-$\ldots$-$p_n$ is a $P_n$ in $G$}. 
Similarly, when $G|\{c_1,c_2,\ldots,c_n\}$ is the cycle $C_n$, we say that \emph{$c_1$-$c_2$-$\ldots$-$c_n$-$c_1$ is a $C_n$ in $G$}. 
% A \emph{clique} in $G$ is a set of vertices which are all pairwise adjacent, and a \emph{stable set} is a set of vertices which are all pairwise non-adjacent. 
A \emph{Hamiltonian path} is a path that contains all vertices of $G$.

% A \emph{partition} of a set $S$ is a collection of disjoint subsets of $S$ whose union is $S$. 
Let $A$ and $B$ be disjoint subsets of $V(G)$. 
For a vertex $b\in V(G)\setminus A$, we say that \emph{$b$ is complete to $A$} if $b$ is adjacent to every vertex of $A$, and that \emph{$b$ is anticomplete to $A$} if $b$ is non-adjacent to every vertex of $A$. 
If every vertex of $A$ is complete to $B$, we say \emph{$A$ is complete to $B$}, and if every vertex of $A$ is anticomplete to $B$, we say that \emph{$A$ is anticomplete to $B$}. 
If $b\in V(G)\setminus A$ is neither complete nor anticomplete to $A$, we say that \emph{$b$ is mixed on $A$}. 
% We say $G$ is \emph{connected} if $V(G)$ cannot be partitioned into two disjoint non-empty sets anticomplete to each other. 
The \emph{complement} $\overline{G}$ of $G$ is the graph with vertex set $V(G)$ such that two vertices are adjacent in $\overline{G}$ if and only if they are non-adjacent in $G$. 
If $\overline{G}$ is connected we say that $G$ is \emph{anticonnected}. 
For $X \subseteq V(G)$, we say that $X$ is {\em connected} if $G|X$ is connected, and that $X$ is anticonnected if $G|X$ is anticonnected. 
A \emph{component} of $X \subseteq V(G)$ is a maximal connected subset of $X$, and an \emph{anticomponent} of $X$ is a maximal anticonnected subset of $X$. 
We write {\em component of $G$} to mean a component of $V(G)$.
A subset $D$ of $V(G)$ is called a \emph{dominating set} if every vertex in $V(G)\setminus D$ is adjacent to at least one vertex in $D$; in this case we also say that $G|D$ is a \emph{dominating subgraph} of $G$. 

% A \emph{$k$-coloring} of a graph $G$ is a mapping $c:V(G)\rightarrow \{1,\ldots,k\}$ such that if $x,y\in V(G)$ are adjacent, then $c(x)\ne c(y)$. 
% If a $k$-coloring exists for a graph $G$, we say that $G$ is \emph{$k$-colorable}. 
A \emph{list system $L$} of a graph $G$ is a mapping which assigns each vertex $v\in V(G)$ a finite subset of $\mathbb{N}$, denoted by $L(v)$. 
% Each $L(v)$, $v \in V(G)$, we call a \emph{list}.
A \emph{subsystem} of a list system $L$ of $G$ is a list system $L'$ of $G$ such that $L'(v)\subseteq L(v)$ for all $v\in V(G)$. 
We say a list system $L$ of the graph $G$ has \emph{order} $k$ if $L(v)\subseteq \{1,\ldots,k\}$ for all $v\in V(G)$. 
In this article, we will only consider list systems of order $3$.
Notationally, we write $(G,L)$ to represent a graph $G$ and a list system $L$ of $G$. 
We say that $c$, a coloring of $G$, is an \emph{$L$-coloring} of $G$, or a \emph{coloring of $(G,L)$} provided $c(v)\in L(v)$ for all $v\in V(G)$. 
We say that $(G,L)$ is \emph{colorable}, if there exists a coloring of $(G,L)$. 
A \emph{partial coloring} of $(G,L)$  is a mapping $c:U \to \mathbb N$ such that $c(u) \in L(u)$ for all $u \in U$, where $U$ is a subset of $V(G)$.
Note that here we allow for edges $uv$ of $G|U$ with $c(u)=c(v)$.
If there is no such edge, we call $c$ \emph{proper}.
	
Let $G$ be a graph and let $L$ be a list system of order $3$ for $G$.
We say that $(G,L)$ is a \emph{list-obstruction} if $(G,L)$ is not colorable.
As stated earlier, we call $(G,L)$ a \emph{minimal list-obstruction} if, in addition,  $(G\setminus x,L)$ is colorable for every vertex $x\in G$.

Let $(G,L)$ be a list-obstruction.
We say a vertex $v\in V(G)$ is \emph{critical} if $G\setminus v$ is $L$-colorable and \emph{non-critical} otherwise.
If we repeatedly delete non-critical vertices of $G$ to obtain a new 
graph, $G'$ say, such that 
$(G',L)$ is a minimal list obstruction, we say that
$(G',L)$ is a minimal list obstruction \emph{induced} by $(G,L)$.

Let $u,v$ be two vertices of a list-obstruction $(G,L)$.
We say that $u$ \emph{dominates} $v$ if $L(u) \subseteq L(v)$ and $N(v) \subseteq N(u)$.
It is easy to see that if there are such vertices $u$ and $v$ in $G$, then $(G,L)$ is not  a minimal list-obstruction.
We frequently use this observation without further reference.

%Let $G$ be a graph with list system $L$, and let $X\subseteq V(G)$. For every $\{i,j\}\subseteq \{1,2,3\}$ we write $X_{ij}=\{v\in X \text{ such that } L(v)=\{i,j\}\}$. We also write $X_{123}=\{v\in X \text{ such that } L(v)=\{1,2,3\}\}$.

\subsection{Updating lists}

Let $G$ be a graph and let $L$ be a list system for $G$.
Let $v,w \in V(G)$ be adjacent, and assume that $|L(w)|=1$.
To \emph{update  the list of $v$ from $w$} means to delete
from  $L(v)$ the unique element of $L(w)$.
If the size of the list of $v$ is reduced to one, we sometimes say that
$v$ is {\em colored}, and refer to the unique element in the list of $v$
as the {\em color} of $v$.
Throughout the paper, we make use of distinct updating procedures to reduce 
the sizes of the lists, and we define them below.

% Clearly, such an update does not change the colorability of the graph. 
If $P=v_1$-\ldots-$v_k$ is a path and $|L(v_1)|=1$, then to \emph{update from $v_1$ along $P$} means to update $v_2$ from $v_1$ if possible,  then to  update $v_3$ from $v_2$ if possible, and so on. When $v_k$ is updated from $v_{k-1}$, we stop the updating.

Let $X\subseteq V(G)$ such that $|L(x)| \leq 1$ for all $x\in X$. 
For a subset $A\subseteq V(G)\setminus X$, we say that we \emph{update the lists of the vertices in $A$ with respect to $X$} if we update each $a \in A$ from each $x \in X$. 
We say that we \emph{update the lists with respect to $X$} if $A=V(G)\setminus X$. 
Let $X_0=X$ and $L_0=L$. 
For $i \geq 1$ define  $X_i$ and $L_i$ as follows. $L_i$ is the list system 
obtained from $L_{i-1}$ by updating with respect to $X_{i-1}$. 
Moreover, $X_i=X_{i-1} \cup \{v \in V(G) \setminus X_{i-1} : |L_i(v)| \leq 1\text{ and } |L_{i-1}(v)|>1\}$.
We say that $L_i$ is obtained from $L$ by {\em updating with respect to  
	$X$ $i$ times}.
If $X=\{w\}$  we say that $L_i$ is obtained by {\em updating with respect to $w$
	$i$ times}. If for some $i$, $W_i=W_{i-1}$ and $L_i=L_{i-1}$, we say
that $L_i$ was  obtained from $L$  by {\em updating exhaustively with respect 
to $X$ (or $w$)}. For simplicity, if $X$ is an induced subgraph of $G$, by updating with respect to $X$ we mean updating with respect to $V(X)$.  
We also adopt the following convention. If for some $i$,
if two vertices of $X_{i-1}$ with the same list are adjacent, or
$L_i(v) = \emptyset$ for some $v$, we set
$L_i(v)=\emptyset$ for every $v \in V(G) \setminus X_i$.  Observe that in this 
case $(G|X_i, L_i)$ is not colorable, and so we have preserved at least one
minimal list obstruction induced by $(G,L)$.

\section{Obstructions with lists of size at most two}\label{sec:smalllists}

The aim of this section is to provide an upper bound on the order of the $H$-free minimal list-obstructions in which every list has at most two entries.
%Let us stress the fact that we restrict our attention to lists of order 3.
Let us stress the fact that we restrict our attention to lists which are (proper) subsets of $\{1,2,3\}$.
Before we state our lemma, we need to introduce the following technical definition.

Let $(G,L)$ be a minimal list-obstruction such that $|L(v)|\le 2$ for all $v \in V(G)$.
Let $P=v_1$-$v_2$-\ldots-$v_k$ be a path in $G$, not necessarily induced.
Assume that $|L(v_1)|\ge 1$ and $|L(v_i)|=2$ for all $i \in \{2,\ldots,k\}$.
Moreover, assume that there is a color $\alpha \in L(v_1)$ such that if we give color $\alpha$ to $v_1$ and update along $P$, we obtain a coloring $c$ of $P$. Please note that $c$ may not be a coloring of the graph $G|V(P)$.
For $i \in\{2,\ldots,k\}$ with $L(v_i)=\{\beta,\gamma\}$ and $c(v_i)=\beta$ we 
define the {\em shape} of $v_i$ to be $\beta \gamma$, and denote it by
$S(v_i)$.
If every edge $v_iv_j$ (of $G$) with $3 \le i < j \le k$ and $i \le j-2$ is such that 
\begin{equation}\label{eqn:conditions0}
\mbox{$S(v_i)=\alpha\beta$ and $S(v_j)=\beta\gamma$,} 
\end{equation}
where $\{1,2,3\} = \{\alpha,\beta,\gamma\}$, then we call $P$ a \emph{propagation path} of $G$ and say that \emph{$P$ starts with color $\alpha$}.
As we prove later,~\eqref{eqn:conditions0} implies that the updating process from $v_1$ along $P$ to $v_k$ cannot be shortcut via any edge $v_iv_j$ with $3 \le i < j \le k$ and $i \le j-2$.

The next lemma shows that, when bounding the order of our list-obstructions, we may concentrate on upper bounds on the size of propagation paths.

\begin{lemma}\label{lem:propagationpath}
Let $(G,L)$ be a minimal list-obstruction, where $|L(v)| \leq 2$ and $L(v) \in \{1,2,3\}$  for every $v \in V(G)$.
Assume that all propagation paths in $G$ have at most $\lambda$ vertices for some $\lambda \ge 20$.
Then $|V(G)| \leq 4\lambda+4$.
\end{lemma}

In the next section we prove the above lemma.
First we show  that if $G$ is a minimal list-obstruction in which every list contains at most two colors, then $V(G)=V_1 \cup V_2$, where $|V_1 \cap V_2| \geq 1$, and for some $v \in V_1 \cap V_2$, each $G|V_i$ has a Hamiltonian path $P_i$
starting at $v$. Moreover, if $L(v)=\{c_1,c_2\}$, then for every 
$i \in \{1,2\}$ giving $v$ the color $c_i$ and updating along $P_i$, 
results in a pair of adjacent vertices of $G$ receiving the same list of size 
one. Then  we prove that the edges of $G$ that are not the edges of
$P_1,P_2$ are significantly restricted, and consequently each $P_i$ is 
(almost) the union of two propagation paths, thus proving the lemma.

\subsection{Proof of Lemma~\ref{lem:propagationpath}}\label{sec:propagationpath}

Let $(G=(V,E),L)$ be a minimal list-obstruction such that $|L(v)|\le 2$ and
$L(v) \subseteq \{1,2,3\}$ for all $v \in V$.
If there is a vertex with an empty list, then this is the only vertex of $G$ and we are done.
So, we may assume that every vertex of $G$ has a non-empty list.
Let $V_1=\{v \in V : |L(v)|=1\}$ and $V_2=\{v \in V : |L(v)|=2\}$.

\begin{claim}\label{clm:improper}
Let $x \in V$ and $\alpha \in L(x)$ be arbitrary.
Assume we give color $\alpha$ to $x$ and update exhaustively in the graph $(G|(V_2\cup\{x\}),L)$. 
Let $c$ be partial coloring thus obtained.
For each $y \in V_1$ that did not receive a color so far, let $c(y)$ be the unique color in $L(y)$.
Then there is an edge $uv$ such that $c(u)=c(v)$.
\end{claim}
\begin{proof}
Let us give color $\alpha$ to $x$ and update exhaustively, but only considering vertices and edges in the graph $(G|(V_2\cup\{x\}),L)$.
We denote this partial coloring by $c$.
For each $y \in V_1$ that did not receive a color so far, let $c(y)$ be the unique color in $L(y)$.
For a contradiction, suppose that this partial coloring $c$ is proper.

Since $G$ is an obstruction, $c$ is not a coloring of $G$, meaning there are still vertices with two colors left on their list.
We denote the set of these vertices by $U$.
By minimality of $G$, we know that both graphs $(G',L'):=(G\setminus U,L)$ and $(G'',L''):=(G|(U\cup V_1) \setminus x,L)$ are colorable and have at least one vertex.

Let $c'$ be the coloring of $G'$ such that $c'(u)=c(u)$ for all $u \in V(G')$, and let $c''$ be a coloring of $G''$.
It is clear that $c'$ and $c''$ agree on the vertices in $V(G')\cap V(G'')=V_1\setminus \{x\}$.
Moreover, if $v \in (V_2 \setminus U) \cup \{x\}$ and $u \in U$ such that $uv \in E(G)$, then $c(v) \not \in L(u)$. 
Since $c'(v)=c(v)$ for every $v \in V(G')$, we deduce that $c'(v) \neq c''(u)$ for every $uv \in E(G)$ with
$u \in U$ and $v \in (V_2 \setminus U) \cup \{x\}$. 
Consequently, we found a coloring of $(G,L)$, a contradiction.
%Thus every vertex of $V_2$ receives a color in the partial coloring $c$, a contradiction.
\end{proof}

\begin{claim}\label{clm:V1small}
It holds that $|V_1|\le 2$.
\end{claim}
\begin{proof}
Suppose that $|V_1| \geq 3$, and let $x \in V_1$ and $\alpha \in L(x)$.
Let us give color $\alpha$ to $x$ and update exhaustively, but only considering vertices and edges in the graph $(G|(V_2\cup\{x\}),L)$.
We denote this partial coloring by $c$.

Since $G$ is minimal, there is no edge $uv$ with $u,v \in V_2 \cup \{x\}$ and $c(u) = c(v)$.
Since $(G|(V_1 \setminus \{x\}),L)$ is colorable by the
minimality of $G$, and since $|L(v)|=1$ for every 
$v \in V_1 \setminus \{x\}$, it follows that there is an edge $uv$ with
$u \in V_2 \cup \{x\}$ and $v \not \in V_2 \cup \{x\}$ such that 
$L(v)=\{c(u)\}$. It follows that  
$(G|(V_2 \cup \{x,v\}),L)$ is not colorable, and so by the
minimality of $G$, $V_1=\{x,v\}$, as required. 
This proves the first claim.
\end{proof}

Next we prove that, loosely speaking, $G$ is the union of at most two paths, starting at a common vertex, updating along each of which yields an
improper partial coloring.
Depending on the cardinality of $V_1$, we arrive at three different situations which are described by the following three claims.
%Recall from Claim~\ref{clm:V1small} that $|V_1|\le 2$.

\begin{claim}\label{clm:V1=0}
Assume that $|V_1|=0$, and pick $x \in V$ arbitrarily.
Let us say that $L(x)=\{1,2\}$.
For $\alpha=1,2$ there is a path $P^{\alpha}=v_1^{\alpha}$-\ldots-$v_{k_{\alpha}}^{\alpha}$, not necessarily induced, with the following properties.
\begin{enumerate}[(a)]
	\item If we give color $\alpha$ to $x$ and update along $P^{\alpha}$, then all vertices of $P^{\alpha}$ will be colored.
	\item Assume that $v^{\alpha}_i$ gets colored in color $\gamma_i$, $i=1,\ldots,k_{\alpha}$.
				Then there is an edge of the form $v^{\alpha}_iv^{\alpha}_j$ with $\gamma_i=\gamma_j$.
	\item $V = V(P^1) \cup V(P^2)$.
\end{enumerate}	
\end{claim}
\begin{proof}
We give color $\alpha$ to $x$ and update exhaustively from $x$.
According to Claim~\ref{clm:improper}, after some round of updating an edge appears whose end vertices receive the same color.
We then stop the updating procedure.
During the whole updating procedure we record an auxiliary digraph $D=(W,A)$ as follows.
Initially, $W=\{x\}$ and $A= \emptyset$.
Whenever we update a vertex $u$ from a vertex $v$, we add the vertex $u$ to $W$ and the edge $(v,u)$ to $A$.
This gives a directed tree whose root is $x$.

We can find directed paths $R$ and $S$ in $T$ both starting in $x$ and ending in vertices $y$ and $z$, say, such that $y$ and $z$ are adjacent in $G$ and they receive the same color during the updating procedure.
We may assume that $R=u_1$-\ldots-$u_k$-$v_1$-\ldots-$v_r$ and $S=u_1$-\ldots-$u_k$-$w_1$-\ldots-$w_s$, where $R$ and $S$ share only the vertices $u_1,\ldots,u_k$.
For each vertex $v \in V(R) \cup V(S)$, let $c(v)$ be the color received by $v$ in the updating procedure.
Moreover, let $c'(v)$ be the unique color in $L(v) \setminus \{c(v)\}$. Observe that, setting $w_0=v_0=u_k$, we have that
$c'(w_i)=c(w_{i-1})$ for every $i \in \{1, \ldots, s\}$,
and $c'(v_i)=c(v_{i-1})$ for every $i \in \{1, \ldots, k\}$.

Consider the following, different updating with respect to $x$.
We again give color $\alpha$ to $x$, and then update along $R$.
Now we update $w_s$ from $v_r$, thus giving it color $c'(w_s)$.
This, in turn, means we can update $w_{s-1}$ from $w_s$, giving it color $c'(w_{s-1})$, and so on.
Finally, when we update $w_1$ and it receives color $c'(w_1)$, an edge appears whose end vertices are colored in the same color.
Indeed, $u_kw_1$ is such an edge since $c(u_k)=c'(w_1)$.
Summing up, the path 
$$ \mbox{$P^\alpha=u_1$-\ldots-$u_k$-$v_1$-\ldots-$v_r$-$w_s$-$w_{s-1}$-\ldots-$w_1$} $$
starts in $x$ and, when we give $x$ the color $\alpha$ and update along $P^{\alpha}$, we obtain an improper partial coloring.
As $\alpha\in \{1,2\}$ was arbitrary, the assertions (a) and (b) follow.

To see (c), just note that the graph $G|(V(P_1) \cup V(P_2))$ is an obstruction: giving either color of $L(x)$ to $x$ and updating exhaustively yields a monochromatic edge.
By the minimality of $G$, $G=G|(V(P_1) \cup V(P_2))$ and so (c) holds.
\end{proof}

\begin{claim}\label{clm:V1=1}
Assume that $|V_1|=1$, say $V_1=\{x\}$ with $L(x)=\{\alpha\}$.
Then the following holds:
	\begin{enumerate}[(a)]
		\item there is a Hamiltonian path $P=v_1$-\ldots-$v_k$ of $G$ with $x=v_1$;
		\item updating from $x=v_1$ along $P$ assigns a color $\gamma_i$ to $v_i$, $i=1,\ldots,k$; and
		\item there is an edge of the form $v_iv_j$ with $\gamma_i=\gamma_j$.
	\end{enumerate}
\end{claim}
\begin{proof}
We assign color $\alpha$ to $x$ and update exhaustively from $x$.
Let $c$ be the obtained partial coloring.
According to Claim~\ref{clm:improper}, there is an edge $uv$ of $G$ with $c(u)=c(v)$.
Since $(G,L)$ is a minimal obstruction, every vertex of $G$ received a color in the updating process: otherwise, we could simply remove such a vertex and still have an obstruction.

Repeating the argument from the proof of Claim~\ref{clm:V1=0}, we obtain a path $P$ that starts in $x$ and, when we give $x$ color $\alpha$ and update along $P$, we obtain an improper partial coloring.
This proves (b) and (c).
Due to the minimality of $(G,L)$, $P$ is a Hamiltonian path, which proves (a).
\end{proof}

\begin{claim}\label{clm:V1=2}
Assume $|V_1|=2$, say $V_1=\{x,y\}$ with $L(x)=\{\alpha\}$ and $L(y)=\{\beta\}$.
Then the following holds:
\begin{enumerate}[(a)]
	\item there is a Hamiltonian path $P=v_1$-\ldots-$v_k$ of $G$ with $x=v_1$ and $y=v_k$; and
	\item updating from $v_1$ along $P$ assigns the color $\beta$ to $v_{k-1}$.
\end{enumerate}
\end{claim}
\begin{proof}
We color $x$ with color $\alpha$ and update exhaustively from $x$, but only considering vertices and edges of the graph $G\setminus y$.
Let $c$ be the obtained partial coloring.
By minimality, $c$ is proper.
According to Claim~\ref{clm:improper}, there is a neighbor $u$ of $y$ in $G$ with $c(u)=\beta$.

Like in the proof of Claim~\ref{clm:V1=0} and Claim~\ref{clm:V1=1}, we see that there is a path $P$ from $x$ to $y$ whose last edge is $uy$ such that giving color $\alpha$ to $x$ and then updating along $P$ implies that $u$ is colored with color $\beta$, which implies (b).
Due to the minimality of $(G,L)$, $P$ is a Hamiltonian path , and thus (a) holds.
\end{proof}

We can now prove our main lemma.

\begin{proof}[Proof of Lemma~\ref{lem:propagationpath}]
Recall from Claim~\ref{clm:V1small} that $|V_1|\le 2$.

\subparagraph{Case $|V_1|=0$.}
For this case Claim~\ref{clm:V1=0} applies and we obtain $x$, $P^1$ and $P^2$ as in the statement of the claim.
We may assume that, among all possible choices of $x$, $P^1$ and $P^2$, the value $\max\{|V(P^1)|,|V(P^2)|\}$ is minimum and, subject to this, $\min\{|V(P^1)|,|V(P^2)|\}$ is minimum.

Let us say that $P^1=v_1$-$v_2$-$\ldots $-$v_s$, where $v_1=x$.
Consider $v_1$ to be colored in color 1, and update along $P^1$, but only up to $v_{s-1}$.
Due to the choice of $P^1$ and $P^2$ being of minimum length, the coloring so far is proper.
Now when we update from $v_{s-1}$ to $v_s$, two adjacent vertices receive the same color.
Let the partial coloring obtained so far be denoted $c$.
Let $X$ be the set of neighbors $w$ of $v_s$ in $V(P^1)$ with $c(w)=c(v_s)$, and let $r$ be minimum such that $v_r \in X$.

We claim that $s-r\le \lambda$.
To see this, let $c'(v_j)$ be the unique color in $L(v_j)\setminus \{c(v_j)\}$, for all $j=1,\ldots,s$.
We claim that the following assertions hold.
\begin{enumerate}[(a)]
	\item $c(v_j)=c'(v_{j+1})$ for all $j=r,\ldots,s-1$.
	\item For every edge $v_iv_j$ with $r \le i,j\le s-1$ it holds that $c(v_i) \neq c(v_j)$.
	\item For every edge $v_iv_j$ with $r \le i,j\le s$ and $j-i \ge 2$ it holds that $c(v_i) \neq c'(v_j)$.
	\item For every edge $v_iv_j$ with $r+2 \le i,j\le s$ it holds that $c'(v_i) \neq c'(v_j)$.
\end{enumerate}
Assertion (a) follows from the fact that $P$ obeys the assertions of Claim~\ref{clm:V1=0}.
For (b), note that the choice of $P_1$ to be of minimum length implies that until we updated $v_s$, the partial coloring is proper.

To see (c), suppose there is an edge $v_iv_j$ with $r \le i,j\le s$ and $j-i \ge 2$ such that $c(v_i) = c'(v_j)$.
Then the path $P^1$ can be shortened to the path $v_1$-\ldots-$v_i$-$v_j$-\ldots-$v_s$, which is a contradiction.

Now we turn to (d), and consider the following coloring.
We color $P^1$ as before up to $v_r$.
Now we update from $v_r$ to $v_s$, giving color $c'(v_s)$ to $v_s$.
Then we color $v_{s-1}$ with color $c'(v_{s-1})$, then $v_{s-2}$ with color $c'(v_{s-2})$, and so on, until we reach $v_{r+1}$.
But $c'(v_{r+1})=c(v_r)$ due to (a), which means that the path $Q^1=v_1$-$v_2$-$\ldots $-$v_r$-$v_s$-$v_{s-1}$-$v_{s-2}$-$\ldots $-$v_{r+1}$ is a choice equivalent to $P^1$.
In particular, due to the choice of $P^1$ and $P^2$, the constructed coloring of $Q^1$ is proper if we leave out $v_{r+1}$.
Hence, there is no edge $v_iv_j$ with $r+2 \le i,j\le s$ such that $c'(v_i) = c'(v_j)$.
This yields (d).

From (a)-(d) it follows that every edge $v_iv_j$ with $r+2 \le i < j \le s$ and $i \le j-2$ is such that 
\begin{equation}\label{eqn:conditions2}
\mbox{$S(v_i)=\alpha\beta$ and $S(v_j)=\beta\gamma$,} 
\end{equation}
where $\{1,2,3\} = \{\alpha,\beta,\gamma\}$.
Consequently, the path $v_r$-\ldots-$v_{s-1}$ is a propagation path.
By assumption, $|\{v_r,\ldots,v_{s-1}\}| \le \lambda$ and so $s-r \le \lambda$.

A symmetric consideration holds for $P^2$.
Let us now assume that $|V(P^1)| \ge |V(P^2)|$.
It remains to show that $r$ is bounded by some constant.
To this end, recall that $\lambda \ge 20$.

Suppose that there is an edge $v_iv_j$ with $3 \le i \le j\le r$ such that $c'(v_i) = c'(v_j)$.
We then put $x'=v_j$, $Q^1=v_j$-\ldots-$v_s$, and $Q^2=v_j$-\ldots-$v_i$.
But this is a contradiction to the choice of $x$, $P^1$, and $P^2$, as $\max\{|V(Q^1)|,|V(Q^2)|\}<\max\{|V(P^1)|,|V(P^2)|\}$.
In addition to the assertion we just proved, which corresponds to assertion (d) above, the assumptions (a)-(c) from above also hold here, where we replace $r$ by 1 and $s$ by $r$.
Hence, using the same argumentation as above, we see that $r\le \lambda+1$.
Summing up, we have $|V| \le |V(P^1) \cup V(P^2)| \le 2 |V(P^1)| \le 4\lambda+2$, as desired.

\subparagraph{Case $|V_1|=1$.}
Now Claim~\ref{clm:V1=1} applies and we obtain the promised path, say $P=v_1$-\ldots-$v_s$, with $|L(v_1)|=1$.
Let us say $L(v_1)=\{1\}$.
Consider $v_1$ to be colored in color 1, and update along $P$, but only up to $v_{s-1}$.
Due to the choice of $P$, the coloring so far is proper.
Now when we update from $v_{s-1}$ to $v_s$, two adjacent vertices receive the same color.
Let the partial coloring obtained so far be denoted $c$.
Let $X$ be the set of neighbors $w$ of $v_s$ on $P$ with $c(w)=c(v_s)$, and let $r$ be minimum such that $v_r \in X$.
Moreover, let $c'(v_j)$ be the unique color in $L(v_j)\setminus \{c(v_j)\}$, for all $j=2,\ldots,s$.
Just like in the case $|V_1|=0$, we obtain the assertions (a)-(d) from above and this implies $s-r\le \lambda$.

It remains to show that $r\le \lambda+1$.
To see this, suppose that there is an edge $v_iv_j$ with $2 \le i \le j\le r$ such that $c'(v_i) = c'(v_j)$.
We then put $P^1=v_j$-\ldots-$v_s$ and $P^2=v_j$-\ldots-$v_i$.
Now, if we give color $c(v_j)$ to $v_j$ and update along $P^1$ we obtain an improper coloring.
Moreover, if we give color $c'(v_j)$ to $v_j$ and update along $P^2$ we also obtain an improper coloring.
This means that the pair $(G|(V(P^1)\cup V(P^2)),L)$ is not colorable, in contradiction to the minimality of $(G,L)$.

The assertion we just proved corresponds to assertion (d) above, and the assumptions (a)-(c) also hold here, where we replace $r$ by 1 and $s$ by $r$.
Hence, we know $r\le \lambda+1$ and obtain $|V| = |V(P)| \le 2 \lambda+1$.

\subparagraph{Case $|V_1|=2$.}
Claim~\ref{clm:V1=2} applies and we obtain the promised path, say $P=v_1$-\ldots-$v_s$, with $|L(v_1)|=|L(v_s)|=1$.
Let us say $L(v_1)=\{\alpha\}$ and $L(v_s)=\{\beta\}$.
Consider $v_1$ to be colored in color $\alpha$, and update along $P$, but only up to $v_{s-1}$.
Due to the choice of $P$, the partial coloring so far is proper.
Let the partial coloring obtained so far be denoted $c$, and put $c(v_s)=\beta$.
We now have $c(v_{s-1})=c(v_s)$, and this is the unique pair of adjacent vertices of $G$ that receive the same color.

For each $j=2,\ldots,s-1$, we denote by $c'(v_j)$ the unique color in $L(v_j)\setminus \{c(v_j)\}$.
We will show  that $s \leq \lambda +1$.
Just like in the cases above the following assertions apply.
\begin{enumerate}[(a)]
	\item $c(v_j)=c'(v_{j+1})$ for all $j=1,\ldots,s-2$.
	\item For every edge $v_iv_j$ with $1 \le i,j\le s-1$ it holds that $c(v_i) \neq c(v_j)$.
	\item For every edge $v_iv_j$ with $1 \le i,j\le s-1$ and $j-i \ge 2$ it holds that $c(v_i) \neq c'(v_j)$.
	\item For every edge $v_iv_j$ with $3 \le i,j\le s-1$ it holds that $c'(v_i) \neq c'(v_j)$.
\end{enumerate}
Let $r'=1$ and $s'=s-1$.
As above we see that $s'-r'\le \lambda-1$.
Hence, $s \le \lambda + 1$.
From the fact that $P$ is a Hamiltonian path in $G$ we obtain the desired bound $|V| = |V(P)| \le \lambda+1$.
This completes the proof.
\end{proof}

\section{$P_6$-free minimal list-obstructions}\label{sec:sufficiency-P6}

The aim of this section is to prove that there are only finitely many $P_6$-free minimal list-obstructions.
In Section~\ref{sec:size2} we prove the following lemma which says that there are only finitely many $P_6$-free minimal list-obstructions with lists of size at most two.

\begin{lemma}\label{lem:size2}
Let $(G,L)$ be a $P_6$-free minimal list-obstruction for which $|L(v)|\le 2$ and $L(v) \subseteq \{1,2,3\}$ holds for all $v \in V(G)$.
Then $|V(G)|\le 100$.
\end{lemma}

Our proof of this lemma is computer-aided.
We also have a computer-free proof, but it is tedious and complicated, and 
gives a significantly worse bound on the size of the obstructions, so we will only sketch the idea of the computer-free proof.

In Section~\ref{sec:generalcase} we solve the general case, where each list may have up to three entries, making extensive use of Lemma~\ref{lem:size2}.
We prove the following lemma.

\begin{lemma}\label{lem:3downto2}
There exists an integer $C$ such that the following holds. Let $G$ be a $P_6$-free graph, and let $L$ be a list system such that $L(v) \subseteq \{1,2,3\}$
for every $v \in V(G)$.
Suppose that $(G,L)$ is a minimal list obstruction. Then $|V(G)| \leq C$. Consequently,
there are only finitely many $P_6$-free minimal list-obstructions.
\end{lemma}

The main technique used  in the proof of Lemma~\ref{lem:3downto2} is to guess the coloring on a small set  $S$ of vertices of the minimal list-obstruction at hand, $(G,L)$ say.
After several transformations, we arrive at a list-obstruction $(G,L')$ where each list has size at most two, and so we may apply Lemma~\ref{lem:size2} to show that there is a minimal list-obstruction $(H,L')$ with a bounded number of vertices induced by $(G,L')$.
We can prove that $G$ is essentially the union of these graphs $H$ (one for each coloring of $S$), and so the number of vertices of $G$ is bounded by a function of the number of guesses we took in the beginning.
Since we precolor only a (carefully chosen) small part of the graph, we can derive that the number of vertices of $G$ is bounded by a constant.

To find the right vertex set to guess colors for, we use a structure theorem for $P_t$-free graphs~\cite{CS14} that implies the existence of a well-structured connected dominating subgraph of a minimal list-obstruction.

\subsection{Proof of Lemma~\ref{lem:size2}}\label{sec:size2}

Let $(G,L)$ be a $P_6$-free minimal list-obstruction such that every list contains at most two colors.
Suppose that $P=v_1$-\ldots-$v_k$ is a propagation path in $(G,L)$.
We show that if $G$ is $P_6$-free, then $k\leq 24$.
In view of Lemma~\ref{lem:propagationpath}, this proves that $G$ has at most 100 vertices.

Our proof is computer-aided, but conceptually very simple. 
%(We also have a 
%computer-free proof that produces worse bounds; we present it in the Appendix).
The program generates the paths $v_1$, $v_1$-$v_2$, $v_1$-$v_2$-$v_3$, and so on, 
lists for each $v_i$, as in the definition of a propagation path, and
edges among the vertices in the path.  Whenever a $P_6$ or an edge
violating condition~\eqref{eqn:conditions0} of the definition of a propagation path is found, the 
respective branch of the search tree is closed. Since the program does not 
find such a path on 25 vertices (cf. Table 1), our claim is proved.

\begin{table}[ht!]
\centering
{\footnotesize
\begin{tabu}{ l c c c c c c c c c }
\tabucline[0.9pt]{-}\\[-7pt]
Vertices 			&	1 & 2 &  3 & 4  & 5  & 6   & 7    & 8  &       \\[1pt]

Propagation paths	& 1 & 2 &  6 & 22 & 86 & 350 & 1 220 & 2 656  & \\[1pt]
\hline\\[-7pt]
Vertices 			& 9     & 10    & 11     &	12   & 13   &  14   & 15   & 16 &  \\[1pt]

Propagation paths	& 4 208 & 5 360 & 5 864  & 5 604 & 5 686 &  5 004 & 4 120 & 3 400 &  \\[1pt]
\hline\\[-7pt]
Vertices 			& 17    & 18    & 19    & 20     & 21    & 22     &	23   & 24   &  25    \\[1pt]

Propagation paths	& 2 454 & 1 688 & 1 064 & 516    & 202   & 72     & 18   & 2 &  0  \\[1pt]
\tabucline[0.9pt]{-}
\end{tabu}
}
\caption{Counts of all $P_6$-free propagation paths with lists of size 2 meeting condition~\eqref{eqn:conditions0} generated by Algorithm~\ref{algo:prefix}.}

\label{table:counts_prefixpaths-P6free}
\end{table}

The pseudocode of the algorithm is shown in Algorithm~\ref{algo:prefix} and~\ref{algo:construct}. 
Our implementation of this algorithm can be downloaded from~\cite{listcriticalpfree-site}. 
Table~\ref{table:counts_prefixpaths-P6free} lists the number of configurations generated by our program.

\begin{algorithm}[h]
\caption{Generate propagation paths and lists}
\label{algo:prefix}
  \begin{algorithmic}[1]
	\STATE $H\gets (\{v_2\},\emptyset)$
	\STATE $c(v_1)\gets 1$
	\STATE $L(v_1)\gets \{1\}$
	%\COMMENT{Comment: We may assume that $L(v_r)=\{1,2\}$ and $c(v_r)=1$.}
	\STATE Construct($H,c,L$)\\ ~~~~~// We may assume $c(v_1)=1$ and $L(v_1)=\{1\}$.
  \end{algorithmic}
\end{algorithm}

\begin{algorithm}[ht!]
\caption{Construct(Graph $H$, coloring $c$, list system $L$)}
\label{algo:construct}
  \begin{algorithmic}[1]
		\STATE $j \gets |V(H)|$
		\STATE $V(H)\gets V(H) \cup \{v_{j+1}\}$
		\STATE $E(H)\gets E(H) \cup \{v_{j}v_{j+1}\}$
		\\ ~~~~~// This extends the path by the next vertex $v_{j+1}$.
		\FORALL{$\alpha \in \{1,2,3\}\setminus \{c(v_j)\}$ and all $I \subseteq \{1,2, \ldots, j-1\}$}
			\STATE $H' \gets H$
			\STATE $E(H') \gets E(H') \cup \{v_iv_{j+1}: i \in I\}$
			\\ ~~~~~// This adds edges from $v_{j+1}$ to earlier vertices in all possible ways.
			\STATE $c(v_{j+1})\gets \alpha$
			\STATE $L(v_{j+1})\gets \{\alpha,c(v_j)\}$
			\IF{$(H',c,L)$ is $P_6$-free and satisfies condition~\eqref{eqn:conditions0}}
				\STATE Construct($H',c,L$)
			\\ ~~~~~// If the propagation path is not pruned, we extend it further.
			\ENDIF
		\ENDFOR
  \end{algorithmic}
\end{algorithm}

Next we sketch the idea of the computer-free proof. Let $(G,L)$ be a minimal 
list obstruction with all lists of size at most two, and suppose for a 
contradiction that  there is a (very) long propagation path $P$ in $G$.
We may assume that $G$ does not contain a clique 
with four vertices. It follows from the main result of \cite{GSS} that $G|V(P)$
contains a large induced subgraph $H$, which is a complete bipartite graph;
let $(A,B)$ be a bipartition of $H$. Using Ramsey's Theorem~\cite{Ramsey}  we may assume that
all vertices of $A$ have the same shape, and all vertices of $B$ have the same 
shape (by ``coloring'' the edges of $H$ by the shapes of their ends). We can 
now choose a large subset $A'$ of $A$ all of whose members are pairwise far 
apart  in $P$,  and  such that $A'$ is far in $P$ from some subset  $B'$ of 
$B$. We analyze the structure of short subpaths of $P$ containing each 
$a \in A$,  and the edges between such subpaths, and to $B'$. We can again use 
Ramsey's Theorem to  assume that the structure is the same for every member of 
$A'$ and every member of $B'$. Finally, we accumulate enough structural 
knowledge to find a $P_6$ in $G$, thus reaching a contradiction. 
If instead of $P_6$ we wanted to use the same method to produce $P_5$, the 
argument becomes much shorter, and it was carried out in \cite{Zhang}.
 
\subsection{Reducing obstructions}
In this section we prove three lemmas which help us reduce the size of the obstructions.
These lemmas will be used in the proofs of Sections~\ref{sec:sufficiency-P6}.3 and~\ref{sec:sufficiency-2P3}.2. 

Let $(G,L)$ be a list-obstruction and let $R$ be an induced subgraph of $G$. Let $\mathcal{L}$ be a set of subsystems of $L$ satisfying the following assertions.
\begin{enumerate}
\item For every $L' \in \mathcal{L}$ there exists an induced subgraph $R(L')$ of $R$ such that $|L'(v)|=1$ for every $v \in V(R(L'))$.	
\item For each $L'\in \mathcal{L}$, $L'(v)=L(v)$ for $v\in V(G)\setminus R(L')$ and $L'(v)\subseteq L(v)$ for $v\in R(L')$.
\item For every $L$-coloring $c$ of $R$ there exists a list system $L'\in \mathcal{L}$ such that $c(v) \in L'(v)$ for every $v \in R(L')$.
\end{enumerate}
We call $\mathcal L$ a \emph{refinement of $L$ with respect to $R$}.
Observe that $\{L\}$ with $R(L)$ being the empty graph is a refinement of
$L$ with respect to $G$, though this is not a useful refinement.
For each list system $L'\in \mathcal{L}$ it is clear that $(G,L')$ is again a list-obstruction, though not necessarily a minimal one, even if $(G,L)$ is minimal.

\begin{lemma}\label{precolor}
Assume that $(G,L)$ is a minimal list-obstruction. Let $R$ be an induced 
subgraph of $G$, and let $\mathcal{L}= \{L_1,L_2,\ldots,L_m\}$ be a refinement of $L$ with respect to $R$. 
For every $L_i \in \mathcal{L}$, let $(G_{L_i},L_i)$ be a minimal obstruction
induced by $(G,L_i)$.

Then  $V(G)=R \cup \bigcup_{L_i \in \mathcal{L}}V(G_{L_i})$.
Moreover, if each $G_{L_i}$ can be chosen such that 
$|V(G_{L_i} \setminus R| \leq k$, then $G$ has at most $|V(R)|+km$ vertices.
\end{lemma}
\begin{proof}
	Let $(G_i,L_i|_{G_i})$ be a minimal list-obstruction induced by $(G,L_i)$ such that $|V(G_i)\setminus V(R)| \leq k$, $i=1,\ldots,m$. 
	Suppose for a contradiction that there exists a vertex $v$ in 
$V(G)\setminus R$ such that $v$ is contained in none of the 
$G_i$, $i=1,\ldots,m$. 
	By the minimality of $(G,L)$, $G\setminus \{v\}$ is
	$L$-colorable. Let $c$ be an 
	$L$-coloring
	of $G \setminus \{v\}$.
	We may assume that $c(r) \in L_1(r)$ for every $r \in V(R(L_1))\cap V(G_1)$. 
	Then $c$ is a coloring of $(G_1,L_1)$, which is a contradiction. 
This proves the first assertion of Lemma~\ref{precolor}.
	Consequently,   $$|V(G)| \le |\bigcup_{i=1}^m V(G_i)| \le |V(R)|+|\bigcup_{i=1}^m V(G_i\setminus R)|$$ and the second assertion follows.

\end{proof}

Next we prove a lemma which allows us to update three times with respect to a set of vertices with lists of size $1$. 

\begin{lemma}\label{update4}
	Let $(G,L)$ be a list-obstruction, and let $X \subseteq V(G)$ be a vertex subset such that $|L(x)|=1$ for every $x \in X$. 
	Let $L'$ be the list obtained by updating with respect to $X$ three times. 
	Let $(G',L')$ be a minimal list-obstruction induced by $(G,L')$. 
	Then there exists a minimal list-obstruction induced by $(G,L)$, say $(G'',L)$, with $|V(G'')| \leq 36|V(G')|$.
\end{lemma}
\begin{proof}
	Let $Y=X_1$ and $Z=X_2$, as in the definition of updating $i$ times.
	We choose sets $R$, $S$, and $T$ as follows.
	\begin{itemize}
		\item For every $v \in V(G') \setminus (X \cup Y \cup Z)$, define $R(v)$ to be a minimum subset of $(X \cup Y \cup Z)\cap N(v)$ such that $\bigcup_{s \in R(v)}L'(s)=L(v) \setminus L'(v)$, and let $R=\bigcup_{v \in V(G') \setminus (X \cup Y \cup Z)}R(v)$.
		\item For every $v \in (V(G')\cup R) \cap Z$, define $S(v)$ to be a minimum subset of $(X\cup Y)\cap N(v)$ such that $\bigcup_{s \in S(v)}L'(s)=L(v) \setminus L'(v)$, and let $ S=\bigcup_{v \in (V(G')\cup R) \cap Z}S(v) $. 
		\item For every $v \in (V(G')\cup R\cup S) \cap  Y$, define $T(v)$ to be a minimum subset of $X\cap N(v)$ such that $\bigcup_{s \in T(v)}L'(s)=L(v) \setminus L'(v)$, and let $ T= \bigcup_{v \in (V(G')\cup R\cup S) \cap  Y}T(v) $.
	\end{itemize}
	Clearly, $|R(v)| \leq 3$ for every $v \in V(G') \setminus (X \cup Y \cup Z)$, $|S(v)|\le 2$ for every $v \in (V(G')\cup R) \cap Z$, and $|T(v)|\le 2$ for every $v \in (V(G')\cup R\cup S) \cap  Y$.
	Let $P=R\cup S\cup T\cup V(G')$, and observe that $|P|\leq (1+3+8+24) |V(G')|=36|V(G')|$.
	It remains to prove that $(G|P,L)$ is not colorable.
	Suppose there exists a coloring $c$ of $(G|P,L)$.
	Then $c$ is not a coloring of $(G',L')$, and since $V(G') \subseteq P$, it follows that there exists $w \in V(G')$ such that $c(w) \not \in L'(w)$.
	Therefore $c(w) \in L(w) \setminus L'(w)$.
	
	We discuss the case when $v \in V(G') \setminus (X \cup Y \cup Z)$, as the cases of $v \in (V(G')\cup R) \cap Z$ and $v \in (V(G')\cup R\cup S) \cap  Y$ are similar.
	We can choose $m\in R(w)$ such that $L'(m)=\{c(w)\}$ and one of the following holds.
	\begin{itemize}
		\item $m \in X$, and thus $L(m)=L'(m)=\{c(w)\}$.
		\item $m \in Y$, and thus for any $ i\in L(m)\setminus L'(m)$ there exists $n_i\in T(m)$ such that $L(n_i)=\{i\}$.
		\item $m \in Z$, and thus for any $ i\in L(m)\setminus L'(m)$ there exists $n_i\in S(m)$ such that either $L(n_i)=\{i\}$ or, for any $ j\in L(n_i)\setminus \{i\}$, there exists $l_j\in T(n_i)$ with $L(l_j)=\{j\}$.
	\end{itemize}
	In all cases it follows that $c(m)=c(w)$, in contradiction to the fact that $m$ and $w$ are adjacent.
	This completes the proof.
\end{proof}

Let $A$ be a subset of $V(G)$ and $L$ be a list system; let $c$ be an $L$-coloring of $G|A$, and let $L_c$ be the list system obtained by setting $L_c(v)=\{c(v)\} $ for every $v\in A$ and updating with respect to $A$ three times; we say that $L_c$ is \emph{obtained from $L$ by precoloring $A$ (with $c$)  and updating three times}. If for every $c$, $|L_c(v)|\leq 2$ for every $v\in V(G)$, we call $A$ a \emph{semi-dominating set} of $(G,L)$. If $L(v)=\{1,2,3\}$ for every $ v\in G $ and $A$ is a semi-dominating set of $(G,L)$, we say that $A$ is a \emph{semi-dominating set} of $G$.  Note that a dominating set is always a semi-dominating set.
Last we prove a lemma for the case when $G$ has a bounded size semi-dominating set.

\begin{lemma}\label{P_4}
	Let $(G,L)$ be a minimal list-obstruction, and assume that $G$ has a semi-dominating set $A$ with $|A| \leq t$. Assume also that if $(G',L')$ is a minimal 
	obstruction where $G'$ is an induced subgraph of 
	$G$, and $L'$ is a subsystem of $L$ with $|L'(v)| \leq 2$ for every $v$, then
	$|V(G')| \leq m$.  Then $|V(G)| \leq 36\cdot 3^t \cdot m +t$.
\end{lemma}
\begin{proof}
	Consider all possible $L$-colorings $c_1, \ldots, c_s$ of $A$; then 
	$s \leq 3^t$. For each $i$, let $L_i$ be the list system obtained by updating 
	with respect to $A$ three times. Then  $|L_i(v)| \leq 2$ for every $v \in V(G)$ and
	for every $i \in \{1, \ldots, s\}$.  Now Lemma~\ref{precolor} together with Lemma~\ref{update4} imply that 
	$|V(G)| \leq 36\cdot 3^t \cdot m +t$. This completes the proof.
\end{proof}

\subsection{Proof of Lemma~\ref{lem:3downto2}}\label{sec:generalcase}

%In this section, for a graph $G$ and a vertex $v$, we write $v\in G$ to denote $v\in V(G)$ for convenience.
%We start with a claim which allows us to control the number of vertices when we precolor subgraph of list-obstructions.

We start with several claims that deal with vertices that have a special structure in their neighborhood.

\begin{claim}\label{mixed}
Let $G$ be a graph, and let $X \subseteq V(G)$ be connected. If $v \in V(G) \setminus X$ is mixed on $X$, then there exist adjacent $x_1,x_2 \in X$ such that
$v$ is adjacent to $x_1$ and non-adjacent to $x_2$.
\end{claim}

\begin{proof}
Since $v$ is mixed on $X$, both the sets $N(v) \cap X$ and $X \setminus N(v)$
are non-empty, and since $X$ is connected, there exist $x_1 \in N(v) \cap X$
and $x_2 \in X \setminus N(v)$ such that $x_1$ is adjacent to $x_2$. This
proves Claim~\ref{mixed}.
\end{proof}

\begin{claim}\label{biplemma}
Let $G$ be a $P_6$-free graph and let $v \in V(G)$. 
Suppose that $G|N(v)$ is a connected bipartite graph with bipartition $(A,B)$. 
Let $G'$ be obtained from $G\setminus (A \cup B)$ 
by adding two new vertices $a,b$ with
$N_{G'}(a)=\{b\}  \cup \bigcup_{u \in A} (N_G(u) \cap V(G'))$
and $N_{G'}(b)=\{a\} \cup \bigcup_{u \in B} (N_G(u) \cap V(G'))$. 
Then $G'$ is $P_6$-free.
\end{claim}
\begin{proof}
Suppose $Q$ is a $P_6$ in $G$. Then $V(Q) \cap \{a,b\} \neq \emptyset$.
Observe that if both $a$ and $b$ are in  $V(Q)$, then $v \not \in V(Q)$.
If only one vertex of $Q$, say $q$,  has a neighbor in $\{a,b\}$, say $a$,
then we get a $P_6$ in $G$ by replacing $a$ with a neighbor of $q$ in $A$, 
and, if $b \in V(Q)$, replacing $b$ with $v$. Thus we may assume that
two vertices $q,q'$ of $Q$ have a neighbor in $\{a,b\}$. If $q$ and $q'$ have a
common neighbor $u \in A \cup B$, then $G|((V(Q) \setminus \{a,b\} \cup \{u\}))$
is a $P_6$ in $G$, a contradiction. So no such $u$ exists, and in particular
$v \not \in V(Q)$. 
Let $Q'$ be an induced path from $q$ to 
$q'$ with $V(Q') \setminus \{q,q'\} \subseteq A \cup B \cup \{v\}$,
meeting only one of the sets $A,B$ if possible. Then 
$G|((V(Q) \setminus \{a,b\}) \cup V(Q'))$
is a $P_6$ in $G$, a contradiction. This proves Claim~\ref{biplemma}.
\end{proof}

In the remainder of this section $G$ is a $P_6$-free graph.

\begin{claim}
\label{stable}
Let $(G,L)$ be a minimal list-obstruction.
Let $A$ be a stable  set in $G$. Let $U$ be the set of vertices of 
$V(G) \setminus A$ that are not mixed on $A$, and let
$k=|V(G) \setminus (A \cup U)|$. Then $|A| \leq 7 \times 2^k$.
\end{claim}

\begin{proof}
Partition $A$ by the adjacency in 
$V(G) \setminus (A \cup U)$ and by lists; more precisely let $A=A_1\cup A_2\cup\ldots \cup A_{7 \times 2^k}$ such that for every $i$ and for every $x,y\in A_i$, 
$N(x) \setminus (A \cup U) =N(y) \setminus (A \cup U)$ and $L(x)=L(y)$.
Since $A$ is stable and no vertex of $U$ is mixed on $A$, it follows that
for every $x,y \in A_i$, $N(x)=N(y)$. 
We claim that  $A_i \leq 1$ for every $i$.
Suppose for a contradiction that there exist $x,y \in A_i$ with $x \neq y$. 
By the minimality of $(G,L)$, $G\setminus x$ is $L$-colorable. 
Since  $N(x)= N(y)$ and $L(x)=L(y)$, we deduce that $G$ is $L$-colorable by giving $x$ the same 
color as  $y$, a contradiction. This proves Claim~\ref{stable}.
\end{proof}

\begin{claim}\label{bipartite1}
Let $(G,L)$ be a minimal list-obstruction.
Let $H$ be an induced subgraph of $G$ such that $H$ is connected and bipartite. 
Let $(A,B)$ be the bipartition of $H$, and let $u \in V(G)$ be complete to $A \cup B$.  Let $U$ be the set of all vertices in $V(H) \setminus (A \cup B)$ 
that are not mixed on $A$. Let $K= V(G) \setminus (A \cup B \cup U)$ and 
$k=|K|$. Then $|A|\leq  7\cdot 2^{7\cdot 2^k+k}$.
\end{claim}
\begin{proof}
We partition $A$ according to the adjacency in $K$ and the lists of the vertices. 
More precisely, let $A=A_1\cup A_2\cup\ldots \cup A_{7\cdot 2^k}$ such that for any $i$ and for any $x,y\in A_i$, $N(x)\cap K=N(y)\cap K$ and $ L(x)=L(y)$.
Analogously, let $B=B_1\cup B_2\cup\ldots \cup B_{7\cdot 2^k}$. 

Next, for $i = 1,\ldots,7\cdot 2^k$, we partition $A_i=A_i^1\cup\ldots \cup A_i^{ 7\cdot 2^k}$ according to the sets $B_1, B_{7\cdot 2^k}$ in which they have neighbors. 
More precisely, for any $t$ and any $x,y\in A_i^t$, $N(x)\cap B_j\neq \emptyset$ if and only if $N(y)\cap B_j\neq \emptyset$, $j = 1,\ldots,7\cdot 2^k$. 
We partition the sets in $B$ analogously. 
We claim that $A_i^j \leq 1$.
Suppose that there exist $x,y\in A_i^j$ with $x\neq y$. Let $N=N(x) \cap B$.   
Let $C$ be the component of $H \setminus x$ with $y \in C$. 
Let $N_1$ be the  set of vertices of $N$ whose unique neighbor in $H$ is $x$, and let $N_2=N \setminus N_1$.   

Suppose first that there is a vertex $s \in N_2 \setminus C$. 
Since  $y$ is not dominated by $x$, there exists $t\in (B\cap N(y))\setminus N(x)$. 
Since $s \not \in C$, it follows that $s$ and $t$ have no common neighbor in $H$, and so, since $G$ is $P_6$-free, there is a $5$-vertex path $Q$ in $H$ with ends
$s$ and $t$; let the vertices of $Q$ be $q_1$-$\ldots$-$q_5$, where $s=q_1$ and $t=q_5$.
Since $s \not \in C$, it follows that $q_2=x$.
Since $y$-$t$-$q_4$-$q_3$-$x$-$s$ is not a $P_6$ in $G$, it follows that $y$ is adjacent to $q_3$, and so we may assume that $q_4=y$.
Let $p \in A \setminus \{x\}$ be a neighbor of $s$. 
Since $p$-$s$-$x$-$q_3$-$y$-$t$ is not a $P_6$, it follows that $p$ has a neighbor in $\{t,q_3\}$, contrary to the fact that $s \not \in C$. 
This proves that $N_2 \subseteq C$.

Observe that  $C \cap N_1 = \emptyset$. 
By the minimality of $(G,L)$, there is a coloring $c$ of $(G\setminus (N_1 \cup \{x\}), L)$. 
Since $u$ is complete to $V(C)$, it follows that $C\cap A$ and $C \cap B$ are both monochromatic. 
We now describe a coloring of $G$. 
Color $x$ with $c(y)$.
Let $n_1 \in N_1$ be arbitrary.
Since $x,y \in A_i^j$, there exists $n_1'$ in $B$ such that $n_1'$ is adjacent to $y$, $L(n_1)=L(n_1')$, and $n_1,n_1'$ have the same neighbors in $K$. 
Now color $n_1$ with $c(n_1')$.
Repeating this for every vertex of $N_1$ produces a coloring of $(G,L)$ a contradiction. 
This proves~Claim~\ref{bipartite1}.
\end{proof}

\begin{comment}

\begin{claim}\label{bipartite2}
Let $v\in V(G)$ with $N(v)=A\cup B$ where $A=\{a_1,\ldots,a_s\}$ and $B=\{b_1,\ldots,b_s\}$. 
Assume that $N(a_i)\cap N(v)=B\setminus\{b_i\}$ and $N(b_i)\cap N(v)=A\setminus\{a_i\}$ for all $i \in \{1,\ldots,s\}$.  
Let $K\subseteq V(G) \setminus (N(v) \cup \{v\})$ be the set of vertices mixed on $A \cup B$, and let $k=|K|$.
Then $s\leq 7 \cdot 2^{k}$.
\end{claim}
\begin{proof}	
We may assume that $s\geq 3$. 
We partition $A$ according to the adjacency in $K$. 
More precisely, let $A=A_1\cup A_2\cup\ldots \cup A_{2^k}$ such that for any $i$ and for all $x,y\in A_i$, $N(x)\cap K=N(y)\cap K$.
From the minimality of $(G,L)$ it is sufficient to prove that there do not exist two vertices $x,y\in A^j$ with $L(x)=L(y)$. 
Suppose for a contradiction that such vertices $x,y$ exist. 
By the minimality of $G$, there is an $L$-coloring $c$ of $G\setminus x$. 
Note that both $A\setminus x$ and $B$ are monochromatic in $c$. 
Since $N(x)\subseteq N(y)\cup B$ and $c(y)$ is different from the color of $B$, $c$ can be changed to a $L$-coloring of $G$ by setting $c(x)=c(y)$, a contradiction. 
This completes the proof.
\end{proof}	

\end{comment}

\begin{claim}
\label{reducecomponents}
There is a function $q: \mathbb{N} \rightarrow \mathbb{N}$ such that the following holds.
Let $(G,L)$ be a minimal list-obstruction. Let $D_1, \ldots, D_t$ be connected subsets of $V(G)$ with the following properties.
\begin{itemize}
\item $|D_i|>1$ for every $i$,
\item $D_1, \ldots, D_t$ are pairwise disjoint and anticomplete to each other,
\item for each $i$ there is a set $U_i \subseteq V(G)$ such  that $U_i$ is complete to $D_i$,
\item $D_i$ is anticomplete to $V(G) \setminus (U_i \cup D_i)$,
\item for every $i \in \{1, \ldots, t\}$ there is $c_i \in \{1,2,3\}$
such that $c(u) \neq c_i$ for every coloring $c$ of $(G,L)$ and for every $u \in U_i$, and
\item $V(G) \neq D_1 \cup U_1$.
\end{itemize}
Then there is an induced subgraph $F$ of $G$  such that  
$V(G) \setminus \bigcup_{i=1}^t D_i \subseteq V(F) \subseteq V(G)$ and a list 
system $L'$ such 
that
\begin{itemize}  
\item $|L'(v)| \leq 2$ for every $v \in V(F) \cap  \bigcup_{i=1}^t D_i $,
\item $L'(v)=L(v)$ for every $v \in V(F) \setminus  \bigcup_{i=1}^t D_i $, and
\item $(F,L')$ is a minimal list-obstruction, and $|V(G)| \leq q(|V(F)|)$. 
\end{itemize}
\end{claim}

\begin{proof}
Write $D=\bigcup_{i=1}^t D_i$.
Since $V(G) \neq U_1 \cup D_1$, it follows from the minimality of $(G,L)$ 
that $(G|(D_i \cup U_i),L)$ is colorable for every $i$, and so 
each $G|D_i$ is bipartite. Let $D_i^1,D_i^2$ be the bipartition of $G|D_i$.
Then for every $i$ there is  a coloring of $(G|D_i,L) $ in which each of the 
sets $D_i^1,D_i^2$ is monochromatic, and in particular
$\bigcap_{d \in D_i^j}L(d) \neq \emptyset$ for every $i \leq t$ and 
$j \in \{1,2\}$.

For every $i \leq t$ and $j \in \{1,2\}$, let $d_i^j \in D_i^j$ such that
$d_i^1$ is adjacent to $d_i^2$. Set $L''(d_i^j)=\bigcap_{d \in D_i^j}L(d)$.
Let $F''$ be the graph obtained from $G$ by deleting
$D \setminus (\bigcup_{i=1}^m\{d_i^1,d_i^2\})$.
Set $L''(v)=L(v)$ for every $v \in V(F'') \setminus D$.
 
We may assume that there exists $s \in \{0,1, \ldots, t\}$ such that
$|L''(d_i^1)|=3$ for every  $i \leq s$, and that
for $i \in \{s+1, \ldots, t\}$ and $j \in \{1,2\}$, $|L''(d_i^j)| \leq 2$.

Let $F'$ be obtained from $F''$ by deleting $\{d_1^1, \ldots, d_s^1\}$.
For $i \in \{1, \ldots, s\}$, let $L'(d_i^2)=L''(d_i^1) \setminus  \{c_i\}$. 
Let $L'(v)=L''(v)$ for every other 
vertex of $F'$. Then $V(G) \setminus D \subseteq V(F')$, 
$L'(v)=L(v)$ for every $v \in V(F') \setminus D$, and $|L'(v)| \leq 2$
for $v \in V(F') \cap D$.

We claim  that $(F',L')$ is not colorable. Suppose $c'$ is a coloring of
$(F',L')$.  We construct a coloring of $(G,L)$. Set
$c(v)=c'(v)$ for every $v \in V(G) \setminus D$. 
For $i \in \{1, \ldots, t\}$ and for  every $d \in D_i^2$,
set $c(d)=c'(d_i^2)$. Then $c(d) \neq c_i$ if $i \leq s$. For 
$i \in \{1, \ldots, s\}$, set $c(d)=c_i$ for every $d \in D_i^1$, 
and for $i \in \{s+1, \ldots, t\}$, set $c(d)=c'(d_i^1)$ for every 
$d \in D_i^1$. Now $c$ is a coloring of $(G,L)$, a contradiction.

Let $(F,L')$ be a minimal obstruction induced by $(F',L')$.
We claim that $V(G) \setminus D \subseteq V(F)$. Suppose not, 
let $v \in V(G) \setminus (D \cup V(F))$.
It follows from the minimality of $G$ that
$(G\setminus v, L)$ is colorable. Let $c$ be such a coloring. 
Set $c'(v)=c(v)$ for every $v \in V(F) \setminus D$. 
Let $i \in \{1, \ldots, t\}$. If $U_i \neq \{v\}$, then
each of these sets $U_i,D_i^1,D_i^2$ is monochromatic in $c$.
If $U_i=\{v\}$, then $D_i$ is anticomplete to $V(G) \setminus (D_i \cup \{v\})$, and  we may assume that each of $D_i^1,D_i^2$ is monochromatic in $c$.
Now set $c'(d_i^2)$ to be the
unique color that appears in $D_i^2$, and for $i \in \{s+1, \ldots t\}$,
set $c'(d_i^1)$ to be the unique color that appears in $D_i^1$. Then
$c'$ is a coloring of $(F \setminus v, L)$, a contradiction.
Now Claim~\ref{reducecomponents} follows from at most $t \leq |V(F)|$ 
applications of Lemma~\ref{bipartite1}.
\end{proof}

\begin{claim}\label{clm:Kempe}
Let $x \in V(G)$ such that $L(x)=\{1,2,3\}$ and let $U,W \subseteq V(G)$ be disjoint non-empty sets such that $N(x)= U \cup W$ and $U$ is complete to $W$. 
Then (possibly exchanging the roles of $U$ and $W$) either
\begin{enumerate}
\item there is a path $P$ with $|V(P)|=4$, such that 
the ends of $P$ are in $U$, no internal vertex of $P$ is in $U$, and 
$V(P) \cap W=\emptyset$, or
\item there exist distinct $i,j \in \{1,2,3\}$ 
and vertices $u_i,u_j \in U$ and $w_i, w_j$ such that for every
$k \in \{i,j\}$ 
\begin{itemize}
\item $k \in L(u_k)$, 
\item $|L(w_k) \cap \{i,j\}|=1$, 
\item there is a path $P_k$ from $u_k$ to $w_k$, 
\item if $|V(P_k)|$ is even, then $k \not \in L(w_k)$,
\item if $|V(P_k)|$ is odd, then $k \in L(w_k)$ 
\end{itemize}
and $V(P_i)$ is anticomplete to $V(P_j)$.
\end{enumerate}
\end{claim} 
\begin{proof}
Since we may assume that $G \neq K_4$, it follows that $U$ and $W$ are both stable sets. 
By the minimality of $G$, there is an $L$-coloring of $G \setminus x$, say $c$.
Since $G$ does not have an $L$-coloring, we may assume that there exist $u_1,u_2 \in U$ such that $c(u_i)=i$. 
Then $c(w)=3$ for every $w \in W$, and $c(u) \in \{1,2\}$ for every $u \in U$. 
For $i=1,2$ let $U_i=\{u \in U : c(u)=i\}$. 
Let $V_{12}=\{v \in V(G) : c(v) \in \{1,2\}\}$, and let $G_{12}=G|V_{12}$. 
Suppose first that some component of $G_{12}$ meets both $U_1$ and $U_2$.
Let $P$ be a shortest path from $U_1$ to $U_2$ in $G_{12}$. 
Then $|V(P)|=4$ since $G$ is $P_6$-free.
Moreover, $V(P) \cap W = \emptyset$, and no interior vertex of $P$ is in $U$,
and \ref{clm:Kempe}.1 holds.

So we may assume that no component of $G_{12}$ meets both $U_1$ and $U_2$.
For $i=1,2$ let  $V_i$ be the union of the components of $G_{12}$ that meet $U_i$.
Then $V_1$ is anticomplete to $V_2$.
If we can exchange the colors $1$ and $2$ on every component that meets
$U_1$, then doing so produces a coloring of $G \setminus x$ where every vertex 
of $U$ is colored $2$ and every vertex of $W$ is colored $3$; 
this coloring can then be extended to $G$, which is a contradiction. 
So there is a component $D$ that meets $U_1$ and where such an exchange is 
not possible. Let $(D_1,D_2)$ be a bipartition of $D$.
We may assume that $U_1 \cap D_2 = \emptyset$. Let us say that $d \in D$
is {\em deficient} if either $d \in D_1$ and $2 \not \in L(d)$, or
$d \in D_2$ and $1 \not \in L(d)$. Then there is a deficient vertex in $D$.
Let $P_1$ be a shortest path in $D$ from some vertex $u_1 \in U_1$ to a 
deficient vertex $w_1$ of $D$. Then $u_1,w_1,P_1$ satisfy the conditions of
\ref{clm:Kempe}.2.   It follows from symmetry that there exist 
$u_2, w_2 \in V_2$ and a path $P_2$ from $u_2$ to $w_2$ satisfying the
condition of \ref{clm:Kempe}.2. Since $V_1$ is anticomplete to $V_2$,
it follows that $V(P_1)$ is anticomplete to $V(P_2)$ and \ref{clm:Kempe}.2 
holds.  \end{proof}

We also make use of the following result.		

\begin{theorem}[Camby and Schaudt~\cite{CS14}]\label{structure}
For all $ t \geq 3 $, any connected $ P_t $-free graph $ H $ contains a connected dominating set whose induced subgraph is either $P_{t-2}$-free, or isomorphic to $P_{t-2}$.
\end{theorem}

Our strategy from now on is as follows. 
Let $(G,L)$ be a minimal list-obstruction, where $G$ is $P_6$-free.
At every step, we find a subgraph $R$ of $G$, and consider all  possible 
partial precolorings of $R$.
For each precoloring, we update three times with respect to $R$, and possibly
modify the lists further, to produce
a minimal list-obstruction $(G',L')$ where $G'$ is an induced subgraph of $G$, 
and $|L(v)| \leq 2 $ for every $v \in V(G')$, and such that
$|V(G)|$ is bounded from above by a function of $|V(G')|$ (the function does not depend on $G$, it works for all $P_6$-free graphs $G$).
Since by Lemma~\ref{lem:size2} $V(G')$ has bounded size,
it follows from Lemma~\ref{precolor} and Lemma~\ref{update4}  that 
$V(G) \setminus R$ has bounded size. 
Now we use the minimality of $G$ and the internal structure of $R$ to show 
that $R$ also has a bounded number of vertices, and so $|V(G)|$ is bounded.
Next we present the details of the proof.

\begin{claim}\label{c5free}
There exists an integer $C$ such that the following holds. 
Let  $(G,L)$ is a minimal list-obstruction, where $G$ is $C_5$-free.
Then  $|V(G)| \leq C$.
\end{claim}
\begin{proof}
We may assume that $|V(G)|>8$, and therefore there is no $K_4$ in $G$.
Let $H$ be as in Theorem~\ref{structure}. Then $H$ is either $P_4$ or $P_4$-free.  If $|V(H)| \leq 4$, the result follows from Lemma~\ref{P_4}, so we may assume that $H$ is $P_4$-free. Now by a result of \cite{Scheische} $V(H)=A \cup B$,
where $A$ is complete to $B$, and both $A$ and $B$ are non-empty.
Choose $a \in A, b \in B$. Define the set $S_0$ as follows. If there is a vertex $c$ complete to $\{a,b\}$, let $S_0=\{a,b,c\}$; if no such $c$ exists, let
$S_0=\{a,b\}$. Then no vertex of $V(G)$ is complete to $S_0$. Let $X_0$ be the set of vertices of $G$ with a neighbor in $S_0$, and let $Y_0=V(G) \setminus (S_0 \cup X_0)$. By Theorem~\ref{structure} every vertex of $Y_0$ has a neighbor in $X_0$. By Lemma~\ref{precolor} and Lemma~\ref{update4} we may assume that the vertices of $S_0$ are precolored; let $L_1$ be the list system obtained from $L$ by updating three times with respect to $S_0$. Then $|L_1(x)| \leq 2$
for every $v \in X_0$. Let $S_1'=\{v \in V(G) : |L(v_1)|=1\}$, and
let $S_1$ be the connected component of $S_1'$ such that $S_0 \subseteq S_1$.
Let  $X_1$ be the set of vertices of $V(G) \setminus S_1$ with a neighbor in $S_1$,  and
let $Y_1=V(G) \setminus (S_1 \cup X_1)$. Then $X_1 \cap S_1' = \emptyset$, and
every vertex of $Y_1$ has a 
neighbor in $X_1$. Since $S_0 \subseteq S_1$, no vertex of $G$ is complete to 
$S_1$. For $i,j \in \{1,2,3\}$ let $X_{ij}^1=\{x \in X_1: L_1(x_1)=\{i,j\}\}$.

We now construct the sets $S_2,X_2,Y_2$.
For every $i,j \in \{1,2,3\}$ let
$U_{i,j}$ be defined as follows. If there is a vertex $u \in X^1_{ij}$  
such that there exist $y,z,w \in Y_1$ where $\{y,z,w\}$ is a clique and
$u$ has exactly one neighbor in $\{y,z,w\}$, choose such a vertex $u$
with $N(u) \cap Y_1$ maximal, and let $U_{i,j}=\{u\}$. If no such vertex $u$
exists, let $U_{i,j}=\emptyset$. Let ${X^1}'_{i,j}$ be the set of vertices of
$X^1_{i,j}$ that are anticomplete to $U_{i,j}$.
Let $Y_1'$ be the set of vertices in $Y_1$
that are anticomplete to $U_{1,2} \cup U_{1,3} \cup U_{2,3}$.

Next we define $V_{i,j}$ for every $i,j \in \{1,2,3\}$. If there is a vertex 
$v \in {X^1}'_{ij}$  
such that there exist adjacent $y,z\in Y_1'$ 
where  $v$ has exactly one neighbor in $\{y,z\}$, choose such a vertex $v$
with $N(v) \cap Y_1'$ maximal, and let $V_{i,j}=\{v\}$. If no such vertex $v$
exists, let $V_{i,j}=\emptyset$. 

Now let $S_2 = S_1 \cup \bigcup_{i,j \in \{1,2,3\}}(U_{ij} \cup V_{ij})$.
Precolor the vertices of $S_2$.
Observe that $S_2$ is connected. Let $L_3$ be the list system obtained from 
$L_1$ by updating three times. By~Lemma~\ref{precolor} and Lemma~\ref{update4}
we may assume that $(G,L_3)$ is a minimal list-obstruction.
Let $S_3'=\{v \in V(G) : |L(v_1)|=1\}$, and let $S_3$ be the connected component
of $S_3'$ such that $S_2 \subseteq S_3$.
Let $X_3$ be the set of vertices of $V(G) \setminus S_3$ with a neighbor in $S_3$,  and
let $Y_3=V(G) \setminus (S_3 \cup X_3)$. Then every vertex of $Y_3$ has a 
neighbor in $X_3$. Since $S_1 \subseteq S_3$, no vertex of $G$ is complete to 
$S_3$, and for every $x \in X_3$, $|L(x)|=2$.
For $i,j \in \{1,2,3\}$ let $X_{ij}^3=\{x \in X_1: L_1(x_1)=\{i,j\}\}$.
Then $X^3_{ij}$ is anticomplete to $U_{ij} \cup V_{ij}$.

\begin{equation}
\label{notmixedY}
\begin{minipage}[c]{0.8 \textwidth} \em
No vertex of $X_3$ is mixed on an edge of $Y_3$.
\end{minipage}
\end{equation}

Suppose that there exist $x \in X_3$ and $y,z \in Y_3$ such that
$y$ is adjacent to $z$, and $x$ is adjacent to $y$ and not to $z$.
We may assume that $x \in X^3_{12}$. Then $x \in X^1_{12} \cup Y_1$,
and $y,z \in Y_1$.
Suppose first that $x \in Y_1$. Then there is $s_3 \in S_3 \setminus S_1$ 
such that $x$ is adjacent to $s_3$.  Since $y,z \in Y_3$, it follows that
$s_3$ is anticomplete to $\{y,z\}$. Since $s_3 \in S_3 \setminus S_1$,
there is a path $P$ from $s_3$ to some vertex $s_3' \in S_3$, such 
$s_3'$ has a neighbor in $S_1$, and $V(P) \setminus \{s_3'\}$ is anticomplete to
$S_1$. Then $s_3'$ is not complete to $S_1$, and since $S_1$ is connected,
it follows from Claim~\ref{mixed} that there exist $s_1,s_1' \in S_1$ such that
$s_3'$-$s_1$-$s_1'$ is a path. But now $z$-$y$-$x$-$s_3$-$P$-$s_3'$-$s_1$-$s_1'$ is a 
$P_6$, a contradiction. This proves that $x \not \in Y_1$, and therefore
$x \in X^1_{12}$.

Since $y,z \in Y_1 \cap Y_3$, it follows that $V_{12} \neq \emptyset$. Let
$v$ be the unique element of $V_{12}$.  Then $v$  is non-adjacent to $x,y,z$.  
Since $x$ is adjacent to $y$, and $v$ is non-adjacent to $y$, it follows
from the choice of $v$ that there exists $y' \in Y_1$ such that
$y'$ is adjacent to $v$ and not to $x$. Since $v$ has a neighbor
in $S_1$, and $v$ is not complete to $S_1$, and since $S_1$ is connected,
Claim~\ref{mixed} implies that there exist $s,s' \in S_1$ such that 
$v$-$s$-$s'$ is  a path. 
Since neither of $s'$-$s$-$v$-$y'$-$y$-$z$ and $s'$-$s$-$v$-$y'$-$z$-$y$ is a $P_6$,
it follows that $y'$ is either complete or anticomplete to $\{y,z\}$.
Suppose first that $y'$ is anticomplete to $\{y,z\}$.
Let $P$ be a path from $v$ to $x$ with interior in $S_1$.
Then $|V(P)| \geq 3$. Now $y'$-$v$-$P$-$x$-$y$-$z$ is a $P_6$, a contradiction.
This proves that $y'$ is complete to $\{y,z\}$.

Now $\{y',y,z\}$ is a clique in $Y_1$, and $v$ has exactly one neighbor in it.
This implies that $U_{12} \neq \emptyset$. Let $u$ be the unique element
of $U_{12}$. Then $u$ is anticomplete to 
$\{v,y',y,z\}$. It follows from the maximality of $u$ that there exists
$y'' \in Y_1$ such that $y''$ is adjacent to $u$ and non-adjacent to $v$.
Since $u$ has a neighbor
in $S_1$, and $u$ is not complete to $S_1$, and since $S_1$ is connected,
Claim~\ref{mixed} implies that there exist $t,t' \in S_1$ such that 
$u-t-t'$ is  a path. 
Suppose $y''$ has a neighbor in $\{y',y,z\}$.
Since $G \neq K_4$, there exist $q,q' \in \{y',y,z\}$ such that
$y''$ is adjacent to $q$ and not to $q'$.
But now $t'$-$t$-$u$-$y''$-$q$-$q'$ is a $P_6$, a contradiction.
This proves that  $y''$  anticomplete to $\{y',y,z\}$.
Let $P$ be a path from $u$ to $v$ with interior in $S_1$.
Then $|V(P)| \geq 3$. Now $y''$-$v$-$P$-$u$-$y'$-$y$ is a $P_6$, a contradiction.
This proves \eqref{notmixedY}.

\bigskip
For $i,j \in \{1,2,3\}$ let $X_{ij}$ be the set of vertices in $X^3_{ij}$ with a neighbor in $Y_3$.

\begin{equation}
\label{completeattachments}
\begin{minipage}[c]{0.8 \textwidth} \em
The sets $X_{12},X_{13},X_{23}$ are pairwise complete to each other.
\end{minipage}
\end{equation}

Suppose $x_1 \in X_{12}$ is non-adjacent to $x_2 \in X_{13}$. Since $S_3$ is connected and both
$x_1,x_2$ have neighbors in $S_3$, there is a path $P$ from $x_1$ to $x_2$
with $V(P) \setminus \{x_1,x_2\} \subseteq S_3$.
Since $L_3(x_1)=\{1,2\}$ and $L_3(x_2)=\{1,3\}$, it follows that no vertex
of $S_3$ is adjacent to both $x_1$ and $x_2$, and so $|V(P)| \geq 4$.
Let $y_i \in Y_3$ be adjacent to $x_i$.
If $|V(P)|>4$ or $y_1 \neq y_2$, then $y_1$-$x_1$-$P$-$x_2$-$y_2$ contains a path
with at least six vertices, a contradiction. So
$|V(P)|=4$ and $y_1=y_2$. But now $y_1$-$x_1$-$P$-$x_2$-$y_1$ is a $C_5$ in $G$,
again a contradiction. This proves \eqref{completeattachments}.

\bigskip

Let $D_1, \ldots, D_t$ be the components of $Y_3$ that have size at least two.
Moreover, let $Y'=\bigcup_{i=1}^t D_i$.

\begin{equation}
\label{stableY}
\begin{minipage}[c]{0.8 \textwidth} \em
There is an induced subgraph $F$ of $G$ with 
$V(G) \setminus Y' \subseteq V(F)$ and a list system $L'$ such that
\begin{itemize}  
\item $|L'(v)| \leq 2$ for every $v \in V(F) \cap Y'$,
\item $L'(v)=L(v)$ for every $v \in V(F) \setminus Y'$, and
\item $(F,L')$ is a minimal list-obstruction, and $|V(G)|$ depends only on 
$|V(F)|$. 
\end{itemize}
\end{minipage}
\end{equation}
For $in \in \{1, \ldots, t\}$ let $U_i$ be the set of vertices of $X_3$ with a 
neighbor in $D_i$.
It follows from Claim~\ref{mixed} and \eqref{notmixedY} that
$U_i$ is complete to $D_i$.
Since each $D_i$ contains an edge, \eqref{completeattachments} implies
that  each $U_i$ is a subset of one of $X^3_{1,3}, X^3_{1,3},X^3_{2,3}$. 
Therefore there exists $c_i \in \{1,2,3\}$ such that for every $u \in U_i$,
$c_i \not \in L(u)$. 
Now \eqref{stableY} follows from Claim~\ref{reducecomponents}.
This proves \eqref{stableY}.

\bigskip

Let $(F,L')$ be as in \eqref{stableY}. 
Since our goal is to prove that $(G,L_3)$ induces a minimal obstruction of
bounded size, it is enough to show that $|V(F)|$ has bounded size  (where the 
bound is independent of $G$). 
Therefore we may assume that $G=F$ and $L_3=L'$, and in particular that 
$|L_3(v)| \leq 2$ for every $v \in Y'$. 

Let $Y=Y_3\setminus Y'$. Then the set $Y$ is stable, 
$N(y) \subseteq X_{12} \cup X_{13} \cup X_{23}$ for every $y \in Y$, and  for
$v \in V(G)$, if $|L_3(v)|=3$, then $v \in Y$. Moreover, if
$y \in Y$ has $|L_3(y)|=3$ and $N(y) \subseteq X_{ij}$ for some $i,j \in \{1,2,3\}$, then $(G,L_3)$ is colorable if and only if $(G\setminus y,L_3)$ is colorable,
contrary to the fact that $(G,L_3)$ is a minimal list obstruction. 
Thus for every $y \in Y$
with $|L_3(y)|=3$, $N(y)$ meets at least two of $X_{12},X_{13},X_{23}$.
By \eqref{completeattachments} it follows that the sets $X_{12},X_{13},X_{23}$
are pairwise complete to each other, and therefore no $v \in Y$ has neighbors
in all three of $X_{12},X_{13},X_{23}$.

Next we define a refinement $\mathcal{L}$ of $L_3$. 
\begin{itemize}
\item If exactly one of $X_{12},X_{13},X_{23}$ is non-empty, then $\mathcal{L}=\{L_3\}$.

\item If at least two of  $X_{12},X_{13},X_{23}$ are non-empty and some $X_{ij}$ 
contains two adjacent vertices $a,b$,
let $L'$ be the list obtained by precoloring $\{a,b\}$ and  updating three
times, and let $\mathcal{L}=\{L'\}$.

\item  Now assume that at least two of $X_{ij}$ are non-empty, and each
of $X_{ij}$ is a stable set. Observe that in this case, in every coloring of 
$G$ at least one of $X_{ij}$ is monochromatic. For all $i,j$ such that $X_{ij} \neq \emptyset$ and for all  $k \in \{i,j\}$  add to $\mathcal L$  the list system $L_{ij}^i$, where $L_{ij}^i(x)=\{i\}$ for all  $x \in X_{ij}$ and $L_{ij}^i(v)=L_3(v)$ for all $v \in V(G) \setminus X_{ij}$, and we updated three times with respect to $X_{ij}$.
\end{itemize}

Now $\mathcal L$ is a refinement of $L$ and satisfies the hypotheses of 
Lemma~\ref{precolor}. We claim that 
for every $L' \in \mathcal{L}$ there exist $i,j \in \{1,2,3\}$ such that
after the first step of updating  
 $|L'(x)|=1$
for all 
$x \in (X_{12} \cup X_{13} \cup X_{23}) \setminus X_{ij}$,

In view of \eqref{completeattachments},  this is clear if some  $X_{ij}$ is 
not stable or if only one of the sets $X_{12},X_{13},X_{23}$ is non-empty. So we 
may assume that all $X_{ij}$
are stable, and at least two are non-empty. Let $L'=L_{ij}^i$.
Then $L'(x)=\{i\}$ for all $x \in X_{ij}$. 
Let  $k \in \{1,2,3\} \setminus \{i,j\}$, then by \eqref{completeattachments}
after the first step of
updating $L'(x)=\{k\}$ for every $x \in X_{ik}$. Thus after the first step of updating only
one of the sets $X_{ij}$ may contain vertices with lists of size two.

Since every $y \in Y$ with $|L_3(y)|=3$ has neighbors in at least two of
$X_{12},X_{13},X_{23}$, it follows that after the second step of updating 
all vertices of $Y$ have lists of size at most two, and so 
for all $L' \in \mathcal{L}$ we have that $|L'(v)| \leq 2$ for all $v \in V(G)$.
By Lemma~\ref{lem:size2}, each of $(G,L')$ induces a minimal obstruction
with at most 100 vertices. Applying the Lemma~\ref{precolor} and 
Lemma~\ref{update4}, 
we deduce that $|V(G) \setminus (X_{12} \cup X_{13} \cup X_{23})|$
depends only on the sizes of the minimal obstructions
induced by $(G,L')$, and therefore does not depend on $G$. Now, since each of 
$X_{ij}$ is a stable set, 
Claim~\ref{stable} implies that $|V(G)|$ is bounded, and Claim~\ref{c5free} 
follows.
\end{proof}

In the remainder of this proof we deal with minimal list-obstructions 
$(G,L)$ containing a $C_5$, by taking advantage of the structure that it 
imposes. Let $C$ be a $C_5$ in $G$,  
say $C=c_1$-$c_2$-$c_3$-$c_4$-$c_5$-$c_1$. 
Let $X(C)$ be the set of vertices of $V(G) \setminus V(C)$ that have a neighbor in $C$,
let $Y(C)$ be the set of vertices of $V(G) \setminus (V(C) \cup X(C))$ that have a neighbor in $X$, 
and let $Z(C)=V(G) \setminus (V(C) \cup X(C) \cup Y(C))$.

\begin{claim}\label{notmixed}
Assume that $|V(G)| \ge 7$. 
Then the following assertions hold.
\begin{enumerate}
\item For every
$x \in X(C)$ there exist indices $i,j \in \{1,\ldots,5\}$ such that $x$-$c_i$-$c_j$ is an induced path. 
\item No vertex of $Y(C)$ is mixed on an edge of $G|Z(C)$.
\item If $v \in X(C)$ is mixed on an edge of $G|(Y(C) \cup Z(C))$, then the set of neighbors of $v$ in $C$ is not contained in a $3$-vertex path of $C$.
\item If $v \in X(C)$ has a neighbor in $Y(C)$, then the set of neighbors of $v$ in $C$ is not contained in a $2$-vertex path of $C$.
\item If $z \in Z(C)$ and $u,t \in N(z) \cap Y(C)$ are non-adjacent, then no vertex
of $X(C)$ is mixed on $\{u,t\}$.
\item Let $D$ be a component of $Z(C)$ with $|D|=1$, and let $N$ be the set of vertices of $Y(C)$ with a neighbor in $D$. Then either $N$ is anticonnected,
or $N=U \cup W$ where $U$ and $W$ are stable sets, and  $U$ is complete to $W$.
\item Let $D$  be a component of $Z(C)$ with $|D|>1$, and let $N$ be the set of vertices of $Y(C)$ with a neighbor in $D$.   Then $D$ is bipartite, and
$N$ is a stable set complete to $D$.
\end{enumerate}
\end{claim}
\begin{proof} 
Since $|V(G)| \ge 7$, no vertex is complete to $V(C)$, as that would lead to a list-obstruction on $6$ vertices.
Thus, the first assertion follows from the fact that $G$ is connected and
Claim~\ref{mixed}.

Next we prove the second assertion.
Suppose that $u\in Y(C)$ is mixed on the edge $st$ 
with $s,t \in Z(C) $, namely $u$ is adjacent to $s$ and not to $t$.   
Let $b\in N(u)\cap X$ and $i,j \in \{1,\ldots,5\}$ be such that $b$-$c_i$-$c_j$ is an induced path (as in Claim~\ref{notmixed}.1). 
Then $t$-$s$-$u$-$b$-$c_i$-$c_j$ is a $P_6$, a contradiction.

To see the third assertion, suppose that $x \in X$ is adjacent to $t \in Y$ and non-adjacent to $s \in Y \cup Z$, where $t$ is adjacent to $s$, and suppose that $N(x) \cap V(C) \subseteq \{c_1,c_2,c_3\}$. 
We may assume that $x$ is adjacent to $c_3$. 
Then $c_5$-$c_4$-$c_3$-$x$-$t$-$s$ is a $P_6$ in $G$, a contradiction.

To prove the fourth statement, we may assume that $x \in X$ is adjacent to $c_1$ and to $y \in Y$, and non-adjacent to $c_2$, $c_3$ and $c_4$.
Now $y$-$x$-$c_1$-$c_2$-$c_3$-$c_4$ is a $P_6$ in $G$, a contradiction.

To prove the fifth statement, suppose that $w \in X(C)$ is adjacent to $u$ and 
non-adjacent to $t$.
By Claim~\ref{notmixed}.1 there exists $i,j$ such that $w$-$c_i$-$c_j$ is an induced path.
Then $t$-$z$-$u$-$w$-$c_i$-$c_{i+1}$ is a $P_6$, a contradiction.

Next let  $D=\{v\}$ be a component of $Z$, then $N(v) \subseteq Y$,
and Claim~\ref{notmixed}.6 follows immediately from the fact that there is no $K_4$ in $G$.

%So, we may assume that $N(v)$ is disconnected, and therefore anticonnected.
%Let $w \in X$ have a neighbor in $N(v)$. 
%Since $w$ does not dominate $v$, it follows that $w$ is mixed on $N(v)$, and since $N(v)$ is anticonnected there exist $t,u \in N(v)$ such that $w$ is adjacent to $u$ and not to $t$, and $u$ is non-adjacent to $t$, contrary to ref{notmixed}.5. This proves \ref{notmixed}.6.
	 
Finally let $D$ be a component of $Z(C)$ with $|D|>1$.
By Claim~\ref{notmixed}.2 $N$ is complete to $D$. Since there is no $K_4$
in $G$, it follows that $D$ is bipartite and $N$  is a  stable set. 
This proves Claim~\ref{notmixed}.7.
\end{proof}

By Lemma~\ref{precolor} and Lemma~\ref{update4} we may assume that in $(G,L)$ 
the vertices of $C$ are precolored, and that we have updated three times with respect to $V(C)$. We may assume that $|V(G)|>8$.

\begin{claim}\label{noZ}
There is an induced subgraph $F$ of $G$ with 
$V(G) \setminus Z(C) \subseteq V(F)$ and a list system $L'$ such that
\begin{itemize}  
\item $|L'(v)| \leq 2$ for every $v \in V(F) \cap Z(C)$,
\item $L'(v)=L(v)$ for every $v \in V(F) \setminus Z(C)$, and
\item $(F,L')$ is a minimal list-obstruction, and $|V(G)|$ depends only on
the size of $|V(F)|$.
\end{itemize}
\end{claim}

\begin{proof}
We write $X=X(C)$, $Y=Y(C)$ and $Z=Z(C)$. 
Let $D_1, \ldots, D_t$ be components of $Z$ with $|D_i| \geq 2$.
Write $D= \bigcup_{i=1}^t D_i$.
For every $i$ let $U_i$ be the set of vertices of $Y$ with a neighbor in $D_i$.
By Claim~\ref{notmixed}.7 for every $i$, $U_i$ is  a stable set complete to 
$D_i$. 
By Claim~\ref{notmixed}.3 every $x \in X$ with a neighbor in $U_i$ has 
neighbors of two different colors in $V(C)$, and so every such $x$ has
list of size one after the first step of updating. Now  
by Claim~\ref{notmixed}.5  and since we have updated three times,  it 
follows that for every $i$  there exists $c_i \in \{1,2,3\}$ such that for 
every $u \in U_i$, $c_i \not \in L(u)$.  
By Claim~\ref{reducecomponents} there exist 
an induced subgraph $F$ of $G$ with 
$V(G) \setminus Z(C) \subseteq V(F)$ and a list system $L'$  such that
\begin{itemize}  
\item $|L'(v)| \leq 2$ for every $v \in V(F) \cap D$,
\item $L'(v)=L(v)$ for every $v \in V(F) \setminus Z(C)$, and
\item $(F,L')$ is a minimal list-obstruction, and $|V(G)|$ depends only on
the size of $|V(F)|$.
\end{itemize}

It remains to show that that $L'(v) \leq 2$ for every 
$v \in V(F) \cap Z$.
Suppose there is $v \in V(F) \cap Z$ with $|L(v)|=3$. 
Let $D$ be the component of $Z$ containing $v$. Then  $D=\{v\}$.
If $N(v)$ is anticonnected, then by Claim~\ref{notmixed}.5 every $x \in X$ with a
neighbor in $N(v)$ dominates $v$, contrary to the fact that $(F,L')$ is a minimal list-obstruction. So 
by Claim~\ref{notmixed}.6 $N(v)=U \cup W$, both $U$ and $W$ are stable sets, and 
$U$ is complete to $W$.

We now apply Claim~\ref{clm:Kempe}. We may assume that if Claim~\ref{clm:Kempe}.1
holds then $p_1,p_4 \in U$, and if Claim~\ref{clm:Kempe}.2 holds,
then $u_1,u_2 \in U$. We show that in both cases some vertex
$t \in V(G) \setminus U$   is  mixed on  $U$. If Claim~\ref{clm:Kempe}.1
holds, we can take $t=p_1$, so we may assume that Claim~\ref{clm:Kempe}.2
holds.  We may assume that $u_1=w_1$ and $u_2=w_2$, for otherwise
some vertex of $V(P_1) \cup V(P_2)$  is mixed on $U$.
By Claim~\ref{notmixed}.3 and Claim~\ref{notmixed}.5, and since we 
have updated,  
it  follows that there exists $i \in \{1,2,3\}$ such that for every $u \in U$,
$i \not \in L(u)$.
Since $|L(v)|=3$, and we have updated three times, it follows that
after the second step of updating all $u \in U$ have exactly the same list, 
and this list has size two.
Since $u_1=w_1$ and $u_2=w_2$, it follows that the lists of $u_1$ and $u_2$
changed and became different in the third step of updating, and so some vertex 
$V(G) \setminus U$ is mixed on $U$, as required. This proves the claim.
Let $t$ be  a vertex of $V(G) \setminus U$ that is mixed on $U$.
By Claim~\ref{notmixed}.5, it  follows that $t \in Y \cup Z$.

First we show that if $y \in Y \setminus (U \cup W)$ has a neighbor
$u \in U$, and $x \in X$ is adjacent to $y$, then $x$ is complete to 
$U$. Suppose not, let $i$ be such that
$x-c_i-c_{i+1}$ is a path (such $i$ exists by Claim~\ref{notmixed}.1).
By Claim~\ref{notmixed}.5, $x$ is anticomplete to $U$.
Then $v$-$u$-$y$-$x$-$c_i$-$c_{i+1}$ is a $P_6$, a contradiction. This proves the claim.

Now we claim that $Y \setminus (U \cup W)$ is anticomplete to $U \cup W$.
Suppose $y \in Y$ has a neighbor $u \in U$, and let $x \in X$ be adjacent
to $y$. Then $x$ is complete to $U$. Since $x$ does not
dominate $v$, it follows that $x$ has a non-neighbor $w \in W$, and 
again by the previous claim, $y$ is anticomplete to $W$. 
Let $x_1 \in X$ be adjacent to $w$. By Claim~\ref{notmixed}.5
$x_1$ is complete to $W$. Since $x_1$ does not dominate $v$, it
follows that $x_1$ has a non-neighbor in $U$, and so by Claim~\ref{notmixed}.5
$x_1$ is anticomplete to $U$. By the previous claim,
$x_1$ is non-adjacent to $y$. Let $i$ be such that
$x_1$-$c_i$-$c_{i+1}$ is a path (such $i$ exists by Claim~\ref{notmixed}.1).
Now $c_{i+1}$-$c_i$-$x_1$-$w$-$u$-$y$ is a $P_6$, a contradiction. This proves the claim,
and in particular we deduce that $t \in Z$.

Since $t$ is mixed on $U$, there exists an edge 
$a,b$ with one end in $U$ and the other in $W$, such that
$t$ is adjacent to $b$ and not to $a$. Let $x \in X$ be adjacent to 
$a$. Since $x$ does not dominate $v$, we deduce that $x$ is not complete to 
$U \cup W$, and so by Claim~\ref{notmixed}.5 $x$ is non-adjacent to $b$.
Let $i$ be such that $x$-$c_i$-$c_{i+1}$ is a path (as in \ref{notmixed}.1).
Now $t$-$b$-$a$-$x$-$c_i$-$c_{i+1}$ is a $P_6$, a contradiction.
This proves Claim~\ref{noZ}.
\end{proof}

Let $(F,L')$ be as in Claim~\ref{noZ}. 
Since our goal is to prove that $|V(G)|$ is bounded, it is enough to prove that
$|V(F)|$ is bounded, and so  we may assume that $G=F$, $L=L'$, and in 
particular $|L(v)| \leq 2$ for every $v \in Z(C)$. 

\begin{claim}\label{compsW}
Assume that in the precoloring of $C$ $c_2$ and $c_5$ receive the same color, 
say $j$.
Let $A=\{a \in X(C) : N(a) \cap V(C)=\{c_2,c_5\}\}$ and $W=\{y \in Y : N(y) \cap X \subseteq A \}$. 
Let $D$ be a component of $W$ such that there exists a vertex with list of size $3$ in $D$, and let $N$ be the set of vertices of $A$ with a neighbor in $D$.
Then $D$ is complete to $N$, and either 
\begin{itemize}
	\item $D$ is anticomplete to $V(G) \setminus (D \cup N)$, or
	\item there exists vertices $d\in D$ and $v\in N(d)$ such that precoloring $d$, $v$ with distinct colors and updating with respect to the set $\{d,v\}$ three times reduces the list size of all vertices in $W$ to at most two.
\end{itemize}
\end{claim}
\begin{proof}
By Claim~\ref{notmixed}.3 no vertex of $X(C)$ is mixed on $D$, and so $N$ is complete  to $D$. Also by Claim~\ref{notmixed}.3 $W$ is anticomplete to $Z(C)$.

Let $d \in D$, and let $v \in N(d) \setminus (N \cup D)$. Since $D \subseteq W$, it follows that $v \in Y(C)$. By Claim~\ref{notmixed}.3, $N(v) \cap A=N$.  Let $x \in N(v) \cap (X(C) \setminus A)$. Then 
$N(x)\cap V(C)$ are not contained in a $3$-vertex 
path of $C$, and therefore  $|L(x)|=1$. 

First suppose that   $N(x)\cap \{c_2,c_5\}\neq \emptyset$. Then 
$j \not \in L(x)$.  We precolor $\{v,d\}$ and update with respect to the set 
$\{v,d\}$ three times. We may assume that the precoloring of $G|(V(C) \cup \{v,d\})$ is proper.
Since $\{v,d\}$ is complete to $N$ and not both $v,d$
are precolored $j$, it follows that $|L(n)|=1$ for every $n \in N$, and 
$|L(u)| \leq 2$ for every $u \in W$ such that $u$ has a neighbor $N$.
Suppose there is  $t \in W$ with 
$|L(t)|=3$. Then $t$ is anticomplete to $\{v,d\}$.  
Since $t \in W$, there exists $s \in A$ adjacent to $t$, and so 
$s$ is not complete to $\{v,d\}$.  Since $s \in A$, it 
follows from Claim~\ref{notmixed}.3 that $s$ is not mixed on the edge $vd$, and so $s$ is 
anticomplete to $\{v,d\}$. Since $|L(t)|=3$, it  follows that 
$L(s)=\{1,2,3\} \setminus \{j\}$, and so $s$ is non-adjacent
to $x$ (since we have updated three times with respect to $V(C)$).
Assume by symmetry that $c_2$ is adjacent to $x$, then $t$-$s$-$c_2$-$x$-$v$-$d$ is a
$P_6$, a contradiction. 

Therefore we may assume that  $N(x) \cap \{c_2,c_5\}= \emptyset$, and so
$N(x)=\{c_1,c_3,c_4\}$. It follows that $L(x)=\{j\}$, and consequently 
$L(v)\subseteq \{1,2,3\}\setminus \{j\}$. 
If $D=\{d\}$, then $|L(d)|=3$;
but $L(u)\subseteq \{1,2,3\}\setminus \{j\}$ for all $u\in N(d)$, which contradicts the fact that $(G,L)$ is a minimal list obstruction.
Therefore we may assume there exists $d'\in N(d)\cap D$. 
Since $G$ is not a $K_4$, $d'$ is not adjacent to $v$. 
But now $c_5$-$c_1$-$x$-$v$-$d$-$d'$ is a $P_6$, a contradiction.
\end{proof}

\begin{claim}\label{C5clone}
Assume that there is a vertex $c_1' \in V(G)$ adjacent to $c_1,c_2,c_5$ and non-adjacent to $c_3,c_4$.   Then $|V(G)|$ is bounded from above 
(and the bound does not depend on $G$).

\end{claim}
\begin{proof}
By Lemma~\ref{precolor} and Lemma~\ref{update4}, we can precolor the vertices 
of $V(C) \cup \{c_1'\}$  and update with  respect to $V(C) \cup \{c_1'\}$ 
three times.
By symmetry, we may assume that $L(c_1)=\{1\}$, $L(c_1')=L(c_3)=\{2\}$, $L(c_2)=L(c_5)=\{3\}$ and $L(c_4)=\{1\}$.
Let $C'=c_1'$-$c_2$-$c_3$-$c_4$-$c_5$-$c_1'$.
We write $X=X(C)$, $X'=X(C')$, and define the sets $Y$, $Y'$, $Z$, and $Z'$ in a similar manner. We abuse notation an denote the list system thus obtained by $L$.  Recall that $(G,L)$ is a minimal list-obstruction.

Let $A$ be the set of all vertices $a \in X \cup X'$ for which $N(a) \cap \{c_1,c_1',c_2,c_3,c_4,c_5\} = \{c_2,c_5\}$. 
Let $W$ be the set of vertices  $y \in Y \cap Y'$ such that $N(y) \cap (X \cup X') \subseteq A$. 
Since we have updated $|L(x)| \leq 2$ for every $x \in X \cup X'$.
By Claim~\ref{noZ} applied to $C'$,
we may assume that $|L(z)| \leq 2$ for every 
$z \in Z \cup Z'$.
Thus if $|L(v)|=3$ then $v \in Y \cap Y'$, and an easy case analysis shows that $v \in W$.
By Lemma~\ref{lem:size2} we may assume that $W \neq \emptyset$.
Let $D_1, \ldots ,D_t$ be the components of $W$ that contain vertices with 
lists of size three.
Suppose first that $|D_i|=\{d\}$ for some $i$.
Then, letting $c$ be a coloring of $G \setminus d$, we observe that no vertex of $N(d)$ is colored $3$, and so we can get a coloring of $G$ by setting $c(d)=3$, a contradiction. This proves that $|D_i| \geq 2$ for every $i$.

Let $i \in \{1, \ldots, t\}$. Let $U_i$ be the set of vertices of $A$ with a neighbor in $D_i$. By  Claim~\ref{compsW} , $D_i$ is complete to $U_i$
and anticomplete to $V(G) \setminus (D_i \cup U_i)$. 
Since $U_i \subseteq A$,
it follows that $3 \not \in L(u)$ for every $u \in U_i$. Let $(F,L')$ be as in 
Claim~\ref{reducecomponents}. Since $|L'(v)| \leq 2$ for every $v \in V(F)$,
Lemma~\ref{lem:size2} implies that $|V(F)| \leq 100$. Since
$|V(G)|$ depends only on $|V(F)|$, Claim~\ref{C5clone} follows.
This completes the proof of Claim~\ref{C5clone}.
\end{proof}

We can now prove the following claim, which is the last step of our argument.
We may assume that $C$ is precolored in such a way that the precoloring is 
proper, and the set $\{c_2,c_4\}$ is monochromatic and the set $\{c_3,c_4\}$ 
is monochromatic.
\begin{claim}\label{c5}
$|V(G)|$ is bounded from above (and the bound does not depend on $G$).
\end{claim}
\begin{proof} 
We may assume that $L(c_1)=1$, $L(c_2)=L(c_4)=2$  and  $L(c_3)=L(c_5)=3$. 
Write $X=X(C), Y=Y(C)$ and $Z=Z(C)$.
Let $A'=\{v\in X:N(v)\cap C=\{c_2,c_4\}\}$ and $B'=\{v\in X:N(v)\cap C=\{c_3,c_5\}\}$.  

It follows from Claim~\ref{notmixed}.4 that after the first step of updating 
every $v\in X\setminus(A'\cup B')$ with  a neighbor in $Y$
has list of size one.
Let $ Y' $ be the set of vertices that have lists of size 3 after the third
step of updating.
Since $L(z) \leq 2$ for every $z \in Z$, it follows that  $Y'\subseteq Y$,
and $N(y) \cap X \subseteq A \cup B$ for every $y \in Y'$.

Let $ A,B $ be the subsets of $ A',B' $ respectively consisting of all vertices with a neighbor in $ Y' $.  Then after the second step of updating, the list 
of every vertex in $A$ is $ \{1,3\} $ and the list of every vertex in $B$ is 
$\{1,2\}$. If one of $A',B'$ is not a stable set, Claim~\ref{C5clone} completes 
the proof.  So, we may assume that each of $A',B'$ is a stable set.

Let $ H $ be the graph obtained from $ G|(A\cup B) $ by making each of $ A,B $ a clique.
Let $ C_1,\ldots,C_t $ be the anticomponents of $ H $ such that both $ A_i=C_i \cap A $ and $ B_i=C_i \cap B $ are nonempty. 
Let $A''=A \setminus \bigcup_{i=1}^tC_i$ and let $B''=B_i \setminus \bigcup_{i=1}^t C_i$.

\begin{equation}\label{notmixedanti}
\begin{minipage}[c]{0.8\textwidth}\em
Let $v \in Y'$. Then $N(v) \cap A$ is complete to $B' \setminus N(v)$, and
$N(v) \cap B$ is complete to $A' \setminus N(v)$. In particular, 
$A$ is complete to $B' \setminus B$, $B$ is complete to $A' \setminus A$,
and $v$ is not mixed on $C_i$ for any $i$.

\end{minipage}
\end{equation}
Suppose this is false.
By symmetry, we may assume there exists $w\in A$ non-adjacent to $k \in B'$ 
such that $v$ is adjacent to $w$ but not to $k$. 
Then $v$-$w$-$c_2$-$c_1$-$c_5$-$k$  is a $P_6$ in $G$, a contradiction.
This proves \eqref{notmixedanti}. 
\begin{equation}\label{NsubAB}
\begin{minipage}[c]{0.8\textwidth}\em
Suppose $v \in Y'$ is adjacent to $y \in V(G) \setminus (A \cup B \cup Y')$.
Then precoloring $y$ and $v$ and updating three times reduces the list size of 
all  vertices in $Y'$ to at most two.
\end{minipage}
\end{equation}
Since $v \in Y'$, it follows that  $N(v)\cap X\subseteq  A\cup B$, and 
therefore $y \not \in X$. It follows from Claim~\ref{notmixed}.3 that 
$N(v) \cap X$ is complete to $N(v) \setminus X$.

By Claim~\ref{compsW} $v$ has both a neighbor in $A$ and a neighbor in $B$.
We precolor $v$ and $y$ and update three times; denote the new list system by $L''$.   If $v$ and $y$ have the same color, or  one of $v,y$ is colored $1$, 
then $L''(u)=\emptyset$ for some vertex $u \in N(v) \cap (A \cup B)$, and 
\eqref{NsubAB} holds.
Thus
we may assume that one of $v,y$ is precolored $2$, and the other one $3$.
We claim that, after updating, $|L(x)| = 1$ for every $x \in X$.  Recall that even before we precolored
$v$ and $y$ we had that $|L(x)|=1$ for every $x \in X \setminus (A' \cup B')$.
Since $v$ and $y$ are colored $2,3$, and $\{v,y\}$ is complete
to $N(v) \cap X$, it follows that $|L(x)|=1$ for every $x \in N(v) \cap X$.
Since both $N(v)\cap A$ and $N(v)\cap B$ are nonempty, $L(x)=\{1\}$ for every $x \in N(v) \cap X$.
By \eqref{notmixedanti}, $N(v) \cap A$ is complete
to $B' \setminus N(v)$, and $N(v) \cap B$ is complete to $A' \setminus N(v)$. 
Since we have updated, $L(a)=\{3\}$ for every
$a \in A' \setminus N(v)$ and $L(b)=\{2\}$ for every $b \in B' \setminus N(v)$.
Consequently $|L(w)| \leq 2$ for every $w \in Y$. This proves~\ref{NsubAB}.

\bigskip
In view of \eqref{NsubAB}, Lemma~\ref{precolor} and Lemma~\ref{update4},  we may assume that
$|L(v)| \leq 2$ for every $v \in Y(C)$ for which
$N(v) \setminus (A \cup B) \neq \emptyset$.

\begin{equation}\label{adjCi}
\begin{minipage}[c]{0.8\textwidth}\em
Let $T=\{y \in Y : N(y) \subseteq A'' \cup B''\}$. 
There is collection $\mathcal{L}$ of list systems such that for every
$L' \in \mathcal{L}$
\begin{itemize}  
\item  $|L'(v)| \leq 2$ for every $v \in T$, and
\item $L'(v)=L(v)$ for every $v \in V(G) \setminus T$, 
\end{itemize}
For every $L' \in \mathcal{L}$, let $(G_{L'},L')$ be a minimal list
obstruction induced by $(G,L')$.
Then $|V(G)|$  depends only on $|\bigcup_{L' \in \mathcal{L}}V(G_{L'})|$.

\end{minipage}
\end{equation}

Let $y \in T \cap Y'$. 
First we show that $y$ has a neighbor in $A$ and a neighbor in $B$.
Suppose $N(y) \cap B=\emptyset$. 
Then, by the remark following  \eqref{NsubAB}, $N(y) \subseteq A$. 
But now a coloring of $G \setminus y$ can be extended to a coloring of $G$ by assigning color $2$ to $y$, contrary to the fact $(G,L)$ is a minimal obstruction.
This proves that $y$ has a neighbor in $A$ and a neighbor in $B$.
In particular both $A''$ and $B''$  are non-empty.  

Observe that in every coloring of $G$
either $A''$ or $B''$ is monochromatic (since they are complete to each other).
Let $\mathcal{L}$  be the following collection of list systems.
For each $i \in \bigcap_{a \in A''}L(a)$ we add to $\mathcal L$  the list system $L'$, where $L'(a)=\{i\}$ for all $a \in A''$ and $L'(v)=L(v)$ for all $v \in V(G) \setminus A''$; and we update three times with respect to $A''$.
Moreover, for each $j \in \bigcap_{b \in B''}L(b)$ we add  to $\mathcal L$ the list system $L'$, where $L'(b)=\{j\}$ for all $b \in B''$ and $L'(v)=L(v)$ for all $v \in V(G) \setminus B''$, and we update three times with respect to $B''$.

Now $\mathcal L$ is a refinement of $L$ and satisfies the hypotheses of Lemma~\ref{precolor} with $R=G|(A'' \cup B'')$.
Let $L' \in \mathcal{L}$. Since either
$|L(a)|=1$ for every $a \in A''$, or $|L(b)|=1$ for every $b \in B''$, and
since we have updated three times, we have that
$|L'(y)| \leq 2$ for every $y \in T$.
Let $(G_{L'},L')$ be a minimal list-obstruction induced by $(G,L')$.

By~Lemma~\ref{precolor} and Lemma~\ref{update4},
$$V(G)=A \cup \bigcup_{L' \in \mathcal{L}} V(G_{L'}).$$
Since $A$ is a stable set, Claim~\ref{stable} implies that $|A|$ only depends on
$|\bigcup_{L' \in \mathcal{L}} V(G_{L'})|$, and \eqref{adjCi} follows.
This proves \eqref{adjCi}.

\bigskip

Let $\mathcal{L}$ be as in \eqref{adjCi}.
Since our goal is to prove that $G$ has bounded size, it is enough to show that $(G,L')$ induces a minimal
obstruction of bounded size for every $L' \in \mathcal{L}$.
Therefore we may assume that for every $y \in Y'$ there exists an index $i$ 
such that $y$ is complete to $C_i$.

\begin{equation}\label{hpair}
\begin{minipage}[c]{0.8\textwidth}\em
Let $y_1 \in Y'$ and let $C_1 \subseteq N(y_1)$.
Then we may assume that no vertex of $V(G) \setminus C_1$ is mixed on $A_1$ (and similarly on $B_1$).
\end{minipage}
\end{equation}
Suppose $x \in V(G) \setminus C_1$ is mixed on $A_1$. 
Since $x$ is mixed on $C_1$, and $C_1$ is an anticomponent of $H$, there exist $a_1\in A_1$
and $b_1\in B_1$ such that $a_1b_1$ is a non-edge, and $x$ is mixed on this non-edge.
Let $a_1' \in A_1$ be such that $x$ is mixed on $\{a_1,a_1'\}$. 
By Lemma~\ref{precolor} and Lemma~\ref{update4} we can precolor $T=\{x,a_1,a_1',b_1,y_1\}$,
and update three times with respect to $T$. 
Let $Y''$ be the set of vertices with lists of size $3$ after updating. We 
claim that  $Y''=\emptyset$.
Suppose not and let $v\in Y''$. 
By the remark following \eqref{adjCi} there exists an index $i$ such that $v$ is complete to $C_i$. 
Then $i \neq 1$.
Since $v \in Y''$, $\{a_1,a_1'\}$ is complete to $B_i$ and $b_1$ is complete 
to $A_i$,
and we have updated three times  with respect to $T$, it 
follows that $L(a_1)=L(a_1')=\{3\}$ and  $L(b_1)=\{2\}$.
Since $x$ has a neighbor in $\{a_1,a_1'\}$ we may assume that
$L(x) \neq \{3\}$. 
%Let $a_i \in A_i$ and $b_i \in B_i$ be non-adjacent.

First consider the case that $x$ is adjacent to $a_1$ and not to $b_1$. Then
$x$ is non-adjacent to $a_1'$.
Choose $a_i\in A_i$.
Since $x$-$a_1$-$y_1$-$b_1$-$a_i$-$v$ is not a $P_6$ in $G$, it follows that $x$ is adjacent to $a_i$. Since $v \in Y''$, it follows that $L(x)=\{2\}$,
and therefore $x$ is anticomplete to $B_i$.  
Choose $b_i$ such that $a_ib_i$ is a non-edge, then $x$-$a_i$-$v$-$b_i$-$a'_1$-$y_1$ is a $P_6$ in $G$, a contradiction. 
Therefore $x$ is adjacent to $b_1$ and not to $a_1$. 
Since $x$-$b_1$-$y_1$-$a_1$-$b$-$v$ is not a $P_6$ in $G$ for any $b \in B_i$,
it follows that  $x$ is complete to $B_i$,
which is a contradiction since $L(x) \neq \{3\}$ and $v \in Y''$.
This proves \eqref{hpair}.
\begin{equation}\label{size1}
\begin{minipage}[c]{0.8\textwidth}\em
Let $v \in Y'$ and let $C_i \in N(v)$. Then we may assume $|A_i|=|B_i|=1$.
\end{minipage}
\end{equation}
Suppose this is false.
We may assume that $i=1$. 
By \eqref{hpair}, no vertex of $ G\setminus C_1 $ is mixed on $ A_1 $ and no vertex of $ G\setminus C_1 $ is mixed on $ B_1 $. 
Choose $ a_1 \in A_1 $ and $ b_1 \in B_1 $ such that $a_1b_1$ is an edge if possible.
Then $ (G\setminus (A_1\cup B_1)) \cup \{a_1,b_1\} $ is not $L$-colorable, since otherwise we can color $A_1$ in the color of $a_1$ and $ B_1 $ in the color of $b_1$. 
Since $(G,L)$ is a minimal list-obstruction,
\eqref{size1} follows.
\bigskip

Let $Y_1=\{y \in Y' : N(y) \subseteq (A \setminus A'') \cup (B \setminus B'')\}$,
and let $Y_2=Y' \setminus Y_1$. By~\eqref{size1} and since $(G,L)$
is a minimal list-obstruction, every $y\in Y_1$ is complete to more than one of $C_1, \ldots, C_t$.
We may assume that each of $C_1, \ldots, C_s$
 is complete to some vertex of $Y_1$, and $C_{s+1} \cup \ldots \cup C_t$
is anticomplete to $Y_1$.
Let $F$ be the graph with vertex set $V(F)=\{1, \ldots, s\}$ where
$i$ is adjacent to $j$ if and only if there is a vertex $y \in Y_1$ complete to $C_i \cup C_j$.
We will refer to the vertices of $F$ as $1, \ldots, s$ and $C_1, \ldots, C_s$
interchangeably.

Let $F_1, \ldots, F_k$ be the components of $F$, let 
$A(F_i)=\bigcup_{C_j \in F_i} A_j$, and let $B(F_i)=\bigcup_{C_j \in F_i}B_j$. 
Moreover, let
$Y(F_i)=\{y \in Y_1 : N(y) \subseteq A(F_i) \cup B(F_i)\}$.
\begin{equation}\label{monosets}
\begin{minipage}[c]{0.8\textwidth}\em
Let $i \in \{1, \ldots, k\}$ and let $T \subseteq V(G)$ be such that 
$A(F_i) \cup B(F_i) \cup Y_1 \subseteq T$. Then 
	for every $L$-coloring of $G|T$, both of the sets
$A(F_i)$ and $B(F_i)$ are monochromatic, and the color of $A(F_i)$ is different
from the color of $B(F_i)$. 
\end{minipage}
\end{equation}
Let $c$ be a coloring of $G|T$.
Let $y \in Y(F_i)$.
We may assume that $y$ is complete to $C_1$, and $C_1 \in F_i$.
Let $\alpha =c(A_1)$ and $\beta =c(B_1)$, where $c(A_i)$ and $c(B_i)$ denote the color given to the unique vertices in the sets $A_i$ and $B_i$ respectively.
Since $y$ is complete to at least two of $C_1, \ldots, C_s$, the sets $N(y) \cap A$ and $N(y) \cap B$ are monochromatic, and $\alpha \neq \beta$.
Pick any $t \in F_i$, and let $P$ be a shortest path in $F$ from $C_1$ to $t$.
Let $s$ be the neighbor of $t$ in $P$. 
We may assume that $s=C_2$ and $t=C_3$.
We proceed by induction and assume that $c(A_2)=\alpha$, and $c(B_2)=\beta$.
Since $s$ is adjacent to $t$ in $F$,  there is $y' \in Y_1$ such that $y'$ is complete to $C_2 \cup C_3$.
Then $c(y') \in \{1,2,3\} \setminus \{\alpha,\beta\}$. 
Moreover, $A_2$ is complete to $B_3$, and $A_3$ is complete to $B_2$, and so $c(A_3) \notin \{ c(y'), \beta\}$ and $c(B_3) \notin \{c(y'), \alpha\}$.
It follows that $c(A_3)=\alpha$ and $c(B_3)=\beta$, as required.
This proves \eqref{monosets}.
\bigskip

We now construct a new graph $G'$ where we replace each $F_i$ by a representative in $A$ and a representative in $B$, as follows.
Let $G'$ be the graph obtained from $G \setminus (C_1 \cup \ldots \cup C_s \cup Y_1)$ by adding $2s$ new vertices $a_1,\ldots,a_s,b_1, \ldots, b_s$, where
$$N_{G'}(a_i)=\{b_i\} \cup \bigcup_{a \in A(F_i)}(N_G(a) \cap V(G'))$$ 
and $$N_{G'}(b_i)=\{a_i\} \cup \bigcup_{b \in B(F_i) } (N_G(b) \cap V(G')),$$
for all $i \in \{1,\ldots,s\}$. 
Note that, in $G'$, the set $\{a_1,\ldots,a_s\}$ is complete to the set $\{b_1,\ldots, b_s\}$.
Let $L(a_i)=\{1,3\}$ and $L(b_i)=\{1,2\}$ for every $i$.
By repeated applications of Claim~\ref{biplemma}, we deduce that $G'$ is $P_6$-free.

Let $A^*=(A \setminus (A'' \cup A_1 \ldots \cup A_s)) \cup \{a_1, \ldots, a_s\}$ and $B^*=(B \setminus (B'' \cup B_1 \ldots \cup B_s)) \cup \{b_1, \ldots, b_s\}$.
Note that $A^*$ is complete to $B''$, and $B^*$ is complete to 
$A''$.

Let $R=G|(A^* \cup B^* \cup A'' \cup B'' )$.  

We may assume that $|A^*|\geq 2$, and define the list systems $L_1$, $L_2$, and $L_3$ as follows.
$$L_1(v)= \begin{cases} \{3\} &\text{ if }v \in A'' \\
\{2\} &\text{ if }v \in B''\\
L(v) &\text{ if }v \not\in A'' \cup B'' \end{cases}$$

$$L_2(v)= \begin{cases} \{3\} &\text{ if }v \in A^* \\
L(v) &\text{ if }v \not\in A^* \end{cases}$$

$$L_3(v)= \begin{cases} \{2\} &\text{ if }v \in B^* \\
L(v) &\text{ if }v \not\in B^* \end{cases}$$

Let $\mathcal{L}=\{L_1,L_2,L_3\}$.
It is clear that, for every $L$-coloring $c$ of $G'$, there exists a list system $L' \in \mathcal{L}$ such that $c$ is also an $L'$-coloring of $G'$.
Recall that by the remark following \eqref{adjCi} every vertex of $Y_2$ has a  neighbor in  $A^*$, a neighbor in $B^*$, and a neighbor in $A'' \cup B''$. Therefore, for every $L' \in \mathcal{L}$, every vertex in $Y_2$ is adjacent to some vertex $v$ with  $|L'(v)|=1$.
Now by Lemma~\ref{lem:size2}, Lemma~\ref{precolor}, and Lemma~\ref{update4}, 
for every $L' \in \mathcal{L}$, 
$G'$ contains an induced subgraph $G''$ such that $(G'',L)$ is not 
colorable, and $|V(G'') \setminus V(R)|\leq 3\cdot 36\cdot 100$.
We may assume that for every index $i$, $a_i\in G''$ or $b_i\in G''$, for otherwise we can just delete $F_i $ from $G$ contradicting the minimality of $(G,L)$. 

We claim that the subgraph induced by $G$ on the vertex set $$S=(V(G) \cap V(G'')) \cup Y_1 \cup \bigcup_{i=1}^s (A(F_i) \cup B(F_i))$$
 is not $L$-colorable. 
Suppose this is false and let $c$ be such a coloring.  
By \eqref{monosets}, for every $i \in \{1, \ldots, k\}$ the sets $A(F_i)$ and $B(F_i)$ are both monochromatic, and $c$ can be converted to a coloring of $G''$ by giving $a_i$ the unique color that appears in $A(F_i)$ and $b_i$ the unique color that appears in $B(F_i)$, a contradiction.  
Thus $V(G)=S$, and it is sufficient to show that $|Y_1 \cup \bigcup_{i=1}^s (A(F_i) \cup B(F_i))|$ has bounded size.
 To see this, let $T=S\setminus (A \cup B \cup Y_1)$, then 
$|T|< |V(G'') \setminus R| \leq   3\cdot 36\cdot 100$.

First we bound $s$.
Partition the set of pairs $\{(a_1,b_1), \ldots, (a_s,b_s)\}$ according to the 
adjacency of each $(a_i,b_i)$ in $T$; let $H_1, \ldots, H_l$ be the blocks of this partition.  Then $l \leq 2^{2|T|}$.

We claim that $|H_i|=1$ for every $i$. 
Suppose for a contradiction  that $(a_i,b_i), (a_j,b_j) \in H_1$. 
Let $c$ be an $L$-coloring of $G'' \setminus \{a_i, b_i\}$. 
Note that, since $N(a_i)=N(a_j)$ and $N(b_i)=N(b_j)$, setting $c(a_i)=c(a_j)$ and $c(b_i)=c(b_j)$  gives an $L$-coloring of $G''$, a contradiction. 
This proves that $s\leq 2^{2|T|}$.

Next we bound $|F_i|$ for each $i$. Let $i \in \{1, \ldots, s\}$.
Partition the set $\{C_j: j \in F_i\}$
according to the adjacency of $C_j$ in $T$.
Let $C_1^i, \ldots, C_{q_i}^i$ be the blocks of the partition.
Then $q_i\leq 2^{|T|}$.
Let $C_l \in C_1^i$.
For each $j \in \{2, \ldots, q_i\}$ let $Q^i_j$ be a shortest path from $C_l$ to $C_j^i$ in $F$. 
In $G$, $Q^i_j$ yields a path ${Q^i_j}'= a_1'$-$y_1'$-$a_2'$-$y_2'$-$\ldots$-$y_m'$-$a_m'$ where $a_1' \in C_l$, $a_m' \in A \cap C_j^i$, $a_2',\ldots,a_{m-1}' \in \bigcup_{l \in \{1, \ldots, q\} \setminus \{1,j\}}A \cap C_l^i$ and $y_1', \ldots, y_m' \in Y_1$. 
Let $Y(Q^i_j)=\{y_1', \ldots, y_m'\}$.
Since ${Q^i_j}'$ does not contain a $P_6$, it follows that $|Y(Q_j)| \leq 2$.
Let $Y_1^i=\bigcup_{j=2}^{q_i} Y(Q^i_j)$, and note that 
$|Y_1^i| \leq 2q_i-2 \leq 2(2^{|T|}-1)$. 
Moreover, let $\hat{Y}=\bigcup_{i=1}^s Y_1^i$, and note that 
$|\hat{Y}| \leq  2(2^{|T|}-1)s$.

Next we claim that $\hat{Y}=Y_1$. 
To see this, suppose that there exists a vertex $y \in Y_1 \setminus \hat{Y}$.
Note that $y$ is critical, and let $c$ be a coloring of $G \setminus y$.
We may assume that $N(y) \subseteq \bigcup_{i \in F_1} C_i$.
We will construct a coloring of $G''$ and obtain a contradiction.
By \eqref{monosets}, for every $i \in \{2, \ldots, s\}$ both of the sets $A(F_i)$ and $B(F_i)$ are monochromatic and so we can color $a_i$ and $b_i$ with the corresponding colors.

Let $F'$ be the graph with vertex set $F_1$ and such that $i$ is adjacent to $j$
if and only if there is a vertex $y' \in \hat{Y}$ (and therefore $y' \in Y_1^1$)
complete to $C_i \cup C_j$. Recall the partition $C_1^1, \ldots, C_{q_1}^1$.
By the definition of $Y_1^1$, there exists $C_1' \in C_1^1$ such that for every $i \in \{2, \ldots, q_1\}$ there is a path in $F'$ from $C_1'$ to a member $C_i'$ of $C_i^1$. Write $\{a_1'\}=C_1' \cap A$ and $\{b_1'\}=C_1' \cap B$, and let $\alpha=c(a_1')$ and $\beta=c(b_i')$.
Following the outline of the proof of \eqref{monosets} we deduce that $\alpha \neq \beta$, and that for each $i \in \{1, \ldots, q\}$ some vertex of $\bigcup_{C \in C_i^1} C \cap A$ is colored with color $\alpha$, and some vertex of $\bigcup_{C \in C_i^1} C \cap B$ is colored with color $\beta$.
Observe that for every index $i$ only vertices of $Y_1 \cup \bigcup_{C \in C_i^1} (C \cap B)$ are mixed on $\bigcup_{C \in C_i^1} (C \cap A$), and only vertices of $Y_1 \cup \bigcup_{C \in C_i^1} (C \cap A)$ are mixed on $\bigcup_{C \in C_i^1} (C \cap B)$.
Thus we can color $a_1$ with color $\alpha$ and $b_1$ with color $\beta$, obtaining a coloring of $G''$, a contradiction. 
This proves that $|Y_1| \leq  2(2^{|T|}-1)s$.
Now applying Claim~\ref{bipartite1} $|Y_1|$ times implies that there is a function $q$ that does not depend on $G$, such that 
$|\bigcup_{i=1}^s (A(F_i) \cup B(F_i))| \leq q(|T|)$.
Consequently,  $|V(G)| \leq |T| + |Y_1| + q(|T|) \leq |T| + 2(2^{|T|}-1)s+q(T)$.
This completes the proof. 
\end{proof}

Now Lemma~\ref{lem:3downto2} follows from Claim~\ref{c5}.

\section{$2P_3$-free 4-vertex critical graphs}\label{sec:sufficiency-2P3}

The aim of this section is to show that there are only finitely many $2P_3$-free 4-vertex critical graphs. The proof follows the same  outline as the proof of the  previous section.
Lemma~\ref{lem:2P3-list-size-2} deals with  $2P_3$-free minimal list-obstructions where every list has size at most two. In view of Lemma~\ref{lem:nec-list} 
the exact analogue of Lemma~\ref{lem:size2}
does not hold in this case, 
however if we add the additional assumption that the minimal list-obstruction is contained in a $2P_3$-free 4-vertex-critical graph that  was obtained by updating with respect to a set of precolored vertices, then we can show that the size of the obstruction  is bounded.

\begin{lemma}\label{lem:2P3-list-size-2}
There is a an integer $C>0$ such that the following holds.
Let $(G,L)$ be a list-obstruction.
Assume that $G$ is $2P_3$-free and the following holds.
\begin{enumerate}[(a)]
	\item Every list contains at most two entries.
	\item Every vertex $v$ of $G$ with $|L(v)|=2$ has a neighbor $u$ with $|L(u)|=1$ such that for all $w \in V(G)$ with $|L(w)|=2$, $uw \in E(G)$ implies $L(w)=L(v)$.
\end{enumerate}
Then $(G,L)$ contains a minimal list-obstruction with at most $C$ vertices.
\end{lemma}

Like in the case of $P_6$-free list-obstructions, we can use the precoloring technique to prove that the lemma above implies our main lemma.

\begin{lemma}\label{lem:2P3-reduction}
There is an integer $C>0$ such that every $2P_3$-free 4-vertex-critical graph
has at most $C$ vertices. Consequently,
there are only finitely many $2P_3$-free 4-vertex-critical graphs.
\end{lemma}

\subsection{Proof of Lemma~\ref{lem:2P3-list-size-2}}

Let $G'$ be an induced subgraph of $G$ such that $(G',L)$ is a minimal list-obstruction.
By Lemma~\ref{lem:propagationpath}, it suffices to prove that the length of any propagation path of $(G',L)$ is bounded by a constant.
To see this, let $P=v_1$-$v_2$-$\ldots $-$v_n$ be a propagation path of $(G',L)$ starting with color $\alpha$, say.
Consider $v_1$ to be colored with $\alpha$, and update along $P$ until every vertex of $P$ is colored.
Let this coloring of $P$ be denoted by $c$.
Recall condition~\eqref{eqn:conditions0} from the definition of propagation path: every edge $v_iv_j$ with $3 \le i < j \le n$ and $i \le j-2$ is such that 
\begin{equation*}
\mbox{$S(v_i)=\alpha\beta$ and $S(v_j)=\beta\gamma$,} 
\end{equation*}
where $\{1,2,3\} = \{\alpha,\beta,\gamma\}$.

First we prove that there is a constant $\delta$ such that there is a subpath $Q=v_{m}$-$v_{m+1}$-\ldots-$v_{m'}$ of $P$ of length at least $\lfloor \delta n \rfloor$ with the following property.
After permuting colors if necessary, it holds for all $i \in \{m,\ldots,m'\}$ that
\begin{equation*} 
S(v_i)=
\begin{cases}
32, & \mbox{ if } i \equiv 0 \mod 3\\
13, & \mbox{ if } i \equiv 1 \mod 3\\
21, & \mbox{ if } i \equiv 2 \mod 3
\end{cases}.
\end{equation*}
To see this, suppose there are two indices $i,j \in \{3,\ldots,n-3\}$ such that $i+2 \le j$ and $c(v_i)=c(v_{i+2})=c(v_j)=c(v_{j+2})$.
Moreover, suppose that $c(v_i)=c(v_{i+2})=c(v_j)=c(v_{j+2})=\alpha$ and $c(v_{i+1})=c(v_{j+1})=\beta$ for some $\alpha,\beta$ with $\{\alpha,\beta,\gamma\}=\{1,2,3\}$.
Thus, $L(v_{i+1})=L(v_{i+2})=L(v_{j+1})=L(v_{j+2})=\{\alpha,\beta\}$, $\alpha \in L(v_{j+3})$, and $\alpha \neq c(v_{j+3})$.
But now $v_{i}$-$v_{i+1}$-$v_{i+2}$ and $v_{j+1}$-$v_{j+2}$-$v_{j+3}$ are both induced $P_3$'s, according to~\eqref{eqn:conditions0}, and there cannot be any edge between them.
This is a contradiction to the assumption that $G$ is $2P_3$-free.
The same conclusion holds if $c(v_{i+1})=c(v_{j+1})=\gamma$.
Hence, there cannot be three indices $i,j,k \in \{3,\ldots,n-3\}$ such that $i+2 \le j$, $j+2 \le k$, and $$c(v_i)=c(v_{i+2})=c(v_j)=c(v_{j+2})=c(v_k)=c(v_{k+2})=\alpha.$$

Consider the following procedure.
Pick the smallest index $i \in \{3,\ldots,n-3\}$ such that $c(v_i)=c(v_{i+2})=1$, if possible, and remove the vertices $v_{i}$, $v_{i+1}$, and $v_{i+2}$ from $P$.
Let $P'$ be the longer of the two paths $v_1$-$v_2$-$\ldots $-$v_{i-1}$ and $v_{i+3}$-$v_2$-$\ldots $-$v_{n}$.
Repeat the deletion process and let $P'' = v_r$-$v_{r+1}$-\ldots-$v_{r'}$ be the path obtained.
As shown above, we now know that there is no index $j \in \{r+2,\ldots,r'-3\}$ with $c(v_j)=c(v_{j+2})=1$.

Repeating this process for colors 2 and 3 shows that there is some $\delta >0$ such that there is a path $Q=v_{m}$-$v_{m+1}$-\ldots-$v_{m'}$ of length $\lfloor \delta n \rfloor$ where $c(v_i)\neq c(v_{i+2})$ for all $i \in \{m-1,\ldots,m'-2\}$.
Thus, after swapping colors if necessary we have the desired property defined above.

From now on we assume that $G$ has sufficiently many vertices and hence $m'-m$ is sufficiently large.
Since $G$ is $2P_3$-free and hence $P_7$-free, the diameter of every connected induced subgraph of $G$ is bounded by a constant.
In particular, the diameter of the graph $G|(\{v_{m},\ldots,v_{m'}\})$ is bounded, and so we may assume that there is a vertex $v_i$ with $m \le i \le m'$ with at least 20 neighbors in the path $Q$.
We may assume that $c(v_i)=1$ and, thus, $S(v_i)=13$.

We discuss the case when $|N(v_i) \cap \{v_{m},\ldots,v_{i-1}\}| \ge 10$.
The case of $|N(v_i) \cap \{v_{i+1},\ldots,v_{m'}\}| \ge 10$ can be dealt with in complete analogy.

We pick distinct vertices $v_{i_1},\ldots,v_{i_{10}} \in N(v_i) \cap \{v_{m},\ldots,v_{i-1}\}$ where $i_1<i_2<\ldots<i_{10}$.
Note that~\eqref{eqn:conditions0} implies that $S(v_{i_j})=21$ for all $j \in \{1,\ldots,10\}$.

We can pick three indices $j_1,j_2,j_3$ with $r'< j_1 < j_2 < j_3 < m'$ such that 
\begin{itemize}
	\item $S(v_{j_1})=S(v_{j_2})=S(v_{j_3})=32$, and
	\item $i_2 + 5 = j_1$, $j_1 + 6 = j_2$, $j_2 + 4 \le i_7$, $i_8 + 5 = j_3$, and $j_3 + 4 = i$.
\end{itemize}
Recall that assumption (b) of the lemma we are proving implies the following.
Since $L(v_{j_{u}})=\{2,3\}$, $v_{j_{u}}$ has a neighbor $x_{j_u}$ with $L(x_{j_u})=\{1\}$, $u=1,2,3$, such that $x_{j_u}$ is not adjacent to any vertex $v_j$ with $m \le j \le m'$ and $j \equiv 1 \mod 3$ or $j \equiv 2 \mod 3$.

Suppose that $x_{j_u}=x_{j_{u'}}$ for some $v_{j_{u'}}$ with $u' \in \{1,2,3\}\setminus \{u\}$.
Now the path $v_{j_u}$-$x_{j_u}$-$v_{j_{u'}}$ is an induced $P_3$, and so is the path $v_{i_1}$-$v_i$-$v_{i_2}$, both according to condition~\eqref{eqn:conditions0}.
Moreover, there is no edge between those two paths, due to~\eqref{eqn:conditions0}, which is a contradiction.
Hence, the three vertices $x_{j_u}$, $x_{j_u}$, and $x_{j_u}$ are mutually distinct and, due to the minimality of $(G,L)$, mutually non-adjacent.

Consider the induced $P_3$'s $v_{j_1+1}$-$v_{j_1}$-$x_{j_1}$ and $v_{j_3+1}$-$v_{j_3}$-$x_{j_3}$.
Since $G$ is $2P_3$-free, there must be an edge between these two paths.
According to~\eqref{eqn:conditions0}, it must be the edge $v_{j_1+1}v_{j_3}$.
For similar reasons, the edge $v_{j_2+1}v_{j_3}$ must be present.
Now the path $v_{j_1+1}$-$v_{j_3}$-$v_{j_2+1}$ is an induced $P_3$, and so is the path $v_{i_7}$-$v_i$-$v_{i_8}$.
Moreover, there is no edge between those two paths, due to~\eqref{eqn:conditions0}, which is a contradiction.
This completes the proof.

\subsection{Proof of Lemma~\ref{lem:2P3-reduction}}

%We make use of Claim~\ref{precolor}, Claim~\ref{update4}, and Theorem~\ref{structure} presented earlier. Note that Claim~\ref{precolor} allows us precolor a bounded graph and  Claim~\ref{update4} allow us update in a constant number of times.
We start with two statements that allow us to precolor sets of vertices with certain properties. In this subsection $G$ is always a $2P_3$-free graph, and all lists are subsets of $\{1,2,3\}$.

\begin{claim}\label{color}
Assume that $(G,L)$ is a list-obstruction. Let $X\subseteq V(G)$ be such that there exists a coloring $c$ of $G|X$ with the following property: for each $x\in X$ there exists a set $N_x\subseteq V(G)$ with $|N_x|\leq k$ such that $x$ is colored $c(x)$ in every coloring of $(G|(\{x\}\cup N_x),L)$. 
Let $L'$ be a list system such that 
\begin{equation*} 
L'(v)=
\begin{cases}
L(v), & \mbox{ if } v\in V(G) \setminus X\\
\{c(x)\}, & \mbox{ if } v\in X
\end{cases}.
\end{equation*}
Then the following holds.
	\begin{enumerate}[(a)]
		\item $(G,L')$ is a list-obstruction.
		\item If $K\subseteq V(G)$ is such that $(G|K,L')$ is a minimal list-obstruction induced by $(G,L')$, then $(G,L)$ contains a minimal list-obstruction of size at most $(k+1)|K|$.
	\end{enumerate}
\end{claim}
\begin{proof}
	Since $L'(v)\subseteq L(v)$ for all $v \in V(G)$, $(G,L')$ is also a list-obstruction. This proves (a). 
	
Let $A=G|(K\cup \bigcup\limits_{x\in K\cap X}N_x)$, then $|V(A)|\leq (k+1)|K|$.  
Suppose that there exists a coloring, $c'$ of  $(A,L)$.
Note that for every $x\in V(A)$, $N_x\subseteq A$. Hence by the definition of $X$, $c'(x)=c(x)$ for every $x\in V(A)$. This implies that $c'$ is also a coloring of $(A,L')$, which gives a coloring of $(G|K,L')$, a contradiction. Therefore $(A,L)$ is a list-obstruction induced by $(G,L)$. Since $|V(A)|\leq (k+1)|K|$, (b) holds. This completes the proof. 
\end{proof}

\begin{claim}\label{update3}	Let $(G,L)$ be a list-obstruction, and let $X \subseteq V(G)$ be a vertex subset such that $|L(x)|=1$ for every $x \in X$. Let $Y=N(X)$, and let $Y'\subseteq Y$ be such that for every $v\in Y'$, $|L(v)|=3$. For every $v\in Y'$, pick $x_v\in N(v)\cap X$. 
	Let $L'$ be the list defined as follows.
	\begin{equation*} 
	L'(v)=
	\begin{cases}
	L(v), & \mbox{ if } v\in V(G) \setminus Y'\\
	L(v)\setminus L(x_v), & \mbox{ if } v\in Y'
	\end{cases}.
	\end{equation*}
	Let $(G',L')$ be a minimal list-obstruction induced by $(G,L')$. 
	Then there exists a minimal list-obstruction induced by $(G,L)$, say $(G'',L)$, with $|V(G'')| \leq 2|V(G')|$.

\end{claim}
\begin{proof}
	Let $R=\{x_v:v\in V(G')\cap Y'\}$ and let $P=R\cup V(G')$. It follows that $|V(P)|\leq 2|V(G')|$.
	It remains to prove that $(G|P,L)$ is not colorable.
	Suppose there exists a coloring $c$ of $(G|P,L)$.
	Note that $c$ is not a coloring of $(G',L')$ and $G'$ is an induced subgraph of $G|P$. Hence there exists $w \in V(G')$ such that $c(w) \not \in L'(w)$. By the construction of $L'$,
	it follows that $w\in Y'$ and that $c(w) \in L(w) \setminus L'(w)=\{c(x_w)\}$, which is a contradiction.
	This completes the proof.
\end{proof}
Let $G$ be a $2P_3$-free 4-vertex-critical graph such that $|V(G)|\geq 5$, then the following claim holds.
\begin{claim}\label{T}
	At least one of the following holds
	\begin{enumerate}
		\item There exists $S_0\subseteq V(G)$ such that $|S_0|\leq 5$, $G|S_0$ contains a copy of $P_3$ and $S_0\cup  N(S_0) \cup N(N(S_0)) = V(G)$, or
		\item $G$ has a semi-dominating set of size at most $5$.
	\end{enumerate}

\end{claim}
\begin{proof}
	Since $G$ is $2P_3$-free and thus also $P_7$-free, Theorem~\ref{structure} states that $G$ has a dominating induced $P_5$ or a dominating $P_5$-free connected induced subgraph, denoted by $D_f$. Recall that a dominating set is always a semi-dominating set; so we may assume that the latter case holds and $|V(D_f)|\geq 6$. By applying Theorem~\ref{structure} to $D_f$ again, we deduce that $D_f$ has a dominating induced subgraph $T$, which is isomorphic to $P_3$ or a connected $P_3$-free graph. 
	
	If $T$ is isomorphic to $P_3$, then we are done by setting $S_0=V(T)$.
	Hence we may assume $T$ is a connected $P_3$-free graph. Therefore $T$ is a complete graph, and so  $V(T)\leq 3$.  If there exists a vertex $s'\in V(G\setminus T)$ mixed on $T$, we are done by setting $S_0=V(T)\cup \{s'\}$. Hence we may assume that for every $v\in V(D_f\setminus T)$, $v$ is complete to $T$. Since $|V(G)|\geq 5$, it follows that $D_f$ is $K_4$-free. Therefore there exist $v,w\in V(D_f\setminus T)$ such that $v$ is non-adjacent to $w$ and we are done by setting $S_0=V(T)\cup \{v,w\}$.
\end{proof}

If $G$ has a semi-dominating set of size at most $5$, we are done by Lemma~\ref{P_4} and Lemma~\ref{lem:2P3-list-size-2}. Hence we may assume there exists $S_0$ defined as in Claim~\ref{T}.

%Let $S$ be the set of vertices with lists of size $1$, then $S_0\subseteq S$.
%Let $ B=N(S) $, $ X=V(G) \setminus (S\cup B) $, $ B_{ij}=\{b \in B : L(b)=\{i,j\} \}$ for all $\{i,j\}\subseteq \{1,2,3\}$, and $ S_i=\{s \in S : L(s)=\{i\} \}$ for all $i\in \{1,2,3\}$. 
%Recall that $G$ $2P_3$-free and that  $S$ contains a $ P_3 $. Then $X$ does not contain a $P_3$, and so every component of $X$ is a clique.
%Then we may assume every component of $X$ has size at most 3. 

%Let $ \{i,j,k\}=\{1,2,3\} $. Suppose that there exists $x \in X$ such that $x$ has two adjacent neighbors $a,b \in B_{ij} $.
%Let $C$ be the component of $X$ containing $x$.
% By Claim~\ref{color} we can set $ L(x)=\{k\} $ and update the lists of the vertices in $C$ with respect to $x$, then $ L(y)=\{i,j\} $ for every $y\in V(C)\setminus\{x\}$. We simultaneously apply the above procedure for all components of $X$ containing such a vertex.
%Let $ X_1 $ be the union of all such components of $X$ and let $ X_2=X\setminus X_1 $. 
%Then $ X_2 $ is anticomplete to $ X_1 $ and we may assume that for every $ x \in X_2 $, $ N(x) \cap B_ {ij}  $ is a stable set.
%	Then $ X_2 $ is anticomplete to $ X_1 $ and we may assume that for every $ x \in X_2 $, $ N(x) \cap B_ {ij}  $ is a stable set.
	
For a list system $L'$ of $G$, we say that $(X_1, X_2, B, S)$ is the \emph{partition with respect to $L'$} by setting: 
\begin{enumerate}[(a)]
	\item $S=\{v\in V(G):|L'(v)|=1\}$.
	\item $ B=N(S) $; assume that $|L'(v)|=2$ for every $v\in B$.
	\item Let $ X=V(G) \setminus (S\cup B) $. We say that $C$ is a \emph{good component} of $X$ if there exist $x \in C$ and $\{i,j\}\subseteq \{1,2,3\}$ so that $x$ has two adjacent neighbors $a,b \in B_{ij} $, where $ B_{ij}=\{b\in B $ such that $ L'(b)=\{i,j\}\} $. 
	%Let $C$ be the component of $X$ containing $x$.
	% By Claim~\ref{color} we can set $ L(x)=\{k\} $ and update the lists of the vertices in $C$ with respect to $x$, then $ L(y)=\{i,j\} $ for every $y\in V(C)\setminus\{x\}$. We simultaneously apply the above procedure for all components of $X$ containing such a vertex.
	Let $ X_1 $ be the union of all good components of $X$ and let $ X_2=X\setminus X_1 $. 

\end{enumerate}

Let $(X_1, X_2, B, S)$ be the partition with respect to $L'$.  
%let $ X_{ij} $ be the set of vertices in $ X_2 $ with two neighbors in $ B_{ij} $ for all $\{i,j\}\subseteq \{1,2,3\}$. 
Define $X=X_1\cup X_2$. For every $1\leq i\leq j\leq 3$, define $ B_{ij}=\{b\in B $ such that $ L'(b)=\{i,j\}\} $ and $ X_{ij}=\{x\in X_2$ such that $ |N(x)\cap B_{ij}|\geq 2\} $.  For $ \{i,j,k\}=\{1,2,3\} $, let us say that a component $ C $ of $ X_2 $ is \emph{i-wide} if there exist $ a_j $ in $ B_{ik} $ and $ a_k $ in $ B_{ij} $ such that $ C $ is complete to $ \{a_j,a_k\} $. 
We call $ a_j $ and $ a_k $ \emph{i-anchors} of $ C $. 
Note that a component can be $ i $-wide for several values of $ i $. Let $L''$ be a subsystem of $L'$ and let $(X'_1, X'_2, B', S')$ be the partition with respect to $L''$. Then $S\subseteq S'$, $B'\setminus B \subseteq X_1\cup X_2$ and $ X'_2\subseteq X_2$.
%We call this process \emph{enlarging} $S$. 
%Note that during this process we will keep the property that $S$ is a seed and $B=N(S)$ as an invariant.

 Next we define a sequence of new lists $L_0,\ldots,L_5$. Let $\{i,j,k\}=\{1,2,3\}$. Let $S_0$ be as in Claim~\ref{T}, and let $L_0=L$.

%In view of Lemma~\ref{precolor} and Lemma~\ref{update4}, 

 %Note that the above statements always hold for a partition with respect to $L'$ if $L'$ is a subsystem of $L_0$. 

 \begin{enumerate}
 	\item  Let $L_1$ be the list system obtained by precoloring $S_0$ and updating three times. Let $(X^1_1, X^1_2, B^1, S^1)$ be the partition with respect to $L_1$. 
 	
	\item 	For each $k\in\{1,2,3\}$, choose $ x_k\in X^1_{ij} $ such that
	$ |N(x_k)  \cap  B^1_{ij}| $ is minimum. Let $ a_k,b_k\in N(x_k)  \cap  B^1_{ij} $. Let $L_2$ be the list system obtained from $L_1$ by precoloring $ \bigcup\limits_{i=1}^3 \{a_i,b_i,x_i\} $
	and updating  the lists of vertices three times. Let $(X^2_1, X^2_2, B^2, S^2)$ be the partition with respect to $L_2$ .
	%the lists of vertices in $X_{23}\cup (N(X_{23})\cap B)$ 

	\item For each $k\in\{1,2,3\}$, let $ \hat{B}_k \subseteq B^2_{ij} $ with $ |\hat{B}_k|\leq 1 $ be defined as follows. If there does not exist a vertex $v \in B^2_{ij}$ that starts a path  $ v$-$u$-$w $ where $ u, w \in X^2_2 $, then $ \hat{B}_k = \emptyset $. Otherwise choose $ b_k\in B^2_{ij}$ maximizing the number of pairs $ (u, w) $ where $b_k$-$u$-$w$ is a path and let $ \hat{B}_k = \{b_k\} $.
	 Let $L_3$ be the list system from $L_2$ obtained by precoloring $ \hat{B}_1\cup \hat{B}_2 \cup \hat{B}_3  $ and updating three times. Let $(X^3_1, X^3_2, B^3, S^3)$ be the partition with respect to $L_3$. 
	\item Apply step $2$ to $(X^3_1, X^3_2, B^3, S^3)$ with list system $L_3$; let $L_4$ be the list system obtained and let $(X^4_1, X^4_2, B^4, S^4)$ be the partition with respect to $L_4$. 
	\item For every component $C_t$ of $X^4_2$ with size $2$, if $C_t$ is $ i $-wide with $i$-anchors $a^t,b^t$, set $L_5(a^t)=L_5(b^t)=\{i\}$; then let $L_5$ be the list system after updating with respect to $\bigcup_t\{a^t,b^t\}$ three times. Let $(X^5_1, X^5_2, B^5, S^5)$ be a partition with respect to $L_5$. 
	\end{enumerate}

By Lemma~\ref{precolor} and Lemma~\ref{update4}, it is enough to prove that $ (G,L_4) $ induces a bounded size list-obstruction. To do that, we 
prove the same for $ (G,L_5) $, and then  use Claim~\ref{color} and Lemma~\ref{update4}, as we explain in the remainder of this section.

We start with  a few technical statements. 

\begin{claim}\label{clm:P3(2.5)} Let $1\leq m\leq l\leq 5$. Then the following holds.
	\begin{enumerate}
		\item For every vertex in $x\in X^m$, $|L_m(x)|=3$, and every component of $X^m$ is a clique with size at most $3$.
		\item If no vertex of $B_{ij}^m$ is mixed on an edge in $G|X^m_2$, then no vertex of $B_{ij}^l$ is mixed on an edge in $G|X_2^l$.
		\item If no vertex of $X^m_2$ has two neighbors in $B^m_{ij}$ and no vertex of $B^m$ is mixed on an edge in $G|X^m_2$, then no vertex of $X^l_2$ has two neighbors in $B_{ij}^l$.
	\end{enumerate}
	
\end{claim}
\begin{proof}
	By construction, for every vertex in $x\in X^m$, $|L_m(x)|=3$. Observe that $S_0\subseteq S^m$. Recall that $G$ is $2P_3$-free and that $S_0$ contains a $ P_3 $. Hence $X^m$ does not contain a $P_3$, and so every component of $X^m$ is a clique. Since $|V(G)|\geq 5$, it follows that every component of $X^m$ has size at most $ 3 $. This proves the first statement.
	
	Let $b\in B_{ij}^l$ be mixed on the edge $uv$ such that $u,v\in X_2^l$. Recall that $ X_2^l\subseteq X_2^m$; thus $\{u,v\}\subseteq X_2^m$. By assumption $b\not \in B_{ij}^m$ and hence $b\in X^m$. But now $b$-$u$-$v$ is a $P_3$ in $X^m$, a contradiction. This proves the second statement.
	
	To prove the last statement, suppose that there exists $y\in X_2^l$ with two neighbors $u,v\in B_{23}^l$. Since $y\in X_2^l$, it follows that $u,v$ are non-adjacent. Note that $y\in X_2^m$, hence by assumption and symmetry, we may assume that $v\notin B^m$. Therefore $v\in X^m$. Since $X_1^m$ is the union of components of $X^m$, and $y\in X_2^m$ is adjacent to $v$, it follows that $v\in X^m\setminus X_1^m$, and consequently $v\in X_2^m$. If $u\notin B^m$, then $u$-$y$-$v$ is a $P_3$ in $G|X^m$, contrary to the first statement.  Hence $u\in B^m$  and then $u$ is mixed on the edge $vy$ of $ G|X^m_2$, a contradiction. This completes the proof.
\end{proof}

\begin{claim}\label{clm:P3(0)}

	$X^1_{12}\cup X^1_{23}\cup X^1_{13}\subseteq B^2\cup S^2. 
	$
\end{claim}
\begin{proof}
	Suppose that there exists $ x'  \in  X^1_{ij}\setminus (B^2\cup S^2) $ for some $1\leq i\leq j\leq 3$; then $ |L_2(x')|=3 $. Let $ x_k\in X^1_{ij} $ and $ a_k,b_k\in N(x_k)  \cap  B^1_{ij} $ be the vertices chosen to be precolored in the step creating $L_2$. Then $ x' $ is non-adjacent to $ \{x_k,a_k,b_k\} $. The minimality of $ |N(x_k) \cap B^1_{ij}| $
	implies that there exist $ a',b' \in (N(x')\cup B^1_{ij})\setminus N(x_k) $.  
Since $G$ is $2P_3$-free, there exists an edge between $ \{a_k,b_k,x_k\} $ and $ \{a',b',x'\} $. Specifically, there exists an edge between $\{a_k,b_k\}$ and $\{a',b'\}$. We may assume that $L_2(a_k)=\{i\}$ and $a_k$ is adjacent to at least one of $a',b'$. Recall that $L_1$ is obtained by precoloring $ \bigcup\limits_{i=1}^3 \{a_i,b_i,x_i\} $
and updating three times. It follows that $j\notin L_2(x')$, a contradiction. 
\end{proof}

 \begin{claim}\label{clm:P3(2)}
 	No vertex of $B^3$ is mixed on an edge of $ X_2 $.
 \end{claim}
 \begin{proof}
 
 	Suppose that there exists a path $ b'$-$x_1'$-$x_2' $ such that $ b' \in B^3_{ij} $ and $ x_1',x_2'\in X^3_2 $. Note that $ x_1',x_2'\in X^2_2 $ since $L_3$ is a subsystem of $L_2$. By Claim~\ref{clm:P3(2.5)}.1, $X^2$ is $P_3$-free. Hence $b'\in B^2_{ij}$. By Claim~\ref{clm:P3(2.5)}.3, there exists $ b \in B^2_{ij} $ such that $ b-x-y $ is a path where $ x, y \in X^2_2 $. Then in step $3$, $ \hat{B}_k \neq \emptyset $ and let $ b\in \hat{B}_k $. 
 	By the construction of $L_3$ and since $ x_1',x_2'\in X^3_2 $, $ b $ is anticomplete to $ \{b',x_1',x_2'\} $.
 	By the construction of $\hat{B}_k $, there exist $ x_1, x_2\in X^2_2 $ such that $ b$-$x_1$-$x_2 $ is a path and $ b' $ is not mixed on $x_1x_2$.
 	If $ \{x_1,x_2\} $ is not anticomplete to $ \{x_1',x_2'\} $, then by Claim~\ref{clm:P3(2.5)}.1 $ G|\{x_1,x_2,x_1',x_2'\}$  is a $K_4$, a contradiction to the fact that $|V(G)|\geq 5$. Hence $ \{x_1,x_2\} $ is anticomplete to $ \{x_1',x_2'\} $. Since $G$ is $2P_3$-free, there exists an edge
 	between $ b' $ and $ \{x_1,x_2\} $. Consequently, $ b' $ is complete to $ \{x_1,x_2\} $.
 	Now $ x_1 $ has two neighbors in $ B^2_{ij} $, namely $ b $ and $ b' $. By Claim~\ref{clm:P3(0)}, $x_1\notin B^1_{ij}$. It follows that either $b\in X^1$ or $b'\in X^1$. If $b\in X^1$, then $b-x-y$ is a $P_3$ in $X^1$, contrary to Claim~\ref{clm:P3(2.5)}.1. Hence $b'\in X^1$. It follows that $b'$-$x'_1$-$x'_2$ is a $P_3$ in $X^1$, again contrary to Claim~\ref{clm:P3(2.5)}.1. This completes the proof.
 \end{proof}
 
We are now ready to prove that it suffices to show that $(G,L_5)$ induces a 
mininal list-obstruction of bounded size.  Let $C_t$ be an $i$-wide component 
of $X_1^4$ with $C_t=\{x_t,y_t\}$, and let $a_t,b_t$ be the $i$-anchors of 
$C_t$ that were  chosen in step 5. By the definition of $i$-anchors, 
$L_4(a_t)\cap L_4(b_t)=\{i\}$ and $\{a_t,b_t\}$ is 
complete to $C_t$; therefore  $c(a_t)=c(b_t)=i$ for every coloring $c$ of 
$ (G|\{x_t,y_t,a_t,b_t\}, L_4)$. Hence we can apply Claim~\ref{color} to 
$L_4$. By Claim~\ref{color} and Lemma~\ref{update4}, it is enough to show that 
$ (G,L_5) $ induces a bounded size list-obstruction.

\begin{claim}\label{clm:P3(3)}
$ X^5_2 $ is stable.
\end{claim}
\begin{proof}
Since $|V(G)|\geq 5$ and since no vertex of $B^5$ is mixed on an edge of $G|X^5_2$, by Claim~\ref{clm:P3(2.5)}.1 every component of $ X^5_2 $ has size at most $2$. We may assume some component $C$ of $X^5_2$ has size exactly $2$, for otherwise the claim holds. Then $C$ is a component of $ X^4_2 $. By Claim~\ref{clm:P3(2.5)} and Claim~\ref{clm:P3(0)}, no vertex of $X_2^4$ has two neighbors in $B^4_{ij}$. Since every vertex in $G$ has degree at least $3$, every vertex of $C$ has a neighbor in at least two of $ B^4_{12},B^4_{23},B^4_{13}$. It follows that 
%every component of $ X^4_2 $ 
$C$ is $ i $-wide for some $ i $ and therefore $C\subseteq S^5\cup B^5$, a contradiction.
%in step $5$ we reduce the lists of vertices of all the components in $X_2^4$ with size $2$.  This completes the proof.  
\end{proof}

By Claim~\ref{clm:P3(2.5)} and Claim~\ref{clm:P3(2)}, no vertex of $X_2^5$ has two neighbors in $B^5_{ij}$. Since every vertex in $G$ has degree at least $3$, it  follows that every vertex  of $X_2^5$ has exactly one neighbor in each of $ B^5_{ij} $.
Let $ Y_0 $, $ Y_1 $, \ldots, $ Y_6 $ be a partition of $ X^5_2 $ as follows. Let $x\in X_2^5$ and $ a_k= N(x)\cap B^5_{ij} $ for $\{i,j,k\}=\{1,2,3\}$. If $\{a_1,a_2,a_3\}$ is a stable set, then $ x\in Y_0 $; if $E(G|\{a_1,a_2,a_3\})=\{a_ia_j\}$, then $x\in Y_k$; and if $E(G|\{a_1,a_2,a_3\})=\{a_ia_j, a_ia_k\}$, then $x\in Y_{i+3}$. Note that $ G|\{a_1,a_2,a_3\} $ cannot be a clique since $V(G)\geq 5$.
For each non-empty $ Y_s $, pick $ x_s \in Y_s $, and let $ a_{sk}\in N(x_s)\cap B^5_{ij} $ for $\{i,j,k\}=\{1,2,3\}$.
Let $L_6$ be the list system obtained by precoloring $ \bigcup\limits_{i=0}^{6}\{x_i,a_{i1},a_{i2},a_{i3}\} $ with $c$ and updating three times. 

\begin{claim}\label{clm:P3(4)}
For every $x\in X^5_2$,	$|L_6(x)|\leq 2$
\end{claim}
\begin{proof}
Suppose there exists $ y \in Y_i $ such that $|L_6(y)|=3$; let $ b_1=N(y)\cap B^5_{23}$, $b_2=N(y)\cap B^5_{13}$, and $b_3=N(y)\cap B^5_{12} $. Then $ \{a_{i1},a_{i2},a_{i3},x_i\} $ and $ \{b_1,b_2,b_3,y\} $ are disjoint sets.
Note that $ c(a_{i1}), c(a_{i2}), c(a_{i3})$ can not all be pairwise different, and so by symmetry we may assume that $ c(a_{i1})=c(a_{i2})=3 $ and $ c(a_{i3})=2 $. 
Thus, $a_{i1}a_{i2}$ is a non-edge.
By the construction of $L_6$ and since $|L_6(y)|=3$, the only possible edges
between the sets $ \{a_{i1},a_{i2},a_{i3}\} $ and $ \{b_1,b_2,b_3\} $ are $ a_{i3}b_2 $, $ a_{i1}b_3 $ and $ a_{i2}b_3 $. Recall that every vertex of $X_2^5$ has exactly three neighbors in $B^5$, and so $ N(y)\cap B^5= \{b_1,b_2,b_3\} $ and $ N(x_i)\cap B^5= \{a_{i1},a_{i2},a_{i3}\} $. Since $ G|\{a_{i1},x_i,a_{i2},b_1,y,b_2\} $ is not a $ 2P_3$, it follows that $ b_1 $ is adjacent to $ b_2 $.
But this contradicts to the fact that both $ x_i $ and $ y $ belong to $ Y_i $.  
\end{proof}

	Let $(X^1_6, X^2_6, B^6, S^6)$ be the partition with respect to $L_6$. 
For every component $C_s \subseteq X^6_1$,  let  $\{i,j,k\}=\{1,2,3\}$ be
such that there exists $x^k_s\in C_s$ with two adjacent neighbors in $B_{ij}^6$. Define  $L'_6(x_k^s)=\{k\}$; let $P$ be the set of all such vertices $x_k^s$, and   let $L'_6(v)=L_6(v)$ for every $v \not \in P$. Let $L^*$ be the list system 
obtained from $L_6'$ by  updating with respect to $P$ three times. Pick $x\in P$, then there exist $i,j \in \{1,2,3\}$ for which some $a,b\in N(x)\cap B_{ij}^6$
are adjacent. Then $L_6(a)=L_6(b)=\{i,j\}$. As a result, for every coloring $c$ of $(G|\{x,a,b\}, L_6)$, $c(x)=k$. This implies that we can apply Claim~\ref{color} to $L_6$. By Lemma~\ref{update4} and Claim~\ref{color},  it is enough to prove that $ (G,L^*) $ induces a bounded size list-obstruction. Let $(X^*_1, X^*_2, B^*, S^*)$ be the partition with respect to $L^*$. 
Then by Claim~\ref{clm:P3(2.5)}.1 $ X^*_1, X^*_2 $ are empty.  Now $(G,L^*)$ satisfies the hypotheses of Lemma~\ref{lem:2P3-list-size-2}, and this finishes the proof of Lemma~\ref{lem:2P3-reduction}.

\section{$P_4+k P_1$-free minimal list-obstructions}\label{sec:sufficiency-P4+kP1}

In this section we prove that there are only finitely many $P_4+k P_1$-free minimal list-obstructions.
This also implies that there are only finitely many $P_4+k P_1$-free 4-vertex-critical graphs.

\begin{lemma}\label{lem:size-2-P4}
Let $(G,L)$ be a minimal list-obstruction such that each list has at most two entries.
Moreover, let $G$ be $(P_4 + k P_1)$-free, for some $k \in \mathbb N$.
Then $V(G)$ is bounded from above by a constant depending only on $k$.
\end{lemma}
\begin{proof}
By Lemma~\ref{lem:propagationpath}, it suffices to prove that every propagation path in $(G,L)$ has a bounded number of vertices.
To see this, let $P=v_1$-\ldots-$v_n$ be a propagation path in $(G,L)$ starting with color $\alpha$, say.
Consider $v_1$ to be colored with $\alpha$, and update along $P$ until every vertex is colored.
Call this coloring $c$.
Suppose that $n \ge 100 k^2+100$.
Our aim is to show that this assumption is contradictory.
Recall condition~\eqref{eqn:conditions0} from the definition of propagation path: every edge $v_iv_j$ with $3 \le i < j \le n$ and $i \le j-2$ is such that 
\begin{equation*}
\mbox{$S(v_i)=\alpha\beta$ and $S(v_j)=\beta\gamma$,} 
\end{equation*}
where $\{1,2,3\} = \{\alpha,\beta,\gamma\}$.

First we suppose that there is a sequence $v_i,v_{i+1},\ldots,v_{j}$ with $2 \le i \le j \le n$ and $j-i \ge 5+2k$ such that $c(v_{i'})=c(v_{i'+2})$ for all $i'$ with $i \le i' \le j-2$.
But then~\eqref{eqn:conditions0} implies that $v_{i+1}$-$v_{i+2}$-\ldots-$v_{j}$ is an induced path, and thus $G$ is not $P_4+k P_1$-free, a contradiction.

Suppose now that there is an index $i$ with $2 \le i \le \lceil n/2\rceil-3$ such that $c(v_i)=c(v_{i+2})=\alpha$ and $c(v_{i+1})=c(v_{i+3})=\beta$.
In particular, $L(v_{i+3})=\{\alpha,\beta\}$.
Now condition~\eqref{eqn:conditions0} of the definition of a propagation path implies that there cannot be an edge between $v_i$ and $v_{i+3}$, and so $v_i$-$v_{i+1}$-$v_{i+2}$-$v_{i+3}$ is an induced $P_4$. Therefore no such sequence exists.

We now pick $k$ disjoint intervals of the form $\{j,\ldots,j+7+2k\} \subseteq \{\lceil n/2\rceil+1,\ldots,n\}$.
As shown above, each of these intervals contains an index $i'$ in its interior with $c(v_{i'})=\alpha$.
These $v_{i'}$ form a stable set and~\eqref{eqn:conditions0} implies that the induced path $v_i$-$v_{i+1}$-$v_{i+2}$-$v_{i+3}$ is anticomplete to each $v_{i'}$, a contradiction to the fact that $G$ is $P_4 + k P_1$-free.

Now suppose that there is an index $i$ with $r+1 \le i \le \lceil(r+s)/2\rceil-3$ such that $c(v_i)=c(v_{i+2})=\alpha$ and $c(v_{i+1})=\beta$.
From what we have shown above we know that $c(v_{i-1})=c(v_{i+3})=\gamma$, where $\{\alpha,\beta,\gamma\}=\{1,2,3\}$.
Thus, we have $S(v_i)=\alpha\gamma$,
$S(v_{i+1})=\beta\alpha$,
$S(v_{i+2})=\alpha\beta$, and
$S(v_{i+3})=\gamma\alpha$.
According to~\eqref{eqn:conditions0}, the path $v_i$-$v_{i+1}$-$v_{i+2}$-$v_{i+3}$ is induced.

Pick a vertex $v_j$ with $\lceil n/2\rceil +1\le j \le n$.
According to~\eqref{eqn:conditions0}, $v_j$ is anticomplete to the path $v_i$-$v_{i+1}$-$v_{i+2}$-$v_{i+3}$ unless one of the following holds.
\begin{enumerate}[(a)]
	\item $S(v_j)=\alpha\beta$,
	\item $S(v_j)=\alpha\gamma$,
	\item $S(v_j)=\beta\gamma$, or
	\item $S(v_j)=\gamma\beta$.
\end{enumerate}
Let us say that $v_j$ is of \emph{type A} if it satisfies one of the above conditions.
If $v_j$ is not of type A, we say it is of \emph{type B}.

We claim that there are at most $3k-3$ vertices of type B.
To see this, suppose there are at least $3k-2$ vertices of type B.
By definition, each vertex of type B is anticomplete to the set the path $v_i$-$v_{i+1}$-$v_{i+2}$-$v_{i+3}$.
Since $(G,L)$ is a minimal obstruction and not every vertex is of type B, the graph induced by the vertices of type B is 3-colorable.
Picking the vertices of the majority color yields a set $S$ of $k$ independent vertices of type B.
But now the set $\{v_i,\ldots,v_{i+3}\} \cup S$ induces a $P_4 + k P_1$ in $G$, a contradiction.

So, there are at most $3k-3$ vertices of type B.
Suppose there are more than $(3k-2)(7+2k)$ many vertices of type A.
Then there is an index $t \ge \lceil n/2\rceil+1$ such that $v_t+j'$ is of type A for all $j' \in \{0,\ldots,6+2k\}$.
Suppose that there is an index $j' \in \{0,\ldots,5+2k\}$ such that $c(v_{t+j'}) = \alpha$.
Then $S(v_{t+j'+1})=\cdot~\alpha$, in contradiction to the fact $v_{t+j'+1}$ is of Type A.
So, for all $j' \in \{0,\ldots,5+2k\}$ we have that $c(v_{t+j'}) \neq \alpha$, in contradiction to what we have shown above.
Summing up, $n$ is bounded by $2(3k-2)(7+2k)+1$ if there is an index $i$ with $2 \le i \le \lceil n/2\rceil-3$ such that $c(v_i)=c(v_{i+2})=\alpha$ and $c(v_{i+1})=\beta$.

Hence, our assumption $n \ge 100 k^2+100$ implies that $c(v_i)\neq c(v_{i+2})$ for all $i$ with $2 \le i \le \lceil n/2\rceil-3$.
This means that, without loss of generality,
\begin{equation}
c(v_i)=
\begin{cases}
1, & i = 1 ~(3)\\
2, & i = 2 ~(3)\\
3, & i = 0 ~(3)
\end{cases}
\end{equation}
for all $i$ with $2 \le i \le \lceil n/2\rceil-3$.

Consider the path $v_4$-$v_5$-\ldots-$v_{7+2k}$.
Since $G$ is $P_4 + k P_1$-free, this is not an induced path. 
Hence, there is an edge of the form $v_iv_j$ with $i < j$.
If $S(v_i)=\alpha\beta$, we must have $S(v_j)=\beta\gamma$, due to~\eqref{eqn:conditions0}.
Consequently, $S(v_{i-1})=\beta\gamma$, and $S(v_{j+1})=\alpha\beta$.
In particular,~\eqref{eqn:conditions0} implies that $v_{i-1}$ is non-adjacent to $v_{j+1}$, and so $v_{i-1}$-$v_i$-$v_j$-$v_{j+1}$ is an induced path.

Like above, we now pick $k$ disjoint intervals of the form $\{j,\ldots,j+7+2k\} \subseteq \{\lceil n/2\rceil+1,\ldots,n\}$.
Each of these intervals contains an index $i'$ in it's interior with $c(v_{i'})=\alpha$.
These $v_{i'}$ form a stable set and~\eqref{eqn:conditions0} implies that the induced path $v_{i-1}$-$v_i$-$v_j$-$v_{j+1}$ is anticomplete to each $v_{i'}$, a contradiction.
This completes the proof.
\end{proof}

Using the above statement, we can now derive our main lemma.

\begin{lemma}\label{lem:P4+kP1}
There are only finitely many $P_4+k P_1$-free minimal list-obstructions, for all $k \in \mathbb N$.
\end{lemma}
\begin{proof}
Let $(G,L)$ be a $P_4+kP_1$-free minimal list-obstruction.
If $G$ is $P_4$-free, we are done, since there is only a finite number of $P_6$-free minimal obstructions.
So, we may assume that $G$ contains an induced $P_4$, say $v_1$-$v_2$-$v_3$-$v_4$.
Let $R=V(G)\setminus N(\{v_1,v_2,v_3,v_4\})$.
Let $S$ be a maximal stable set in $R$; then every vertex of 
$V(R) \setminus S$ has a neighbor in $S$.  Since $G$ is $P_4+kP_1$-free, it 
follows that   $|S| \leq k-1$, and so $\{v_1, v_2,v_3,v_4\} \cup S$ is a 
dominating set of size at most $k+3$ in $G$. Now Lemma~\ref{lem:P4+kP1} follows from Lemma~\ref{P_4} and Lemma~\ref{lem:size-2-P4}.
\end{proof}

\section{Necessity}\label{sec:necessity}

The aim of this section is to prove the following two statements.

\begin{lemma}\label{lem:nec-col}
There are infinitely many $H$-free 4-vertex-critical graphs if 
$H$ is a claw, a cycle, or $2P_2+P_1$.
\end{lemma}

Here, a \emph{claw} is the graph consisting of a central vertex plus three pairwise non-adjacent pendant vertices attached to it.
In the list-case, the following variant of this statement holds.

\begin{lemma}\label{lem:nec-list}
There are infinitely many $H$-free minimal list-obstructions if 
$H$ is a claw, a cycle, $2P_2+P_1$, or $2P_3$.
\end{lemma}

We remark that Lemma~\ref{lem:nec-col} implies the following.
Whenever $H$ is a graph containing a claw, a cycle, or $2P_2+P_1$ as an induced subgraph, there are infinitely many $H$-free 4-vertex-critical graphs.
A similar statement is true with respect to Lemma~\ref{lem:nec-list} and minimal list-obstructions.

\subsection{Proof of Lemma~\ref{lem:nec-col}}

Recall that there are infinitely many 4-vertex-critical claw-free graphs.
For example, this follows from the existence of 4-regular bipartite graphs of arbitrarily large girth (cf.~\cite{LU95} for an explicit construction of these) whose line graphs are necessarily 4-chromatic.
Moreover, there are 4-chromatic graphs of arbitrarily large girth, which follows from a classical result of Erd\H{o}s~\cite{Erd59}.
This, in turn, implies that there exist 4-vertex-critical graphs of arbitrary large girth.
Putting these two remarks together, we see that if $H$ is the claw or a cycle, then there are infinitely many 4-vertex-critical graphs. 

We now recall a construction due to Pokrovskiy~\cite{Pok14} which gives an infinite family of $4$-vertex-critical $P_7$-free graphs.
It is presented in more detail in our earlier work~\cite{CGSZ15}.

For each $r \ge 1$, let $G_r$ be the graph defined on the vertex set $v_0,\ldots,v_{3r}$ with edges as follows.
For all $i \in \{0,1,\ldots,3r\}$ and $j \in \{0,1,\ldots,r-1\}$, the vertex $v_i$ is adjacent to $v_{i-1}$, $v_{i+1}$, and $v_{i+3j+2}$. Here, we consider the indices to be taken modulo $3r+1$.
The graph $G_5$ is shown in Figure~\ref{fig:G16}. 

Up to permuting the colors, there is exactly one 3-coloring of $G_r\setminus v_0$.
Indeed, we may assume that $v_i$ receives color $i$, for $i=1,2,3$, since $\{v_1,v_2,v_3\}$ forms a triangle in $G_r$.
Similarly, $v_4$ receives color 1, $v_5$ receives color 2 and so on.
Finally, $v_{3r}$ receives color 3.
It follows that $G_r$ is not 3-colorable, since $v_0$ is adjacent to all of $v_1,v_2,v_{3r}$.

As the choice of $v_0$ was arbitrary, we know that $G_r$ is 4-vertex-critical.
The graph $G_r$ is $2P_2+P_1$-free which can be seen as follows.

\begin{claim}
For all $r$ the graph $G_r$ is $2P_2+P_1$-free.
\end{claim}
\begin{proof}
Suppose there is some $r$ such that $G_r$ is not $2P_2+P_1$-free.
Let $v_{i_1},\ldots,v_{i_5}$ be such that $G_r[\{v_{i_1},\ldots,v_{i_5}\}]$ is a $2P_2+P_1$.
Since $G_r$ is vertex-transitive, we may assume that $i_1=1$ and $N(v_{i_1}) \cap \{v_{i_2},\ldots,v_{i_5}\}=\emptyset$.
In particular, $i_2, \ldots, i_5 \neq 0$.

Consider $G_r\setminus v_0$ to be colored by the coloring $c$ proposed above, where each $v_i$ receives the color $i\mod 3$.
Due to the definition of $G_r$, $v_{i_1}$ is adjacent to every vertex of color 3, and thus $c(v_{j})\neq 3$ for all $j \in \{i_2,\ldots,i_5\}$.

We may assume that $c(v_{i_2})=c(v_{i_4})=1$, $c(v_{i_3})=c(v_{i_5})=2$, and both $v_{i_2}v_{i_3}$ and $v_{i_4}v_{i_5}$ are edges of $E(G_r)$.
For symmetry, we may further assume that $i_2 < i_4$.
Due to the definition of $G_r$, $v_{i_2}$ and $v_{i_4}$ are adjacent to every vertex of color 2 with a smaller index, and thus $i_4 < i_3$.
But now $i_2< i_4 < i_3$, a contradiction to the fact that $v_{i_2} v_{i_3} \in E(G_r)$.
This completes the proof.
\end{proof}

\begin{figure}
	\centering
		\includegraphics[width=0.40\textwidth]{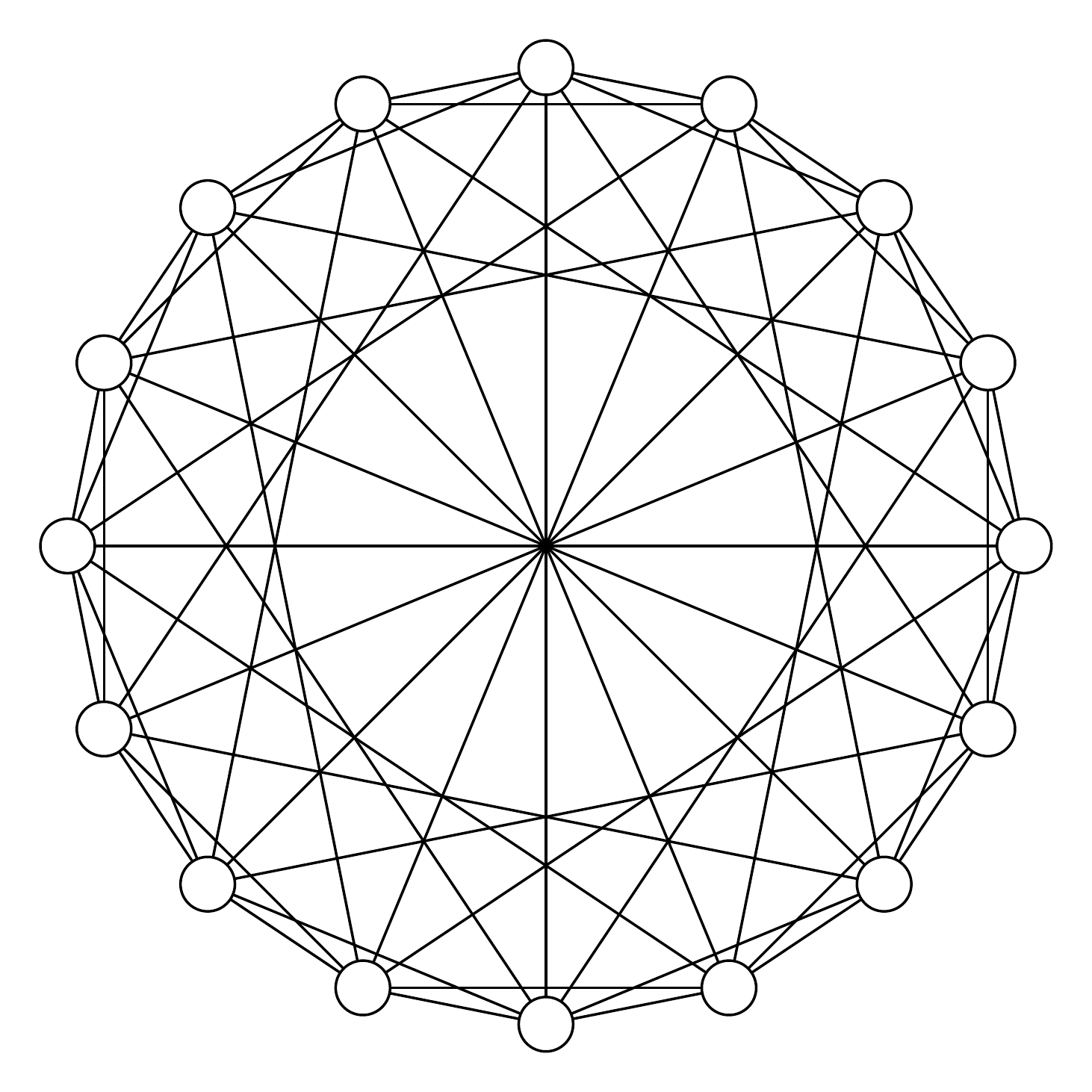}
	\label{fig:construction}
	\caption{A circular drawing of $G_{5}$}
	\label{fig:G16}
\end{figure}

Consequently, there are infinitely many $2P_2+P_1$-free 4-vertex-critical graphs, as desired.

\subsection{Proof of Lemma~\ref{lem:nec-list}}\label{sec:list-necessity}

In view of Lemma~\ref{lem:nec-col}, it remains to prove that there are infinitely many $2P_3$-free minimal list-obstructions.

For all $r \in \mathbb N$, let $H_r$ be the graph defined as follows.
The vertex set of $H_r$ is $V(H_r)=\{v_i:1\le i \le 3r-1\}$.
There is an edge from $v_1$ to $v_2$, from $v_2$ to $v_3$ and so on.
Thus, $P:=v_1$-$v_2$-\ldots-$v_{3r-1}$ is a path.
Moreover, there is an edge between a vertex $v_i$ and a vertex $v_j$ if $i \le j-2$, $i \equiv 2\mod 3$, and $j \equiv 1\mod 3$.
There are no further edges.
The graph $H_5$ is shown in Figure~\ref{fig:H5}.

The list system $L$ is defined by $L(v_1)=L(v_{3r-1})=\{1\}$ and, assuming $2 \le i \le 3r-2$,
\begin{equation*}
L(v_i)=
\begin{cases}
\{2,3\}, & \mbox{ if } i \equiv 0 \mod 3\\
\{1,3\}, & \mbox{ if } i \equiv 1 \mod 3\\
\{1,2\}, & \mbox{ if } i \equiv 2 \mod 3
\end{cases}.
\end{equation*}
Next we show that the above construction has the desired properties.

\begin{claim}
The pair $(H_r,L)$ is a minimal $2P_3$-free list-obstruction for all $r$.
\end{claim}
\begin{proof}
Let us first show that, for any $r$, $H_r$ is not colorable.
Consider the partial coloring $c$ that assigns color $1$ to $v_1$.
Since $L(v_2)=\{1,2\}$, the coloring can be updated from $v_1$ to $v_2$ by putting $c(v_2)=2$.
Now we can update the coloring from $v_2$ to $v_3$ by putting $c(v_3)=3$.
Like this we update the coloring along $P$ until $v_{3r-2}$ is colored.
However, we have to put $c(v_{3r-2})=1$, in contradiction to the fact that $L(v_{3r-1})=\{1\}$.
Thus, $H_r$ is not colorable.

Next we verify that $(H_r,L)$ is a minimal list-obstruction.
If we delete $v_1$ or $v_{3r-1}$, the graph becomes colorable.
So let us delete a vertex $v_i$ with $2 \le i \le 3r-2$.
We can color $(H_r\setminus v_i,L)$ as follows.
Give color 1 to $v_1$ and update along $P$ up to $v_{i-1}$.
Moreover, give color 1 to $v_{3r-1}$ and update along $P$ backwards up to $v_{i+1}$.
Call this coloring $c$.

To check that $c$ is indeed a coloring, we may focus on the non-path edges for obvious reasons.
Pick an edge between a vertex $v_j$ and a vertex $v_k$ with $j \le k-2$, if any.
By definition, $j \equiv 2\mod 3$ and $k \equiv 1\mod 3$.
If $j < i < k$, $c(v_j)=2$ and $c(v_k)=3$.
Moreover, if $j<k<i$, $c(v_j)=2$ and $c(v_k)=1$.
Finally, if $i<j<k$, $c(v_j)=1$ and $c(v_k)=3$.
So, $c$ is indeed a coloring of $H_r\setminus v_i$ and it remains to prove that $H_r$ is $2P_3$-free.

Suppose this is false, and let $r$ be minimum such that $H_r$ contains an induced $2P_3$.
Let $F$ be a copy of such a $2P_3$ in $H_r$.
It is clear that $r\ge 2$.
Note that $H_r \setminus N(v_2)$ is the disjoint union of complete graphs of order 1 and 2, and so $v_2 \notin V(F)$.
Since $N(v_1)=\{v_2\}$, we know that $v_1 \notin V(F)$. 
Moreover, as $F\setminus (N(v_5) \cup \{v_1,v_2\})$ is the disjoint union of complete graphs of order 1 and 2, we deduce that $v_5 \notin V(F)$.
But $F':=F\setminus \{v_1,v_2,v_3\}$ is isomorphic to $H_{r-1}$, and thus the choice of $r$ implies that $F'$ is $2P_3$-free.
Consequently, $v_3 \in V(F)$.
Since $N(v_3)=\{v_2,v_4\}$ and $v_2 \notin V(F)$, we know that $v_4 \in V(F)$. 
Finally, the fact that $N(v_4)=\{v_2,v_3,v_5\}$ implies that 
$v_3$ and $v_4$ both have degree one in $F$, and they are adjacent, a
contradiction. 
\end{proof}

\begin{figure}
	\centering
		\includegraphics[width=0.70\textwidth]{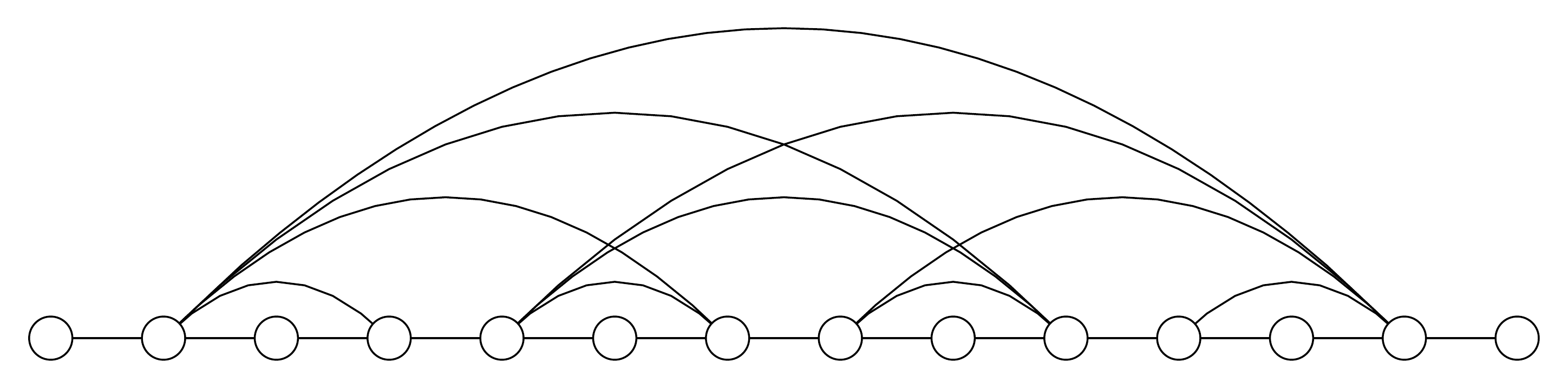}
	\caption{A drawing of $H_{5}$. The vertices $v_1$ to $v_{14}$ are shown from left to right.}
	\label{fig:H5}
\end{figure}

\section{Proof of Theorem~\ref{thm:coloring-characterization} and Theorem~\ref{thm:list-characterization}}\label{sec:together}

We now prove our main results. We start with a lemma.

\begin{lemma}
\label{graphs}
For every graph $H$, one of the following holds.
\begin{enumerate}
\item $H$ contains a cycle, a claw or $2P_2+P_1$.
\item $H=2P_3$.
\item $H$ is contained in $P_6$.
\item There exists $k>1$ such that $H$ is contained in $P_4+kP_1$.
\end{enumerate}
\end{lemma}

\begin{proof}
We may assume that $H$ does not contain $2P_2+P_1$, a cycle, or a claw.
It follows that every component of $H$ induces a path.
Let $H_1, H_2,\ldots H_k$ be the components of $H$, ordered so that $|H_1|\ge|H_2|\ge\ldots \ge |H_k|$.

If $|H_2| \geq 2$, then, since $H$ is $2P_2+P_1$-free, it follows that $k=2$,
$|H_1| \leq 3$, and $|H_2| \leq 3$, and so either $H$ is contained in $P_6$
or $H=2P_3$. This proves that  $|H_2|= \ldots = |H_k|=1$.

If $|H_1|\ge 5$, then since $H$ is $2P_2+P_1$-free, it follows that $k=1$, and 
$H$ is contained in  $P_6$. This proves that $|H_1| \leq 4$, and so 
$H$ is contained in $P_4+(k-1)P_1$.
This proves~Lemma~\ref{graphs}.
\end{proof}

Next  we prove Theorem~\ref{thm:coloring-characterization}, which we restate:

%\begin{theorem}\label{thm:coloring-characterization}
%Let $H$ be a graph. There are only finitely many $H$-free 4-vertex critical graphs if and only if $H$ is a subgraph of $P_6$, of $2P_3$, or of $P_4 + kP1$ for some $k \in \mathbb N$.
%\end{theorem}
%
%\begin{theorem}\label{thm:list-characterization}
%Let $H$ be a graph. There are only finitely many $H$-free minimal list-obstructions if and only if $H$ is a subgraph of $P_6$ or of $P_4 + kP1$ for some $k \in \mathbb N$.
%\end{theorem}

%\begin{lemma}\label{lem:nec-col}
%There are infinitely many $H$-free 4-vertex-critical graphs if 
%$H$ is a claw, a cycle, $2P_2+P_1$, or $P_4+P_2$.
%\end{lemma}
%
%In the list-case, the following stronger statement holds.
%
%\begin{lemma}\label{lem:nec-list}
%There are infinitely many $H$-free minimal list-obstructions if 
%$H$ is a claw, a cycle, $2P_2+P_1$, $P_4+P_2$, or $2P_3$.
%\end{lemma}

\secondthm*

\begin{proof}
If $H$ contains a cycle, a claw or $2P_2+P_1$, then  there is an infinite 
list of 4-vertex-critical graphs by Lemma~\ref{lem:nec-col}. By 
Lemma~\ref{graphs}, 
$H=2P_3$, $H$ is contained in $P_6$, or for some $k>1$, $H$ is contained in
$P_4+kP_1$, and Lemmas~\ref{lem:2P3-reduction}, \ref{lem:3downto2}
and \ref{lem:P4+kP1}, respectively,  imply that there are  there are only finitely many $H$-free $4$-vertex-critical graphs.  
\end{proof}

Finally, we prove  the list version of the result, Theorem~\ref{thm:list-characterization}, which we restate:

\thirdthm*

\begin{proof}
If $H$ contains a cycle, a claw, $2P_2+P_1$ or $2P_3$, then  there is an infinite list of obstructions by Lemma~\ref{lem:nec-list}. By Lemma~\ref{graphs}, 
$H$ is contained in $P_6$, or for some $k>1$, $H$ is contained in
$P_4+kP_1$.  Now Lemmas~\ref{lem:3downto2} and \ref{lem:P4+kP1}, respectively,
imply that there are  there are only finitely many $H$-free list-obstructions.
\end{proof}

\subsection*{Acknowledgments}

We thank Alexey Pokrovskiy for suggesting the construction of the graph $G_r$, 
and Fr\'ed\'eric Maffray for many useful discussions and suggestions.
%suggesting the first problem mentioned in Section~\ref{sec:further-research}.
Several of the computations for this work were carried out using the Stevin Supercomputer Infrastructure at Ghent University.
This material is based upon work supported in part by the U. S. Army  Research 
Laboratory and the U. S. Army Research Office under    grant number 
W911NF-16-1-0404.

\bibliographystyle{amsplain}
\bibliography{bnm-j,references,bnm}

\end{document}